\documentclass[9pt]{amsart}
\usepackage{amsmath}
\usepackage{amsxtra}
\usepackage{amscd}
\usepackage{amsthm}
\usepackage{amsfonts}
\usepackage{amssymb}
\usepackage{eucal}
\usepackage[all]{xy}
\usepackage{graphicx}
\usepackage[usenames]{color}
\usepackage[hidelinks]{hyperref}
\usepackage{tikz-cd}

\newtheorem{cor}[subsubsection]{Corollary}

\newtheorem{lem}[subsubsection]{Lemma}
\newtheorem{prop}[subsubsection]{Proposition}

\newtheorem{conj}[subsubsection]{Conjecture}

\newtheorem{thm}[subsubsection]{Theorem}

\newtheorem{quest}[subsubsection]{Question}

\theoremstyle{remark}
\newtheorem{rem}[subsubsection]{Remark}

\theoremstyle{definition}

\theoremstyle{remark}

\newcommand{\thmref}[1]{Theorem~\ref{#1}}

\newcommand{\secref}[1]{Sect.~\ref{#1}}
\newcommand{\lemref}[1]{Lemma~\ref{#1}}
\newcommand{\propref}[1]{Proposition~\ref{#1}}
\newcommand{\corref}[1]{Corollary~\ref{#1}}
\newcommand{\conjref}[1]{Conjecture~\ref{#1}}

\numberwithin{equation}{section}

\newcommand{\nc}{\newcommand}
\nc{\renc}{\renewcommand}
\nc{\ssec}{\subsection}
\nc{\sssec}{\subsubsection}
\nc{\on}{\operatorname}

\nc{\ips}{{\iota_P^{(S)}}}
\nc{\ipms}{{\iota_{P^-}^{(S)}}}
\nc{\sfpps}{{\sfp_P^{(S)}}}
\nc{\sfppms}{{\sfp_{P^-}^{(S)}}}

\nc\ol{\overline}
\nc\ul{\underline}
\nc\wt{\widetilde}
\nc\tboxtimes{\wt{\boxtimes}}
\nc\tstar{\wt{\star}}
\nc{\alp}{\alpha}

\nc{\ZZ}{{\mathbb Z}}
\nc{\NN}{{\mathbb N}}
\nc{\OO}{{\mathbb O}}
\renc{\SS}{{\mathbb S}}
\nc{\DD}{{\mathbb D}}
\nc{\GG}{{\mathbb G}}

\nc{\Fq}{{\mathbb F}_q}
\nc{\Fqb}{\ol{\mathbb F}_q}
\nc{\Ql}{{\mathbb Q}_\ell}
\nc{\Qlb}{{\ol{\mathbb Q}_\ell}}
\nc{\id}{\text{id}}
\nc\X{\mathcal X}

\nc{\red}{\on{red}}
\nc{\Ho}{\on{Ho}}
\nc{\Hom}{\on{Hom}}
\nc{\coHom}{\ul{\on{coHom}}}
\nc{\coMaps}{{\bf{coMaps}}}
\nc{\coef}{\on{coef}}
\nc{\Lie}{\on{Lie}}
\nc{\Loc}{\on{Loc}}
\nc{\coLoc}{\on{coLoc}}
\nc{\Pic}{\on{Pic}}
\nc{\Bun}{\on{Bun}}
\nc{\IC}{\on{IC}}
\nc{\Aut}{\on{Aut}}
\nc{\rk}{\on{rk}}
\nc{\Sh}{\on{Sh}}
\nc{\Perv}{\on{Perv}}
\nc{\pos}{{\on{pos}}}
\nc{\Conv}{\on{Conv}}
\nc{\Sph}{\on{Sph}}
\nc{\Sym}{\on{Sym}}
\nc{\BunBb}{\overline{\Bun}_B}
\nc{\BunNb}{\overline{\Bun}_N}
\nc{\BunTb}{\overline{\Bun}_T}
\nc{\BunBbm}{\overline{\Bun}_{B^-}}
\nc{\BunBbel}{\overline{\Bun}_{B,el}}
\nc{\BunBbmel}{\overline{\Bun}_{B^-,el}}
\nc{\Buno}{\overset{o}{\Bun}}
\nc{\BunPb}{{\overline{\Bun}_P}}
\nc{\BunBM}{\Bun_{B(M)}}
\nc{\BunBMb}{\overline{\Bun}_{B(M)}}
\nc{\BunPbw}{{\widetilde{\Bun}_P}}
\nc{\BunBP}{\widetilde{\Bun}_{B,P}}
\nc{\GUb}{\overline{G/U}}
\nc{\GUPb}{\overline{G/U(P)}}

\nc{\Hhom}{\underline{\on{Hom}}}
\nc\syminfty{\on{Sym}^{\infty}}
\nc\lal{\ol{\lambda}}
\nc\xl{\ol{x}}
\nc\thl{\ol{\theta}}
\nc\nul{\ol{\nu}}
\nc\mul{\ol{\mu}}
\nc\Sum\Sigma
\nc{\oX}{\overset{o}{X}{}}
\nc{\hl}{\overset{\leftarrow}h{}}
\nc{\hr}{\overset{\rightarrow}h{}}
\nc{\M}{{\mathcal M}}
\nc{\N}{{\mathcal N}}
\nc{\F}{{\mathcal F}}
\nc{\D}{{\mathcal D}}
\nc{\Q}{{\mathcal Q}}
\nc{\Y}{{\mathcal Y}}
\nc{\G}{{\mathcal G}}
\nc{\E}{{\mathcal E}}
\nc{\CalC}{{\mathcal C}}
\nc\Dh{\widehat{\D}}

\nc{\C}{{\mathcal C}}
\nc{\K}{{\mathcal K}}
\renewcommand{\H}{{\mathcal H}}

\nc{\T}{{\mathcal T}}
\nc{\V}{{\mathcal V}}
\renc{\P}{{\mathcal P}}
\nc{\A}{{\mathcal A}}
\nc{\B}{{\mathcal B}}
\nc{\U}{{\mathcal U}}

\nc{\Gr}{{\on{Gr}}}

\nc{\frn}{{\check{\mathfrak u}(P)}}

\nc{\fC}{\mathfrak C}
\nc{\fT}{\mathfrak T}
\nc{\p}{\mathfrak p}
\nc{\q}{\mathfrak q}
\nc\f{{\mathfrak f}}

\nc{\qo}{{\mathfrak q}}
\nc{\po}{{\mathfrak p}}
\nc{\s}{{\mathfrak s}}
\nc\w{\text{w}}

\nc\mathi\iota
\nc\Spec{\on{Spec}}
\nc\Proj{\on{Proj}}
\nc\Mod{\on{Mod}}
\nc{\tw}{\widetilde{\mathfrak t}}
\nc{\pw}{\widetilde{\mathfrak p}}
\nc{\qw}{\widetilde{\mathfrak q}}
\nc{\jw}{\widetilde j}

\nc{\grb}{\overline{\Gr}}
\nc{\I}{\mathcal I}

\nc{\lambdach}{{\check\lambda}}
\nc{\Lambdach}{{\check\Lambda}{}}
\nc{\much}{{\check\mu}}
\nc{\omegach}{{\check\omega}}
\nc{\nuch}{{\check\nu}}
\nc{\etach}{{\check\eta}}
\nc{\alphach}{{\check\alpha}}
\nc{\oblvtach}{{\check\oblvta}}
\nc{\rhoch}{{\check\rho}}
\nc{\ch}{{\check h}}

\nc{\Hb}{\overline{\H}}

\nc{\BA}{{\mathbb{A}}}
\nc{\BC}{{\mathbb{C}}}
\nc{\BE}{{\mathbb{E}}}
\nc{\BF}{{\mathbb{F}}}
\nc{\BG}{{\mathbb{G}}}
\nc{\BL}{{\mathbb{L}}}
\nc{\BM}{{\mathbb{M}}}
\nc{\BO}{{\mathbb{O}}}
\nc{\BD}{{\mathbb{D}}}
\nc{\BN}{{\mathbb{N}}}
\nc{\BP}{{\mathbb{P}}}
\nc{\BQ}{{\mathbb{Q}}}
\nc{\BR}{{\mathbb{R}}}
\nc{\BV}{{\mathbb{V}}}
\nc{\BZ}{{\mathbb{Z}}}
\nc{\BS}{{\mathbb{S}}}
\nc{\Deep}{{\bf{deep}}}
\nc{\deep}{deep}

\nc{\CA}{{\mathcal{A}}}
\nc{\CB}{{\mathcal{B}}}

\nc{\CE}{{\mathcal{E}}}
\nc{\CF}{{\mathcal{F}}}
\nc{\CH}{{\mathcal{H}}}

\nc{\CL}{{\mathcal{L}}}
\nc{\CC}{{\mathcal{C}}}
\nc{\CG}{{\mathcal{G}}}
\nc{\CalD}{{\mathcal{D}}}
\nc{\CM}{{\mathcal{M}}}
\nc{\CN}{{\mathcal{N}}}
\nc{\CK}{{\mathcal{K}}}
\nc{\CO}{{\mathcal{O}}}
\nc{\CP}{{\mathcal{P}}}
\nc{\CQ}{{\mathcal{Q}}}
\nc{\CR}{{\mathcal{R}}}
\nc{\CS}{{\mathcal{S}}}
\nc{\CT}{{\mathcal{T}}}
\nc{\CU}{{\mathcal{U}}}
\nc{\CV}{{\mathcal{V}}}
\nc{\CW}{{\mathcal{W}}}
\nc{\CX}{{\mathcal{X}}}
\nc{\CY}{{\mathcal{Y}}}
\nc{\CZ}{{\mathcal{Z}}}
\nc{\CI}{{\mathcal{I}}}

\nc{\csM}{{\check{\mathcal A}}{}}
\nc{\oM}{{\overset{\circ}{\mathcal M}}{}}
\nc{\obM}{{\overset{\circ}{\mathbf M}}{}}
\nc{\oCA}{{\overset{\circ}{\mathcal A}}{}}
\nc{\obA}{{\overset{\circ}{\mathbf A}}{}}
\nc{\ooM}{{\overset{\circ}{M}}{}}
\nc{\osM}{{\overset{\circ}{\mathsf M}}{}}
\nc{\vM}{{\overset{\bullet}{\mathcal M}}{}}
\nc{\nM}{{\underset{\bullet}{\mathcal M}}{}}
\nc{\oD}{{\overset{\circ}{\mathcal D}}{}}
\nc{\obD}{{\overset{\circ}{\mathbf D}}{}}
\nc{\oA}{{\overset{\circ}{A}}{}}
\nc{\op}{{\overset{\bullet}{\mathbf p}}{}}
\nc{\cp}{{\overset{\circ}{\mathbf p}}{}}
\nc{\oU}{{\overset{\bullet}{\mathcal U}}{}}
\nc{\oZ}{{\overset{\circ}{\mathcal Z}}{}}
\nc{\ofZ}{{\overset{\circ}{\mathfrak Z}}{}}
\nc{\oF}{{\overset{\circ}{\fF}}}

\nc{\fa}{{\mathfrak{a}}}
\nc{\ofa}{\overset{\circ}{\mathfrak{a}}}
\nc{\fb}{{\mathfrak{b}}}
\nc{\fd}{{\mathfrak{d}}}
\nc{\ff}{{\mathfrak{f}}}
\nc{\fg}{{\mathfrak{g}}}
\nc{\fgl}{{\mathfrak{gl}}}
\nc{\fh}{{\mathfrak{h}}}
\nc{\fj}{{\mathfrak{j}}}
\nc{\fl}{{\mathfrak{l}}}
\nc{\fm}{{\mathfrak{m}}}
\nc{\ofm}{\overset{\circ}{\mathfrak{m}}}
\nc{\fn}{{\mathfrak{n}}}
\nc{\fu}{{\mathfrak{u}}}
\nc{\fp}{{\mathfrak{p}}}
\nc{\fr}{{\mathfrak{r}}}
\nc{\fs}{{\mathfrak{s}}}
\nc{\ft}{{\mathfrak{t}}}
\nc{\oft}{\overset{\circ}{\mathfrak{t}}}
\nc{\fz}{{\mathfrak{z}}}
\nc{\fsl}{{\mathfrak{sl}}}
\nc{\hsl}{{\widehat{\mathfrak{sl}}}}
\nc{\hgl}{{\widehat{\mathfrak{gl}}}}
\nc{\hg}{{\widehat{\mathfrak{g}}}}
\nc{\hm}{{\widehat{\mathfrak{m}}}}
\nc{\chg}{{\widehat{\mathfrak{g}}}{}^\vee}
\nc{\hn}{{\widehat{\mathfrak{n}}}}
\nc{\chn}{{\widehat{\mathfrak{n}}}{}^\vee}

\nc{\fA}{{\mathfrak{A}}}
\nc{\fB}{{\mathfrak{B}}}
\nc{\fD}{{\mathfrak{D}}}
\nc{\fE}{{\mathfrak{E}}}
\nc{\fF}{{\mathfrak{F}}}
\nc{\fG}{{\mathfrak{G}}}
\nc{\fK}{{\mathfrak{K}}}
\nc{\fL}{{\mathfrak{L}}}
\nc{\fM}{{\mathfrak{M}}}
\nc{\fN}{{\mathfrak{N}}}
\nc{\fP}{{\mathfrak{P}}}
\nc{\fU}{{\mathfrak{U}}}
\nc{\fV}{{\mathfrak{V}}}
\nc{\fZ}{{\mathfrak{Z}}}

\nc{\ba}{{\mathbf{a}}}
\nc{\bb}{{\mathbf{b}}}
\nc{\bc}{{\mathbf{c}}}
\nc{\bd}{{\mathbf{d}}}
\nc{\bbf}{{\mathbf{f}}}
\nc{\be}{{\mathbf{e}}}
\nc{\bi}{{\mathbf{i}}}
\nc{\bj}{{\mathbf{j}}}
\nc{\bh}{{\mathbf{h}}}
\nc{\bm}{{\mathbf{m}}}
\nc{\bn}{{\mathbf{n}}}
\nc{\bo}{{\mathbf{o}}}
\nc{\bp}{{\mathbf{p}}}
\nc{\bq}{{\mathbf{q}}}
\nc{\bu}{{\mathbf{u}}}
\nc{\bv}{{\mathbf{v}}}
\nc{\bx}{{\mathbf{x}}}
\nc{\bs}{{\mathbf{s}}}
\nc{\by}{{\mathbf{y}}}
\nc{\bw}{{\mathbf{w}}}
\nc{\bA}{{\mathbf{A}}}
\nc{\bK}{{\mathbf{K}}}
\nc{\bB}{{\mathbf{B}}}
\nc{\bC}{{\mathbf{C}}}
\nc{\bG}{{\mathbf{G}}}
\nc{\bD}{{\mathbf{D}}}
\nc{\bE}{{\mathbf{E}}}
\nc{\bH}{{{\mathbf{H}}}}
\nc{\bL}{{\mathbf{L}}}
\nc{\bM}{{\mathbf{M}}}
\nc{\bN}{{\mathbf{N}}}
\nc{\bO}{{\mathbf{O}}}
\nc{\bQ}{{\mathbf{Q}}}
\nc{\bV}{{\mathbf{V}}}
\nc{\bW}{{\mathbf{W}}}
\nc{\bX}{{\mathbf{X}}}
\nc{\bZ}{{\mathbf{Z}}}
\nc{\bS}{{\mathbf{S}}}

\nc{\sA}{{\mathsf{A}}}
\nc{\sB}{{\mathsf{B}}}
\nc{\sC}{{\mathsf{C}}}
\nc{\sD}{{\mathsf{D}}}
\nc{\sF}{{\mathsf{F}}}
\nc{\sG}{{\mathsf{G}}}
\nc{\sH}{{\mathsf{H}}}
\nc{\sK}{{\mathsf{K}}}
\nc{\sM}{{\mathsf{M}}}
\nc{\sN}{{\mathsf{N}}}
\nc{\sO}{{\mathsf{O}}}
\nc{\sW}{{\mathsf{W}}}
\nc{\sQ}{{\mathsf{Q}}}
\nc{\sP}{{\mathsf{P}}}
\nc{\sR}{{\mathsf{R}}}
\nc{\sT}{{\mathsf{T}}}
\nc{\sZ}{{\mathsf{Z}}}
\nc{\sfi}{{\mathsf{i}}}
\nc{\sfj}{{\mathsf{j}}}
\nc{\sfp}{{\mathsf{p}}}
\nc{\sfq}{{\mathsf{q}}}
\nc{\sfs}{{\mathsf{s}}}
\nc{\sft}{{\mathsf{t}}}
\nc{\sr}{{\mathsf{r}}}
\nc{\bk}{{\mathsf{k}}}
\nc{\sa}{{\mathsf{s}}}
\nc{\sg}{{\mathsf{g}}}
\nc{\sn}{{\mathsf{n}}}
\nc{\sh}{{\mathsf{h}}}
\nc{\sff}{{\mathsf{f}}}
\nc{\sfb}{{\mathsf{b}}}
\nc{\sfc}{{\mathsf{c}}}
\nc{\sfe}{{\mathsf{e}}}
\nc{\sd}{{\mathsf{d}}}

\nc{\BK}{{\bar{K}}}

\nc{\tA}{{\widetilde{\mathbf{A}}}}
\nc{\tB}{{\widetilde{\mathcal{B}}}}
\nc{\tg}{{\widetilde{\mathfrak{g}}}}
\nc{\tG}{{\widetilde{G}}}
\nc{\TM}{{\widetilde{\mathbb{M}}}{}}
\nc{\tO}{{\widetilde{\mathsf{O}}}{}}
\nc{\tU}{{\widetilde{\mathfrak{U}}}{}}
\nc{\TZ}{{\tilde{Z}}}
\nc{\tx}{{\tilde{x}}}
\nc{\tbv}{{\tilde{\bv}}}
\nc{\tfP}{{\widetilde{\mathfrak{P}}}{}}
\nc{\tz}{{\tilde{\zeta}}}
\nc{\tmu}{{\tilde{\mu}}}

\nc{\urho}{\underline{\rho}}
\nc{\uB}{\underline{B}}
\nc{\uC}{{\underline{\mathbb{C}}}}
\nc{\ui}{\underline{i}}
\nc{\uj}{\underline{j}}
\nc{\ofP}{{\overline{\mathfrak{P}}}}
\nc{\oB}{{\overline{\mathcal{B}}}}
\nc{\og}{{\overline{\mathfrak{g}}}}
\nc{\oI}{{\overline{I}}}

\nc{\eps}{\varepsilon}
\nc{\hrho}{{\hat{\rho}}}

\nc{\one}{{\mathbf{1}}}
\nc{\two}{{\mathbf{t}}}

\nc{\Rep}{{\mathop{\operatorname{\rm Rep}}}}
\nc{\Tot}{{\mathop{\operatorname{\rm Tot}}}}
\nc{\Ker}{{\mathop{\operatorname{\rm Ker}}}}
\nc{\im}{{\mathop{\operatorname{\rm Im}}}}
\nc{\Hilb}{{\mathop{\operatorname{\rm Hilb}}}}
\nc{\End}{{\mathop{\operatorname{\rm End}}}}
\nc{\Ext}{{\mathop{\operatorname{\rm Ext}}}}
\nc{\CHom}{{\mathop{\operatorname{{\mathcal{H}}\it om}}}}
\nc{\CEnd}{{\mathop{\operatorname{{\mathcal{E}}\it nd}}}}
\nc{\GL}{{\mathop{\operatorname{\rm GL}}}}
\nc{\gr}{{\mathop{\operatorname{\rm gr}}}}
\nc{\HN}{{\mathop{\operatorname{\rm HN}}}}
\nc{\Id}{{\mathop{\operatorname{\rm Id}}}}
\nc{\de}{{\mathop{\operatorname{\rm def}}}}
\nc{\length}{{\mathop{\operatorname{\rm length}}}}
\nc{\supp}{{\mathop{\operatorname{\rm supp}}}}

\nc{\Cliff}{{\mathsf{Cliff}}}
\nc{\Fl}{\on{Fl}}
\nc{\Fib}{{\mathsf{Fib}}}
\nc{\Coh}{{\on{Coh}}}
\nc{\QCoh}{{\on{QCoh}}}
\nc{\IndCoh}{{\on{IndCoh}}}
\nc{\FCoh}{{\mathsf{FCoh}}}

\nc{\reg}{{\text{\rm reg}}}

\nc{\cplus}{{\mathbf{C}_+}}
\nc{\cminus}{{\mathbf{C}_-}}
\nc{\cthree}{{\mathbf{C}_\bullet}}
\nc{\Qbar}{{\bar{Q}}}
\nc\Eis{\on{Eis}}
\nc\Eisb{\ol\Eis{}}
\nc\Eisr{\on{Eis}^{rat}{}}
\nc\wh{\widehat}
\nc{\Def}{\on{Def_{\check{\fb}}(E)}}
\nc{\barZ}{\overline{Z}{}}
\nc{\barbarZ}{\overline{\barZ}{}}
\nc{\barpi}{\overline\pi}
\nc{\barbarpi}{\overline\barpi}
\nc{\barpip}{\overline\pi{}^+}
\nc{\barpim}{\overline\pi{}^-}

\nc{\fq}{\mathfrak q}

\nc{\fqb}{\ol{\sfq}{}}
\nc{\fpb}{\ol{\sfp}{}}
\nc{\fpr}{{\mathsf{pair}^{rat}}{}}
\nc{\fqr}{{\sfq^{rat}}{}}

\nc{\hattimes}{\wh\otimes}

\nc{\bOmega}{{\overline{\Omega(\check \fn)}}}

\nc{\seq}[1]{\stackrel{#1}{\sim}}

\nc{\cT}{{\check{T}}}
\nc{\cG}{{\check{G}}}
\nc{\cM}{{\check{M}}}
\nc{\cB}{{\check{B}}}
\nc{\cP}{{\check{P}}}

\nc{\ct}{{\check{\mathfrak t}}}
\nc{\cg}{{\check{\fg}}}
\nc{\cb}{{\check{\fb}}}
\nc{\cn}{{\check{\fn}}}

\nc{\cLambda}{{\check\Lambda}}

\nc{\cla}{{\check\lambda}}
\nc{\cmu}{{\check\mu}}
\nc{\cnu}{{\check\nu}}
\nc{\ceta}{{\check\eta}}

\nc{\DefbE}{{\on{Def}_{\cB}(E_\cT)}}

\nc{\imathb}{{\ol{\imath}}}
\nc{\rlr}{\overset{\longrightarrow}{\underset{\longrightarrow}\longleftarrow}}

\nc{\oBun}{\overset{\circ}\Bun}
\nc{\LS}{\on{LS}}
\nc{\BunBbb}{\ol{\ol{Bun}}_B}
\nc{\BunBr}{\Bun_B^{rat}}
\nc{\BunBrsg}{\Bun_B^{rat,\on{s.g.}}}
\nc{\BunBrp}{\Bun_B^{rat,polar}}
\nc{\BunBrpbg}{\Bun_B^{rat,polar,\on{b.g.}}}
\nc{\BunBrpsg}{\Bun_B^{rat,polar,\on{s.g.}}}
\nc{\BunTrp}{\Bun_T^{rat,polar}}
\nc{\BunTrpbg}{\Bun_T^{rat,polar,\on{b.g.}}}
\nc{\BunTrpsg}{\Bun_T^{rat,polar,\on{s.g.}}}
\nc{\BunNr}{\Bun_N^{rat}}
\nc{\BunNre}{\Bun_N^{enh,rat}}
\nc{\BunTr}{\Bun_T^{rat}}
\nc{\Vect}{\on{Vect}}
\nc{\Whit}{\on{Whit}}
\nc{\CTb}{\ol{\on{CT}}}
\nc{\Ran}{{\on{Ran}}}
\nc{\Ranu}{{\on{Ran}^{\on{untl}}}}
\nc{\Ranustr}{{\on{Ran}^{\on{untl}}_{\on{str}}}}
\nc{\Ranusubset}{{\Ran^{\on{untl}}_{\subseteq}}}
\nc{\Ranusubsetx}{{\Ran^{\on{untl}}_{x\subseteq}}}
\nc{\CTr}{\on{CT}^{rat}{}}
\nc\jmathr{\jmath^{rat}{}}
\nc{\ux}{\underline{x}}
\nc{\clambda}{{\check\lambda}}
\nc{\calpha}{{\check\alpha}}
\nc{\ind}{{\mathbf{ind}}}
\nc{\coinv}{{\mathbf{coinv}}}
\nc{\oblv}{{\mathbf{oblv}}}
\nc{\free}{{\mathbf{free}}}
\nc{\ox}{{\overline{x}}}
\nc{\cLa}{\check{\Lambda}}
\nc{\StinftyCat}{\on{DGCat}}
\nc{\inftyCat}{\infty\on{-Cat}}
\nc{\inftygroup}{\infty\on{-Grpd}}
\nc{\Dmod}{\on{D-mod}}
\nc{\CMaps}{{\mathcal Maps}}
\nc{\Maps}{\on{Maps}}
\nc{\affSch}{\on{Sch}^{\on{aff}}}
\nc{\dr}{{\on{dR}}}
\nc{\oCF}{\overset{\circ}\CF}
\nc{\oCY}{\overset{\circ}\CY}
\nc{\oCZ}{\overset{\circ}\CZ}
\nc{\opi}{\overset{\circ}\pi}
\nc{\leqG}{\underset{G}\leq}
\nc{\leqM}{\underset{M}\leq}
\nc{\leqGad}{\underset{G_{ad}}\leq}
\nc{\leqMad}{\underset{M_{ad}}\leq}
\nc{\Tr}{\on{Tr}}
\nc{\Frob}{{\on{Frob}}}
\nc{\DGCat}{\on{DGCat}}
\nc{\tDGCat}{2\on{-DGCat}_{\on{u.g.}}}
\nc{\ev}{\on{ev}}
\nc{\mmod}{\on{-}\mathbf{mod}}
\nc{\sotimes}{\overset{!}\otimes}
\nc{\Shv}{\on{Shv}}
\nc{\Spc}{\on{Spc}}
\nc{\Res}{\on{Res}}
\nc{\bDelta}{{\mathbf{\Delta}}}
\nc{\bMaps}{{\mathbf{Maps}}}
\nc{\cD}{\mathcal D}
\nc{\ocD}{\cD^\times}
\nc{\ppart}{(\!(t)\!)}
\nc{\qqart}{[\![t]\!]}
\nc{\oCU}{\overset{\circ}{\CU}}
\nc{\Exc}{{\mathcal{E}xc}}
\nc{\Sht}{\on{Sht}}
\nc{\Nilp}{{\on{Nilp}}}
\nc{\Drinf}{\on{Drinf}}
\nc{\Sing}{\on{Sing}}
\nc{\IndLisse}{\Lisse}
\nc{\Shvl}{\on{Shv}_{\on{lisse}}} 
\nc{\Lisse}{\on{Lisse}}
\nc{\Mir}{\on{Mir}}
\nc{\fSet}{\on{fSet}}
\nc{\qLisse}{\on{QLisse}}
\nc{\Ev}{\on{Ev}}
\nc{\Sat}{\on{Sat}}
\nc{\Se}{\on{Se}}
\nc{\coSht}{\on{co-Sht}}
\nc{\coCK}{\on{co-}\!\CK}
\nc{\FLE}{\on{FLE}}
\nc{\BRST}{\on{BRST}}
\nc{\KL}{\on{KL}}
\nc{\crit}{{\on{crit}}}
\nc{\Op}{{\on{Op}}}
\nc{\MOp}{\on{MOp}}
\nc{\Wak}{\on{Wak}}
\nc{\Av}{\on{Av}}
\nc{\semiinf}{{\frac{\infty}{2}}}
\nc{\DS}{\on{DS}}
\nc{\dR}{{\on{dR}}}
\nc{\Poinc}{{\on{Poinc}}}
\renc{\det}{\on{det}}
\nc{\oG}{\overset{\circ}{G}}
\nc{\Sectna}{\on{Sect}_\nabla}

\begin{document}

\dedicatory{To Sasha Beilinson}


\vskip1cm

\title[Proof of the the geometric Langlands conjecture I]{Proof of the geometric Langlands conjecture I: \\ construction of the functor}

\author{Dennis Gaitsgory and Sam Raskin}

\begin{abstract}
In this paper we construct the geometric Langlands functor in one direction (from the automorphic to the spectral side) in characteristic zero
settings (i.e., de Rham and Betti). 
We prove that various forms of the conjecture (de Rham vs Betti, restricted vs. non-restricted, tempered vs. non-tempered)
are equivalent. We also discuss structural properties of Hecke eigensheaves.
\end{abstract}

\date{\today}

\maketitle

\bigskip

\bigskip

\tableofcontents

\section*{Introduction}

This paper is the first in a series of five, in which the\footnote{Global and unramified.}
geometric Langlands conjecture will be proved. 

\medskip

The entire project is joint work of D.~Arinkin, D.~Beraldo, L.~Chen, J.~Faegerman, D.~Gaitsgory, K.~Lin, S.~Raskin and
N.~Rozenblyum. 

\medskip

The individual papers in the five-paper series will have different subsets of the above people as authors. 

\ssec{The Langlands functor in the de Rham setting}

We first consider the de Rham version of the geometric Langlands conjecture (GLC). 

\sssec{}

Let $k$ be a field of characteristic $0$. 
Let $X$ be a smooth and complete curve and $G$ a 
connected reductive group over $k$. Let $\cG$ be the Langlands dual of $G$, viewed also as a reductive group over $k$.

\medskip

The automorphic side of GLC is the category 
$$\Dmod_{\frac{1}{2}}(\Bun_G)$$
of half-twisted D-modules on the moduli stack $\Bun_G$ of principal $G$-bundles on $X$,
see \secref{ss:autom}.

\medskip

The spectral side is the category $$\IndCoh_\Nilp(\LS_\cG)$$ of ind-coherent sheaves with singular
support in the nilpotent cone (see \cite[Sect. 11.1.5]{AG}) on the moduli stack $\LS_\cG$ of
de Rham $\cG$-local systems on $X$.

\sssec{}

The Langlands functor, to be constructed in this paper, goes from the automorphic to the spectral side
\begin{equation} \label{e:L intro}
\BL_G:\Dmod_{\frac{1}{2}}(\Bun_G)\to \IndCoh_\Nilp(\LS_\cG). 
\end{equation} 

The geometric Langlands conjecture (GLC) in the de Rham setting says that the functor \eqref{e:L intro}
is an equivalence. 

\sssec{}

However, there is an important piece of structure that enters the game before we attempt to construct
$\BL_G$, namely, the \emph{spectral action}.

\medskip

The assertion is that the combined action of the \emph{Hecke functors} gives rise
to a uniquely defined monoidal action of $\QCoh(\LS_\cG)$ on $\Dmod_{\frac{1}{2}}(\Bun_G)$. 

\medskip

Such an action can be viewed as a ``spectral decomposition" of $\Dmod_{\frac{1}{2}}(\Bun_G)$
along the stack $\LS_\cG$. 

\medskip 

The existence of the spectral action was established by V.~Drinfeld and the first author (it is recorded 
in \cite[Corollary 4.5.5]{Ga2}), but its idea goes back to A.~Beilinson, in the form of what is called the
\emph{Beilinson projector} (see \cite[Sect. 15]{AGKRRV} for a detailed discussion).  

\sssec{}

The construction of $\BL_G$ proceeds in two steps: we first construct the \emph{coarse} version of $\BL_G$
$$\BL_{G,\on{coarse}}:\Dmod_{\frac{1}{2}}(\Bun_G)\to \QCoh(\LS_\cG).$$

It is related to the desired functor $\BL_G$ by
\begin{equation}\label{eq:psi-lg}
\Psi_{\Nilp,\{0\}}\circ \BL_G\simeq \BL_{G,\on{coarse}},
\end{equation}
where 
$$\Psi_{\Nilp,\{0\}}:\IndCoh_\Nilp(\LS_\cG)\to  \QCoh(\LS_\cG)$$
is the \emph{coarsening} functor. 

\sssec{}

The functor $\BL_{G,\on{coarse}}$ stems from the classical idea in the theory of automorphic functions
that says that Langlands correspondence is normalized by the Whittaker model. In the geometric context this 
idea is incarnated by the fact that the functor $\BL_{G,\on{coarse}}$ is uniquely determined by the following two
requirements:

\begin{itemize}

\item The functor $\BL_{G,\on{coarse}}$ is compatible with the Hecke action, i.e., it 
intertwines the action of $\QCoh(\LS_\cG)$ on the two sides;

\smallskip

\item It makes the triangle 
$$
\xy
(0,0)*+{\Dmod_{\frac{1}{2}}(\Bun_G)}="X";
(40,0)*+{\QCoh(\LS_\cG)}="Y";
(20,20)*+{\Vect}="Z";
{\ar@{->}^{\on{coeff}^{\on{Vac,glob}}} "X";"Z"};
{\ar@{->}_{\Gamma(\LS_\cG,-)} "Y";"Z"};
{\ar@{->}^{\BL_{G,\on{coarse}}} "X";"Y"}
\endxy
$$
commute, where $\on{coeff}^{\on{Vac,glob}}$ is the functor of \emph{first Whittaker coefficient}, see \secref{sss:coeff}. 

\end{itemize} 

\sssec{}

Once the functor $\BL_{G,\on{coarse}}$ is constructed, we lift it to the sought-for functor $\BL_G$ using
cohomological estimates provided by \thmref{t:compact to bdd below dr}.  

\medskip

Namely, \thmref{t:compact to bdd below dr} says that the functor $\BL_{G,\on{coarse}}$ sends compact objects of 
$\Dmod_{\frac{1}{2}}(\Bun_G)$ to \emph{eventually coconnective} objects in $\QCoh(\LS_\cG)$, i.e., objects that are
cohomologically bounded on the left.
As $\Psi_{\Nilp,\{0\}}$ is an equivalence on eventually coconnective subcategories,
we obtain the functor $\BL_G$, which is then uniquely characterized by
\eqref{eq:psi-lg} and the requirement that $\BL_G$ send compacts to eventually
coconnective objects.\footnote{We remark that the latter requirement is
necessarily satisfied by any putative geometric Langlands functor: 
such a functor ought to be an equivalence, so must preserve compacts;
and compact objects are eventually coconnective in $\IndCoh_{\Nilp}(\LS_\cG)$.}

\medskip

In fact, we show that the functor $\BL_{G,\on{coarse}}$ has a cohomological amplitude bounded on the left. (Note, however,
that this does \emph{not} mean that the Langlands functor $\BL_G$ has a cohomological amplitude bounded on the left; it rather has 
unbounded amplitude once $G$ is non-abelian.)

\medskip

Of less immediate practical importance, we also show that the functor $\BL_{G,\on{coarse}}$ has a 
cohomological amplitude bounded on the \emph{right}. This property is automatically inherited by $\BL_G$. 

\sssec{}

One can thus say that the upgrade 
\begin{equation} \label{e:upgrade}
\BL_{G,\on{coarse}}\rightsquigarrow \BL_G,
\end{equation}
carried out in the present paper, is cohomological in nature. Its main 
ingredient, \thmref{t:compact to bdd below dr}, 
uses non-trivial input: ultimately we deduce it from 
\thmref{t:L coarse left exact Nilp}, proved in \cite{FR} using methods developed
in \cite{AGKRRV} and \cite{Lin}.

\medskip

An alternative approach to constructing $\BL_G$ is being 
developed in separate work of D.~Beraldo, L.~Chen and K.~Lin. 
Their work is based on ``gluing" the category $\Dmod_{\frac{1}{2}}(\Bun_G)$
from the categories $\Dmod_{\frac{1}{2}}(\Bun_M)_{\on{temp}}$ (see \secref{sss:temp}) for Levi subgroups $M$ of $G$, using the
\emph{Eisenstein series} functors, see \cite{BeLi}.
Then they construct the functor using a parallel procedure 
on the spectral side, i.e., gluing 
$\IndCoh_\Nilp(\LS_\cG)$ from the categories $\QCoh(\LS_\cM)$, using the
\emph{spectral Eisenstein series} functors, developed in \cite{Be2}. Unlike the approach in the present
paper, their work pursues (and better develops) the original proposal 
from \cite{Ga2}\footnote{An additional merit of this approach is that, unlike using the definition of $\IndCoh(-)$ from scratch, i.e., as the ind-completion
of $\Coh(-)$, it uses the formal properties of the $\IndCoh$ category.}.

\ssec{The Betti setting}

We now consider GLC in the Betti setting. 

\sssec{}

In this case our algebraic geometry happens over the ground field $\BC$
(so $X$ and $G$ are over $\BC$), and we work with sheaves of $\sfe$-vector spaces, where $\sfe$ is an arbitrary
field of characteristic $0$. 

\medskip

So, $\cG$ is an reductive group over $\sfe$, and $\LS_\cG$ 
is an algebraic stack over $\sfe$.

\sssec{}

A historical and conceptual difference between the de Rham and Betti situations is that while it has ``always" 
(i.e., since the inception of the subject by A.~Beilinson and V.~Drinfeld) been 
understood what the automorphic side in the de Rham setting was (i.e., the entire category of half-twisted D-modules),
in the Betti setting it was a relatively recent discovery, due to D.~Ben-Zvi and D.~Nadler, see \cite{BZN}.  

\medskip

Namely, the entire category $\Shv^{\on{Betti}}_{\frac{1}{2}}(\Bun_G)$
of (twisted) Betti sheaves on $\Bun_G$ is too big to be equivalent
to something on the spectral side. What Ben-Zvi and Nadler discovered was that there is a reasonable automorphic candidate:
namely, this is the full category 
$$\Shv^{\on{Betti}}_{\frac{1}{2},\Nilp}(\Bun_G) \subset \Shv^{\on{Betti}}_{\frac{1}{2}}(\Bun_G)$$
of Betti sheaves with singular support in the global nilpotent cone. 

\sssec{}

The spectral side in the Betti setting must also be modified as compared to the de Rham setting,
but that modification is evident: one should replace the stack $\LS_\cG$ of de Rham local
systems on $X$ (i.e., the stack classifying $\cG$-bundles with a connection) by the stack $\LS^{\on{Betti}}_\cG$
of Betti $\cG$-local systems (i.e., if we ignore the derived structure, this is the stack of homomorphisms 
$\pi_1(X)\to \cG$, divided by the action of $\cG$ by conjugation). 

\medskip

In retrospect, knowing the spectral side, one can convince oneself that $\Shv^{\on{Betti}}_{\frac{1}{2},\Nilp}(\Bun_G)$ is the right candidate for the
automorphic category as follows: this is the largest subcategory of $\Shv^{\on{Betti}}_{\frac{1}{2}}(\Bun_G)$, on 
which the Hecke action gives rise to an action of $\QCoh(\LS_\cG^{\on{Betti}})$, see \cite[Theorem 18.1.4]{AGKRRV}. 

\sssec{}

With the candidates for the automorphic and spectral sides in place, the construction of the Langlands functor
in the Betti setting
\begin{equation} \label{e:L Betti intro}
\BL^{\on{Betti}}_G:\Shv^{\on{Betti}}_{\frac{1}{2},\Nilp}(\Bun_G) \to \IndCoh_\Nilp(\LS^{\on{Betti}}_\cG)
\end{equation} 
proceeds along the same lines as in the de Rham setting, with one modification: 

\medskip

The functor $\on{coeff}^{\on{Vac,glob}}$
in the de Rham setting appealed to the exponential D-module on $\BG_a$. There is (obviously) no Betti analog
of the exponential D-module, but one finds a
substitute using an additional group of symmetries (given by the torus $T$) acting in our situation, see
\secref{ss:Poinc Betti}. 

\sssec{}

The geometric Langlands conjecture (GLC) in the Betti setting says that the functor \eqref{e:L Betti intro}
is an equivalence. 

\ssec{Comparison of the different versions of GLC} \label{ss:comparison Intro} 

In addition to the de Rham (resp., Betti) versions of GLC, which we will refer to as \emph{full} GLC, in both settings
there exist other versions, which we refer to as \emph{restricted} and \emph{tempered}, respectively. And there exist
also the restricted tempered versions. 

\medskip

However, it turns out that all these versions are logically equivalent in the sense
that any one implies the others.

\sssec{}

The restricted version of GLC, in either de Rham or Betti versions, was introduced\footnote{Without specifying the functor in either
direction.} in \cite[Conjecture 21.2.7]{AGKRRV}. 

\medskip

On the automorphic side of the restricted GLC in the de Rham setting, we have the full subcategory
$$\Dmod_{\frac{1}{2},\Nilp}(\Bun_G)\subset \Dmod_{\frac{1}{2}}(\Bun_G)$$
that consists of (twisted) D-modules with singular support in the nilpotent cone $\Nilp\subset T^*(\Bun_G)$. 
Objects of this category are (obviously) ind-holonomic. However, one can show that actually
have regular singularities (this is a non-trivial result \cite[Corollary 16.5.6]{AGKRRV}).

\medskip

On the automorphic side of the restricted GLC in the Betti setting, we have the full subcategory 
$$\Shv^{\on{Betti,constr}}_{\frac{1}{2},\Nilp}(\Bun_G)\subset \Shv^{\on{Betti}}_{\frac{1}{2},\Nilp}(\Bun_G)$$
that consists of \emph{ind-constructible} sheaves.

\medskip

In fact, both categories 
$$\Dmod_{\frac{1}{2},\Nilp}(\Bun_G) \text{ and } \Shv^{\on{Betti,constr}}_{\frac{1}{2},\Nilp}(\Bun_G)$$
fall into the paradigm of \cite[Sect. 14.1]{AGKRRV}, where the restricted automorphic category is defined
starting from a \emph{constructible sheaf theory} $\Shv(-)$ as in \cite[Sect. 1.1.1]{AGKRRV}.
The two sheaf theories in question are
$$\Shv(-):=\Dmod^{\on{RS}}(-) \text{ and } \Shv(-):=\Shv^{\on{Betti,constr}}(-),$$
respectively.

\sssec{}

On the spectral side of the restricted GLC we have the categories 
\begin{equation} \label{e:spectral side restr Intro}
\IndCoh_\Nilp(\LS^{\on{restr}}_\cG) \text{ and } \IndCoh_\Nilp(\LS^{\on{Betti,restr}}_\cG),
\end{equation} 
respectively, where $\LS^{\on{restr}}_\cG$ and $\LS^{\on{Betti,restr}}_\cG$ are the de Rham and Betti 
versions of the \emph{stack of local systems with restricted variation} of \cite[Sect. 1.4]{AGKRRV}. 

\sssec{}

Starting from the full geometric Langlands functors $\BL_G$ and $\BL_G^{\on{Betti}}$ one produces their 
restricted versions
$$\Dmod_{\frac{1}{2},\Nilp}(\Bun_G) \overset{\BL_G^{\on{restr}}}\longrightarrow \IndCoh_\Nilp(\LS^{\on{restr}}_\cG)$$
and 
$$\Shv^{\on{Betti,constr}}_{\frac{1}{2},\Nilp}(\Bun_G) \overset{\BL_G^{\on{Betti,restr}}}\longrightarrow \IndCoh_\Nilp(\LS^{\on{Betti,restr}}_\cG).$$
by applying the operations
$$-\underset{\QCoh(\LS_\cG)}\otimes \QCoh(\LS_\cG)_{\on{restr}} \text{ and }
-\underset{\QCoh(\LS^{\on{Betti}}_\cG)}\otimes \QCoh(\LS^{\on{Betti}}_\cG)_{\on{restr}},$$
where
$$ \QCoh(\LS_\cG)_{\on{restr}} \subset \QCoh(\LS_\cG) \text{ and }
\QCoh(\LS^{\on{Betti}}_\cG)_{\on{restr}} \subset \QCoh(\LS^{\on{Betti}}_\cG)$$
are full subcategories that consist of objects set-theoretically supported on
$$\LS^{\on{restr}}_\cG\subset \LS_\cG \text{ and } \LS^{\on{Betti,restr}}_\cG\subset \LS^{\on{Betti}}_\cG,$$
respectively. 

\medskip

It is clear that if $\BL_G$ and $\BL_G^{\on{Betti}}$ are equivalences, then so are $\BL_G^{\on{restr}}$ and
$\BL_G^{\on{Betti,restr}}$. However, it turns out that the converse implications also take place.
This is proved in \secref{s:restricted vs full} of the present paper; see also \secref{sss:restr implies Intro} below. This material builds on 
\cite[Sect. 21.4]{AGKRRV}.

\sssec{}

Thus, we have:
$$\text{Full de Rham GLC}\, \Leftrightarrow\, \text{Restricted de Rham GLC}$$
and  
$$\text{Full Betti GLC}\, \Leftrightarrow\, \text{Restricted Betti GLC}.$$ 

\medskip

However, Riemann-Hilbert implies that we also have
$$\text{Restricted de Rham GLC}  \Leftrightarrow\, \text{Restricted Betti GLC}.$$

As a result, we deduce that
\begin{equation} \label{e:dR vs B}
\text{Full de Rham GLC}\, \Leftrightarrow\, \text{Full Betti GLC}.
\end{equation} 

\sssec{}

In the rest of this series of papers we will focus on the de Rham version of GLC. 

\medskip

This is because it is 
in this context that one can use Kac-Moody localization, which provides a powerful tool
in the study of $\Dmod_{\frac{1}{2}}(\Bun_G)$. It would not be an exaggeration to say that the proofs
of all the key results\footnote{Including the existence of the spectral action.} 
in this series of papers ultimately appeal to this method.

\medskip

However, thanks to \eqref{e:dR vs B}, once we prove the de Rham version of GLC, we 
deduce the Betti case as well. 

\sssec{} \label{sss:restr implies Intro}

Let us now comment on how we deduce 
$$\text{Restricted de Rham GLC}\, \Rightarrow \, \text{Full de Rham GLC}$$
(the Betti case is similar).

\medskip

To do so, we introduce yet another version of GLC: the tempered one. Namely, the category $\Dmod_{\frac{1}{2}}(\Bun_G)$
admits a localization
$$\bu:\Dmod_{\frac{1}{2}}(\Bun_G)_{\on{temp}} \rightleftarrows \Dmod_{\frac{1}{2}}(\Bun_G):\bu^R,$$
and similarly for its the restricted version
$$\bu:\Dmod_{\frac{1}{2},\Nilp}(\Bun_G)_{\on{temp}} \rightleftarrows \Dmod_{\frac{1}{2},\Nilp}(\Bun_G):\bu^R.$$

The tempered versions of the Langlands functor map
$$\Dmod_{\frac{1}{2}}(\Bun_G)_{\on{temp}}  \overset{\BL_{G,\on{temp}}}\longrightarrow \QCoh(\LS_\cG)$$
and 
$$\Dmod_{\frac{1}{2},\Nilp}(\Bun_G)_{\on{temp}}  \overset{\BL^{\on{restr}}_{G,\on{temp}}}\longrightarrow \QCoh(\LS^{\on{restr}}_\cG),$$
respectively.

\medskip

The tempered (resp., tempered restricted) version of GLC says that the functor $\BL_{G,\on{temp}}$ (resp., $\BL^{\on{restr}}_{G,\on{temp}}$)
is an equivalence. 

\sssec{}

It is clear that we have the implications
$$\text{Full GLC} \, \Rightarrow \, \text{Tempered GLC} \,\,\text{  and  }\,\, 
\text{Restricted GLC} \, \Rightarrow \, \text{Tempered restricted GLC}$$

In \secref{ss:temp to all} we establish an inverse implication
$$\text{Tempered GLC}  \, \Rightarrow \,  \text{Full GLC}.$$

(A similar argument shows that $\text{Tempered restricted GLC} \, \Rightarrow \, \text{Restricted GLC}$.) 

\sssec{}

It is also clear that
$$\text{Tempered GLC}  \, \Rightarrow \, \text{Tempered restricted GLC}.$$

However, by mimicking the argument of \cite[Sect. 21.4]{AGKRRV} one proves that we in fact have 
$$\text{Tempered GLC}  \, \Leftrightarrow \, \text{Tempered restricted GLC}.$$

\begin{rem}
In fact, in \cite[Sect. 21.4]{AGKRRV} the equivalence 
$$\text{Full GLC}\, \Leftrightarrow\, \text{Restricted GLC}$$
was proved directly, but under the assumption that the functor $\BL_G$ admits a left adjoint. 

\medskip

The advantage of working with the tempered category is that the existence of the left adjoint 
of $\BL_{G,\on{temp}}$ is automatic from the construction.

\end{rem} 

\ssec{Characteristic cycles of Hecke eigensheaves}

We prove an additional result  the interaction between
de Rham and Betti geometric Langlands.

\begin{thm}\label{t:cc-intro}

Suppose that $G$ has connected center, the genus of $X$ is $\geq 2$, and 
assume the geometric Langlands conjecture. Let $\sigma$ be an irreducible
$\cG$-local system and let $\CF_{\sigma}$ be the corresponding Hecke eigensheaf.

Then the characteristic cycle $\on{CC}(\CF_{\sigma})$ of $\CF_{\sigma}$ equals
the global nilpotent cone.

\end{thm}

We prove this result via the interaction of Betti and de Rham geometric Langlands,
and it is closely related to perversity properties of $\BL_{G,\on{temp}}$
used elsewhere in the paper, 
which is why we consider it relevant to this paper.

\medskip 

We refer to \secref{s:eigensheaves} for more details, including a review
of other known properties of $\CF_{\sigma}$.

%

\ssec{Contents}

We now review the contents of the paper by section.

\sssec{}

In \secref{s:functor} we construct the Langlands functor $\BL_G$ in the
de Rham context.

\medskip

We first collect the ingredients needed for the construction of the 
\emph{coarse} Langlands functor (the spectral action, the vacuum Poincar\'e object).

\medskip

We state \thmref{t:compact to bdd below dr}, which is the tool that allows us to upgrade 
$\BL_{G,\on{coarse}}$ to the Langlands functor $\BL_G$. 

\medskip

We show that $\BL_G$ is compatible with the actions of $\QCoh(\LS_\cG)$ on the two sides. 

\sssec{}

In \secref{s:proof of dr} we prove \thmref{t:compact to bdd below dr}.

\medskip

We show that compact objects of $\Dmod_{\frac{1}{2}}(\Bun_G)$ are bounded below (in fact, this
is a general property of the category of D-modules on a \emph{truncatable} algebraic stack with an
affine diagonal), see \propref{c:bdd right}.  This reduces \thmref{t:compact to bdd below dr} to
\thmref{t:L coarse left exact}, which says that the functor $\BL_{G,\on{coarse}}$ has a cohomological amplitude bounded 
on the left.

\medskip

We express the condition on an object of $\QCoh(\CY)$ (here $\CY$ is an arbitrary eventually coconnective stack)
to be bounded on the left in terms of its !-fibers, see \propref{p:when >0}. 

\medskip

The latter proposition allows us to reduce \thmref{t:L coarse left exact} to its version, where instead of $\BL_{G,\on{coarse}}$,
we are dealing with its restricted version, $\BL^{\on{restr}}_{G,\on{coarse}}$. However, in the latter case, the required assertion has
been already established in \cite{FR}. 

\medskip

By a similar manipulation, we prove \propref{p:bdd right}, which says that the functor $\BL_{G,\on{coarse}}$ (and, hence, $\BL_G$)
has a cohomological amplitude bounded on the right. 

\sssec{}

In \secref{s:Betti} we construct the geometric Langlands functor $\BL_G$ in the Betti setting.

\medskip

We first review the basics of the category $\Shv^{\on{Betti}}_{\frac{1}{2},\Nilp}(\Bun_G)$
(compact generation, relation to the entire category $\Shv^{\on{Betti}}_{\frac{1}{2}}(\Bun_G)$).

\medskip

We construct the Betti version of the vacuum Poincar\'e object, which requires a little bit of work,
due to the fact that the exponential sheaf does not exist in the Betti context.

\medskip

We state \thmref{t:compact to bdd below Betti}, which is a Betti counterpart of \thmref{t:compact to bdd below dr}.
Assuming this theorem, we construct $\BL_G^{\on{Betti}}$, the Betti version of the Langlands functor, by the
same procedure as in the de Rham context.

\medskip

We state \thmref{t:dR => Betti bis}, which says that the de Rham and Betti versions of GLC (along with their
restricted variants) are all logically equivalent.

\sssec{}

In \secref{s:proof of Betti >0}, we prove \thmref{t:compact to bdd below Betti}. 

\medskip

We note that the proof of \thmref{t:compact to bdd below Betti} does \emph{not} mimic that of 
\thmref{t:compact to bdd below dr}:

\medskip

We do not a priori know that the compact generators of 
$\Shv^{\on{Betti}}_{\frac{1}{2},\Nilp}(\Bun_G)$ are bounded below (although ultimately one can show that they are). 
Instead, we describe these compact generators explicitly and show directly that the functor 
$\BL_{G,\on{coarse}}^{\on{Betti}}$ sends them to bounded below objects. While doing so,
we study the interactions between the following functors:

\smallskip

\begin{itemize}

\item The \emph{left} adjoint to $\bi:\Shv^{\on{Betti}}_{\frac{1}{2},\Nilp}(\Bun_G)\hookrightarrow \Shv^{\on{Betti}}_{\frac{1}{2}}(\Bun_G)$;

\item The \emph{right} adjoint to $\iota^{\on{Betti}}:\Shv^{\on{Betti,constr}}_{\frac{1}{2},\Nilp}(\Bun_G)\hookrightarrow \Shv^{\on{Betti}}_{\frac{1}{2},\Nilp}(\Bun_G)$;

\item The \emph{right} adjoint to 
$\bi^{\on{constr}}:\Shv^{\on{Betti,contsr}}_{\frac{1}{2},\Nilp}(\Bun_G)\hookrightarrow \Shv^{\on{Betti,constr}}_{\frac{1}{2}}(\Bun_G)$.

\end{itemize}

\medskip

In fact, it turns out that $(\bi^{\on{constr}})^R \simeq (\iota^{Betti})^R\circ \bi\circ \oblv^{\on{constr}}$, as functors
$$\Shv^{\on{Betti,constr}}_{\frac{1}{2}}(\Bun_G)\to \Shv^{\on{Betti,constr}}_{\frac{1}{2},\Nilp}(\Bun_G),$$
where $\oblv^{\on{constr}}$ is the forgetful functor
$$\Shv^{\on{Betti,constr}}_{\frac{1}{2}}(\Bun_G)\to \Shv^{\on{Betti}}_{\frac{1}{2}}(\Bun_G).$$

\sssec{}

In \secref{s:restricted vs full}, we prove the implications
$$\text{Restricted de Rham GLC}\, \Rightarrow \, \text{Full de Rham GLC}\,\,\, \text{ and }\,\,\, 
\text{Restricted Betti GLC}\, \Rightarrow \, \text{Full Betti GLC}$$
along the lines explained in \secref{ss:comparison Intro}.

\medskip

We introduce the tempered GLC, in both full and restricted variants, and show, following \cite[Sect. 21.4]{AGKRRV}
that these two variants are logically equivalent.

\medskip

We then show that the tempered GLC is equivalent to the original full GLC.  

\sssec{} 

In \secref{s:eigensheaves}, we calculate characteristic cycles of Hecke eigensheaves
(at least when $G$ has connected center). We do this by proving constancy
of characteristic cycles along the moduli of irreducible local systems
and then appealing to \cite{BD} to compute this constant value.

\sssec{Notations and conventions}

Notations and conventions in this paper largely follow those adopted in \cite{AGKRRV}. 

\ssec{Acknowledgements}

We wish to thank our collaborators on this five paper series: D.~Arinkin, D.~Beraldo, L.~Chen, J.~Faegerman, K.~Lin
and N.~Rozenblyum. 

\medskip

Special thanks are due to A.~Beilinson and V.~Drinfeld, who initiated the study of the geometric Langlands
phenomenon in the context of D-modules, and produced what is still the most significant piece of work on the
subject to date, namely, \cite{BD}. In addition, the idea of Langlands correspondence as an equivalence of
categories, generalizing Fourier-Laumon transform, is also due to them. 

\medskip

We also wish to express our gratitude to the following mathematicians, whose work was crucial for the development of 
the field and whose ideas shaped our understanding of the subject: 
D.~Ben-Zvi, J.~Bernstein, R.~Bezrukavnikov, A.~Braverman, J.~Campbell, P.~Deligne,
R.~Donagi, G.~Dhillon, L.~Fargues, B.~Feigin, M.~Finkelberg, E.~Frenkel, D.~Gaiotto, A.~Genestier, V.~Ginzburg,
S.~Gukov, J.~Heinloth, A.~Kapustin, D.~Kazhdan,
V.~Lafforgue, G.~Laumon, G.~Lusztig, S.~Lysenko, I.~Mirkovi\'c, D.~Nadler, T.~Pantev, P.~Scholze,
C.~Teleman, Y.~Varshavsky, K.~Vilonen, E.~Witten, C.~Xue, Z.~Yun and X.~Zhu.

\medskip

We dedicate this paper with love to Sasha Beilinson, who introduced us both
to this subject, and who suggested long ago that the 
geometric Langlands functor ought to take a simple form.

\medskip 

The work of D.G. was supported by NSF grant DMS-2005475. 
The work of S.R. was supported by NSF grant DMS-2101984 and a Sloan Research Fellowship 
while this work was in preparation.

\section{Construction of the Langlands functor (de Rham context)} \label{s:functor}

In this section we will first review the ingredients that are needed for the construction of the Langlands functor,
and then carry out the construction in question, all in the context of D-modules. 

\medskip

The ``coarse" version of the functor will be automatic from the spectral action (see \secref{ss:coarse}). 
To upgrade it to the actual Langlands functor, we will need some cohomological estimates, which are
provided by \thmref{t:compact to bdd below dr}. 

\ssec{The automorphic category} \label{ss:autom}

In this subsection we will introduce the main player on the automorphic side: the category $\Dmod_{\frac{1}{2}}(\Bun_G)$
of (half-twisted) D-modules on the moduli space $\Bun_G$ of $G$-bundles on $X$. 

\sssec{} \label{sss:crit 1/2 can}

Let $\det_{\Bun_G}$ be the determinant line bundle on $\Bun_G$, normalized so that it sends a $G$-bundle $\CP_G$ 
to 
$$\det\left(\Gamma(X,\fg_{\CP_G})\right) \otimes \det\left(\Gamma(X,\fg_{\CP^0_G})\right)^{\otimes -1},$$
where $\CP^0_G$ is the trivial bundle.

\medskip

Note also that up to the (constant) line $\det\left(\Gamma(X,\fg_{\CP^0_G})\right)$, the line bundle $\det_{\Bun_G}$
identifies with the canonical line bundle on $\Bun_G$.

\sssec{} \label{sss:pfaff}

According to \cite[Sect. 4]{BD}, the choice of $\omega^{\otimes \frac{1}{2}}_X$ gives rise to a square root of $\det_{\Bun_G}$.

\sssec{}

Let $\CY$ be a space, and let $\CG$ be an \'etale $\mu_n$-gerbe on $\CY$ for some $n\in \BN$.
Denote by 
$$\Dmod_\CG(\CY)$$
the corresponding category of twisted D-modules on $\CY$.

\sssec{} \label{sss:autom categ}

Let $$\Dmod_{\frac{1}{2}}(\Bun_G)$$
be the twisted category of D-modules, corresponding to the $\mu_2$-gerbe $\det^{\frac{1}{2}}_{\Bun_G}$ of square roots of $\det_{\Bun_G}$. 

\medskip

The category $\Dmod_{\frac{1}{2}}(\Bun_G)$ is the de Rham incarnation of 
the automorphic category. It is the primary object
of study in the geometric Langlands theory.  

\sssec{}

Note, however, that by \secref{sss:pfaff}, we have an identification
\begin{equation} \label{e:untwist}
\Dmod_{\frac{1}{2}}(\Bun_G)\simeq \Dmod(\Bun_G).
\end{equation}

Such an identification can be convenient for the study of local properties of the category 
$\Dmod_{\frac{1}{2}}(\Bun_G)$. In particular, the results from \cite{AGKRRV} pertaining
to $\Dmod_\Nilp(\Bun_G)$ apply to the corresponding category 
$$\Dmod_{\frac{1}{2},\Nilp}(\Bun_G).$$

\begin{rem}
It is crucial, however, to consider the twisted version, when we study the Hecke functors. 
Otherwise, instead of the usual category $\Rep(\cG)$ one has to consider its twist by a 
certain canonical $Z_\cG$-gerbe, see \cite[Sect. 6.3]{GLys}.
\end{rem} 

\ssec{The spectral action} \label{ss:spectral}

An ingredient, whch is crucial for both for the construction of the Langlands functor and to the eventual proof
that it is an equivalence, is the fact that the action of Hecke functors on $\Dmod_{\frac{1}{2}}(\Bun_G)$
gives rise to a monoidal action of the category $\QCoh(\LS_\cG)$.

\sssec{}

Let $\Rep(\cG)_\Ran$ be the de Rham version of the category $\Rep(\cG)$ spread over
the Ran space, see \cite[Sect. 4.2]{Ga2}.

\medskip

The (naive) geometric Satake functor and the Hecke action 
define an action of $\Rep(\cG)_\Ran$ on $\Dmod_{\frac{1}{2}}(\Bun_G)$.

\sssec{}

Recall now that we have a (symmetric) monoidal localization functor
\begin{equation} \label{e:spec Loc}
\Loc_\cG^{\on{spec}}:\Rep(\cG)\to \QCoh(\LS_\cG),
\end{equation} 
see \cite[Sect. 4.3]{Ga2}. 

\medskip

The category $\Rep(\cG)_\Ran$ is compactly generated, and the functor $\Loc_\cG^{\on{spec}}$
preserves compactness (being a symmetric monoidal functor between rigid symmetric monoidal
categories). Hence, it admits a continuous right adjoint.

\medskip

It is a basic fact that this right adjoint is fully faithful.
I.e., the functor $\Loc_\cG^{\on{spec}}$ is a localization. 

\sssec{}

We have the following result (see \cite[Corollary 4.5.5]{Ga2}):

\begin{thm} \label{t:spectral decomp}
The action of $\Rep(\cG)_\Ran$ on $\Dmod_{\frac{1}{2}}(\Bun_G)$ 
factors via the functor \eqref{e:spec Loc}.
\end{thm} 

\sssec{} \label{sss:spectral action}

Thanks to \thmref{t:spectral decomp}, we obtain a canonically defined action of the monoidal category $\QCoh(\LS_\cG)$
on $\Dmod_{\frac{1}{2}}(\Bun_G)$. 

\medskip

We refer to it as the \emph{spectral action}. 

%
%
%
%
%

\ssec{The vacuum Poincar\'e sheaf}

The second ingredient for the construction of the Langlands functor $\BL_G$ is the object of 
$\Dmod_{\frac{1}{2}}(\Bun_G)$, which is the image of $\CO_{\LS_\cG}$ under the 
would-be left adjoint $\BL_G^L$ of $\BL_G$. 

\medskip

This object is the \emph{vacuum} Poincar\'e sheaf, introduced in this subsection. This particular
choice for the image of $\CO_{\LS_\cG}$ reflects the idea, familiar from the classical theory of
automorphic functions, that the Langlands correspondence is normalized by the Whittaker model. 

\sssec{} \label{sss:crit to BunN}

Let $\CP_T$ be any $T$-bundle. Consider the stack
$$\Bun_{N,\CP_T}\simeq \Bun_B\underset{\Bun_T}\times \on{pt},$$
where the map $\on{pt}\to \Bun_T$ corresponds the point $\CP_T$. Denote by $\fp$ the map
$$\Bun_{N,\CP_T}\to \Bun_G.$$
Note that the pullback of $\det_{\Bun_G}$ along this map is canonically constant. 
Denote the resulting line by
$$\fl_{G,N_{\CP_T}}.$$

\sssec{}

The Langlands functor depends on the choice of the square root $\omega_X^{\frac{1}{2}}$
of the canonical line bundle on $X$, which we fix from now on. With this choice we claim:

\begin{prop} \label{p:det to BunN}
The line $\fl_{G,N_{\CP_T}}$ admits a canonical 
square root. 
\end{prop}

\begin{proof} 

We claim more generally that for any parabolic $P$ with Levi quotient $M$ and an $M$-bundle $\CP_M$, line 
$$\det_{\Bun_G}|_{\CP_M}\otimes \det^{\otimes -1}_{\Bun_M}|_{\CP_M}$$
admits a canonical square root.

\medskip

Let $P$ and $P^-$ be the positive and negative parabolics corresponding to $M$, and let
$N_P$ and $N^-_P$ be their respective uniptotent radicals. 

\medskip

By definition,
\begin{multline} \label{e:rel det}
\det_{\Bun_G}|_{\CP_M}\otimes \det^{\otimes -1}_{\Bun_M}|_{\CP_M}\simeq \\
\simeq \det(\Gamma(X,(\fn_P)_{\CP_M})\otimes \det(\Gamma(X,(\fn^-_P)_{\CP_M}))^{\otimes -1}
\otimes \det(\Gamma(X,\fn_P\otimes \CO_X))^{\otimes -1}\otimes \det(\Gamma(X,\fn^-_P\otimes \CO_X)).
\end{multline} 

We claim that the lines
\begin{equation} \label{e:rel det 1}
\det(\Gamma(X,(\fn_P)_{\CP_M}\otimes \omega_X^{\frac{1}{2}})\otimes 
\det(\Gamma(X,\fn_P\otimes \omega_X^{\frac{1}{2}}))^{\otimes -1}
\end{equation} 
and 
\begin{equation} \label{e:rel det 2}
\det(\Gamma(X,(\fn^-_P)_{\CP_M}\otimes \omega_X^{\frac{1}{2}})\otimes 
\det(\Gamma(X,\fn^-_P\otimes \omega_X^{\frac{1}{2}}))^{\otimes -1}
\end{equation} 
are canonically isomorphic and their tensor product is canonically isomorphic to \eqref{e:rel det}.
This would produce the sought-for square root of \eqref{e:rel det}.

\medskip

To construct an isomorphism between \eqref{e:rel det 1} and \eqref{e:rel det 2}, we identify $\fn_P$ with
the dual of $\fn^-_P$ using the Killing form. Hence, we have
\begin{equation} \label{e:vector bundle duality}
((\fn_P)_{\CP_M})^\vee \simeq (\fn^-_P)_{\CP_M} \text{ and } (\fn_P)^*\simeq \fn^-_P
\end{equation}
as vector bundles (resp., vector spaces). 

\medskip

Now the desired identification follows from the fact that
for a vector bundle $\CE$ on $X$, we have
\begin{equation}
\det(\Gamma(X,\CE))\simeq \det(\Gamma(X,\CE^\vee\otimes \omega_X)).
\end{equation} 

\medskip

To construct an isomorphism between the tensor product of \eqref{e:rel det 1} and \eqref{e:rel det 2} and the
right-hand side in \eqref{e:rel det},
we recall the formula 
\begin{equation} \label{e:det formula}
\det(\Gamma(X,\CE\otimes \CL))\otimes \det(\Gamma(X,\CO_X))^{\otimes \on{rk}(\CE)}
\simeq \det(\Gamma(X,\CE))\otimes \det(\Gamma(X,\CL))^{\otimes \on{rk}(\CE)}\otimes \on{Weil}(\det(\CE),\CL),
\end{equation} 
where:

\begin{itemize}

\item $\CE$ is a vector bundle on $X$;

\item $\CL$ is a line bundle on $X$;

\item $\on{Weil}(-,-)$ is the Weil pairing.

\end{itemize}

\medskip

Hence, the ratio of the tensor product of \eqref{e:rel det 1} and \eqref{e:rel det 2} and the
right-hand side in \eqref{e:rel det} is 
\begin{multline} 
\on{Weil}(\det((\fn_P)_{\CP_M}),\omega_X^{\frac{1}{2}})\otimes 
\on{Weil}(\det(\fn_P)\otimes \CO_X,\omega_X^{\frac{1}{2}})^{\otimes -1}\otimes \\
\otimes \on{Weil}(\det((\fn^-_P)_{\CP_M}),\omega_X^{\frac{1}{2}})\otimes 
\on{Weil}(\det(\fn^-_P)\otimes \CO_X,\omega_X^{\frac{1}{2}})^{\otimes -1}.
\end{multline} 

However, the latter line is canonically trivialized thanks to \eqref{e:vector bundle duality}.

\end{proof} 

\begin{cor} \label{c:det to BunN rho}
The line $\fl_{G,N_{\rho(\omega_X)}}$ admits a canonical 
square root. 
\end{cor}

\sssec{}

Thanks to \corref{c:det to BunN rho}, we have a well-defined functor
$$\fp^!:\Dmod_{\frac{1}{2}}(\Bun_G)\to \Dmod(\Bun_{N,\rho(\omega_X)}),$$
with a partially defined left adjoint, denoted $\fp_!$. 

\sssec{}

Note now that there is a canonical map
$$\chi:\Bun_{N,\rho(\omega_X)}\to \BG_a.$$

Indeed, for every vertex $i$ of the Dynkin diagram of $G$, we have a map
$$\Bun_{N,\rho(\omega_X)}\to \Bun_{N_i,\rho_i(\omega_X)},$$
where $(N_i,\chi_i)$ is the corresponding pair for the subminimal Levi subgroup attached to $i$. 

\medskip

For every $i$, we have a canonical map 
$$\chi_i:\Bun_{N_i,\rho_i(\omega_X)}\to \BG_a$$
that records the class of the extension.

\medskip

We let $\chi$ be the map
$$\Bun_{N,\rho(\omega_X)} \to \underset{i}\Pi\, \Bun_{N_i,\rho_i(\omega_X)} \overset{\underset{i}\Pi\, \chi_i}\longrightarrow
\underset{i}\Pi\, \BG_a \overset{\on{sum}}\to \BG_a.$$

\sssec{}

We let 
$$\on{Poinc}^{\on{Vac,glob}}_{G,!}\in  \Dmod_{\frac{1}{2}}(\Bun_G)$$
be the object equal to 
$$\fp_!\circ \chi^*(\on{exp}),$$
where:

\begin{itemize}

\item $\on{exp}\in \Dmod(\BG_a)$ is the exponential D-module, normalized so that
$$\on{add}^*(\on{exp})\simeq \on{exp}\boxtimes \on{exp},$$
(i.e., $\on{exp}\in \Dmod(\BG_a)^\heartsuit[-1]$);

\medskip

\item The functor $\fp_!$ is well-defined on $\chi^*(\on{exp})$, since the latter is holonomic.

\end{itemize} 

\sssec{} \label{sss:coeff}

Let 
$$\on{coeff}^{\on{Vac,glob}}:\Dmod_{\frac{1}{2}}(\Bun_G)\to \Vect$$
denote the functor co-represented by $\on{Poinc}^{\on{Vac,glob}}_{G,!}$.

\medskip

Explicitly,
$$\on{coeff}^{\on{Vac,glob}}\simeq \on{C}^\cdot(\Bun_{N,\rho(\omega_X)},\fp^!(-)\overset{*}\otimes \on{exp}_\chi),$$
where\footnote{Strictly speaking, in the formula below one should have replaces $\chi$ by $-\chi$. However, it does not matter
because the situation is equivariant with respect to the action of $T$.}
$$\on{exp}_\chi:=\chi^*(\on{exp})\in \Dmod(\Bun_{N,\rho(\omega_X)}).$$

\ssec{The coarse version of the functor} \label{ss:coarse}

In this subsection we will construct a \emph{coarse} version of the Langlands functor $\BL_G$,
denoted $\BL_{G,\on{coarse}}$. 

\medskip

The difference between the two versions is that $\BL_G$ is supposed to
take values in the category $\IndCoh_\Nilp(\LS_\cG)$, while $\BL_{G,\on{coarse}}$ takes values
in the usual category $\QCoh(\LS_\cG)$. 

\sssec{}

Note that by \secref{sss:spectral action}, a choice of an object in $\Dmod_{\frac{1}{2}}(\Bun_G)$
defines a $\QCoh(\LS_\cG)$-linear functor $\QCoh(\LS_\cG)\to \Dmod_{\frac{1}{2}}(\Bun_G)$. 

\sssec{}

We define the functor
\begin{equation} \label{e:LL temp}
\BL^L_{G,\on{temp}}:\QCoh(\LS_\cG)\to \Dmod_{\frac{1}{2}}(\Bun_G).
\end{equation}
to correspond to the object 
\begin{equation} \label{e:Poinc Vac glob}
\on{Poinc}^{\on{Vac,glob}}_{G,!}\in  \Dmod_{\frac{1}{2}}(\Bun_G).
\end{equation}

\medskip 

Since the object $\on{Poinc}^{\on{Vac,glob}}_{G,!}$ is compact and the monoidal category $\QCoh(\LS_\cG)$
is compactly generated and rigid, the functor $\BL^L_{G,\on{temp}}$ preserves compactness.

\begin{rem}
The notation $\BL^L_{G,\on{temp}}$ has to do with the fact that the essential image 
of this functor lands in the subcategory 
$$\Dmod_{\frac{1}{2}}(\Bun_G)_{\on{temp}}\subset \Dmod_{\frac{1}{2}}(\Bun_G),$$
see \secref{sss:LL temp}.
\end{rem} 

\begin{rem}

Once GLC (\conjref{c:GLC dR}) is proved, it would follow that the functor $\BL^L_{G,\on{temp}}$
is actually an equivalence
$$\QCoh(\LS_\cG)\overset{\sim}\to \Dmod_{\frac{1}{2}}(\Bun_G)_{\on{temp}},$$
see Remark \ref{r:L temp}.

\end{rem} 

\sssec{}

We let 
$$\BL_{G,\on{coarse}}:\Dmod_{\frac{1}{2}}(\Bun_G)\to \QCoh(\LS_\cG)$$
be the right adjoint of $\BL^L_{G,\on{temp}}$.

\medskip 

Since the category $\QCoh(\LS_\cG)$ is compactly generated and the functor $\BL^L_{G,\on{temp}}$ preserves compactness, 
the functor $\BL_{G,\on{temp}}$ is continuous. 
By rigidity, $\BL_{G,\on{temp}}$ is automatically $\QCoh(\LS_\cG)$-linear. 

\sssec{}

One can equivalently describe $\BL_{G,\on{coarse}}$ as follows. It is uniquely characterized by the following
two pieces of structure: 

\begin{itemize}

\item $\BL_G$ is $\QCoh(\LS_\cG)$-linear;

\smallskip

\item $\Gamma(\LS_\cG,\BL_{G,\on{coarse}}(-))\simeq \on{coeff}^{\on{Vac,glob}}$.

\end{itemize} 

\ssec{The case \texorpdfstring{$G=T$}{G=T}} \label{ss:torus case}

\sssec{}

Let $G=T$ be a torus. Consider the Fourier-Mukai equivalence
$$\on{FM}:\QCoh(\Bun_T)\to \QCoh(\Bun_\cT),$$
given by the Poincar\'e line bundle 
$$\CL_{\on{Poinc}}\in \QCoh(\Bun_T\times \Bun_\cT),$$
as a kernel, where $\CL_{\on{Poinc}}$, viewed as a map $\Bun_T\times \Bun_\cT\to B\BG_m$, is
given by the Weil pairing.

\sssec{}

It is known that $\on{FM}$ can be enhanced to an equivalence
$$\on{FM}^{\on{enh}}:\Dmod(\Bun_T)\to \QCoh(\LS_\cT),$$
that makes the following diagram commute:
$$
\CD
\Dmod(\Bun_T) @>{\on{FM}^{\on{enh}}}>> \QCoh(\LS_\cT) \\
@VVV @VVV \\
\QCoh(\Bun_T) @>{\on{FM}}>> \QCoh(\Bun_\cT),
\endCD
$$
where: 

\begin{itemize}

\item The functor $\Dmod(\Bun_T)\to \Dmod(\Bun_T)$ is $\oblv^r$, the forgetful functor for ``right" D-modules;

\medskip

\item The functor $\QCoh(\LS_\cT)\to \QCoh(\Bun_\cT)$ is direct image along the projection
$$\LS_\cT\to \Bun_\cT.$$

\end{itemize}

\sssec{}

Unwinding the definitions, we obtain that the functor $\BL_{T,\on{coarse}}$ identifies with
$$\on{FM}^{\on{enh}}\circ \tau_T,$$
where $\tau_T$ is the Cartan involution, i.e., the inversion automorphism, of $T$.

\sssec{}

Note that in the case of $G=T$, the subset $\Nilp\subset \on{Sing}(\LS_\cT)$ is the $0$-section. Hence
$$\IndCoh_\Nilp(\LS_\cT)=\QCoh(\LS_\cT).$$

We set
$$\BL_T:=\BL_{T,\on{coarse}}.$$

So in this case there is no difference between the coarse Langlands functor and the sought-for
Langlands functor $\BL_T$. 

\ssec{Statement of the main result}

In this subsection we will formulate \thmref{t:compact to bdd below dr}, which would allow us to upgrade $\BL_{G,\on{coarse}}$
to the actual Langlands functor $\BL_G$. 

\sssec{}

The main technical result of this paper pertaining to the Langlands functor in the de Rham context reads:

\begin{thm} \label{t:compact to bdd below dr}
The functor
$$\BL_{G,\on{coarse}}:\Dmod_{\frac{1}{2}}(\Bun_G)\to \QCoh(\LS_\cG)$$
sends compact objects in $\Dmod_{\frac{1}{2}}(\Bun_G)$ to bounded below (i.e., eventually coconnective)
objects in $\QCoh(\LS_\cG)$.
\end{thm} 

The proof will be given in \secref{s:proof of dr}. We now proceed to the consequences of \thmref{t:compact to bdd below dr}
that pertain to the geometric Langlands functor. 

\sssec{}

Let $\IndCoh_\Nilp(\LS_\cG)$ be the category of \cite[Sect. 11.1.5]{AG}. It is equipped with a forgetful functor  
\begin{equation} \label{e:coarsening}
\Psi_{\Nilp,\{0\}}:\IndCoh_\Nilp(\LS_\cG)\to  \QCoh(\LS_\cG)
\end{equation} 
is t-exact and induces an equivalence 
$$\IndCoh_\Nilp(\LS_\cG)^{>-\infty}\to \QCoh(\LS_\cG)^{>-\infty},$$
see \cite[Proposition 4.4.5]{AG}. 

\sssec{}

Combining this with \thmref{t:compact to bdd below dr} and the fact that $\Dmod_{\frac{1}{2}}(\Bun_G)$
is compactly generated (see \cite[Theorem 0.1.2]{DG1}), we obtain:

\begin{cor} \label{c:main dr}
There exists a (colimit-preserving) functor
$$\BL_G:\Dmod_{\frac{1}{2}}(\Bun_G)\to \IndCoh_\Nilp(\LS_\cG),$$
uniquely characterized by the following properties:

\begin{itemize}

\item The functor $\BL_G$  sends compact objects in 
$\Dmod_{\frac{1}{2}}(\Bun_G)$ to eventually coconnective objects
in $\IndCoh_\Nilp(\LS_\cG)$, i.e., to $\IndCoh_\Nilp(\LS_\cG)^{>-\infty}$. 

\smallskip

\item $\Psi_{\Nilp,\{0\}}\circ \BL_G\simeq \BL_{G,\on{coarse}}$;

\end{itemize}

\end{cor} 

\sssec{}

The functor $\BL_G$ is \emph{the} geometric Langlands functor in the de
Rham context. 

\medskip

We can now state the geometric Langlands conjecture (GLC) in the de Rham context:

\begin{conj} \label{c:GLC dR}
The functor $\BL_G$ is an equivalence.
\end{conj} 

\begin{rem}

By the same logic as in \corref{c:main dr}, using \thmref{t:compact to bdd below dr}, we can lift $\BL_{G,\on{coarse}}$ to a functor 
$$\BL'_G:\Dmod_{\frac{1}{2}}(\Bun_G)\to \IndCoh(\LS_\cG),$$
so that
$$\Psi_{\on{all},\Nilp}\circ \BL'_G\simeq \BL_G,$$
where $\Psi_{\on{all},\Nilp}$ is the right adjoint of the inclusion
$$\Xi_{\Nilp,\on{all}}:\IndCoh_\Nilp(\LS_\cG)\to \IndCoh(\LS_\cG).$$

However, if we accept GLC, it follows from \lemref{l:Eis} below that the functor $\BL'_G$
factors as 
$$\Xi_{\Nilp,\on{all}}\circ \BL_G.$$

So the refinement $\IndCoh_\Nilp(\LS_\cG)\rightsquigarrow \IndCoh(\LS_\cG)$ does not give
us anything new.

\end{rem}

\begin{rem}

That said, as was suggested by D.~Arinkin, 
one can consider a \emph{renormalized} version\footnote{The compact generators of the renormalization are !-extension of \emph{locally compact}
objects on q.c. open substacks of $\Bun_G$.} $\Dmod_{\frac{1}{2}}(\Bun_G)_{\on{ren}}$
of $\Dmod_{\frac{1}{2}}(\Bun_G)$,
equipped with a pair of adjoint functors
$$\on{ren}:\Dmod_{\frac{1}{2}}(\Bun_G)\rightleftarrows \Dmod_{\frac{1}{2}}(\Bun_G)_{\on{ren}}:\on{un-ren},$$
and one can refine $\BL_G$ to a functor
$$\BL_{G,\on{ren}}:\Dmod_{\frac{1}{2}}(\Bun_G)_{\on{ren}}\to \IndCoh(\LS_\cG),$$
so that the diagrams
$$
\CD
\Dmod_{\frac{1}{2}}(\Bun_G)_{\on{ren}} @>{\BL_{G,\on{ren}}}>> \IndCoh(\LS_\cG) \\
@V{\on{un-ren}}VV @VV{\Psi_{\on{all},\Nilp}}V \\
\Dmod_{\frac{1}{2}}(\Bun_G)@>{\BL_G}>> \IndCoh_\Nilp(\LS_\cG) 
\endCD
$$
and
$$
\CD
\Dmod_{\frac{1}{2}}(\Bun_G)_{\on{ren}} @>{\BL_{G,\on{ren}}}>> \IndCoh(\LS_\cG) \\
@A{\on{ren}}AA @AA{\Xi_{\on{all},\Nilp}}A \\
\Dmod_{\frac{1}{2}}(\Bun_G)@>{\BL_G}>> \IndCoh_\Nilp(\LS_\cG) 
\endCD
$$
commute. 

\medskip

One can then formulate a renormalized version of GLC, which says that $\BL_{G,\on{ren}}$ is an equivalence. 
This will be addressed in a separate paper.

\end{rem}

\ssec{Compatibility with the spectral action}

In this subsection we will assume the validity of \thmref{t:compact to bdd below dr}, and hence of
\corref{c:main dr}. 

\medskip

We will show that the functor $\BL_G$ naturally upgrades to a functor between
$\QCoh(\LS_\cG)$-module categories. 

\sssec{}

Note that the category $\IndCoh_\Nilp(\LS_\cG)$ is naturally a module over $\QCoh(\LS_\cG)$,
and the functor $\Psi_{\Nilp,\{0\}}$ is naturally $\QCoh(\LS_\cG)$-linear.

\medskip

Consider $\Dmod_{\frac{1}{2}}(\Bun_G)$ also as a module over $\QCoh(\LS_\cG)$, see \secref{ss:spectral}.
We claim:

\begin{prop} \label{p:L and Hecke}
The functor $\BL_G$ carries a uniquely defined $\QCoh(\LS_\cG)$-linear structure, 
so that the induced $\QCoh(\LS_\cG)$-linear structure on
$$\Psi_{\Nilp,\{0\}}\circ \BL_G\simeq \BL_{G,\on{coarse}}$$
is the natural $\QCoh(\LS_\cG)$-linear structure on $\BL_{G,\on{coarse}}$.
\end{prop} 

\begin{proof} 

Since $\QCoh(\LS_\cG)$ is compactly generated and rigid, in the statement of the proposition
we can replace the monoidal category $\QCoh(\LS_\cG)$ by its full subcategory $\QCoh(\LS_\cG)^c$. 

\medskip

By rigidity, the action of
$\QCoh(\LS_\cG)^c$ on $\Dmod_{\frac{1}{2}}(\Bun_G)$ preserves compactness.
Hence, it suffices to show that the functor
$$\BL_G|_{\Dmod_{\frac{1}{2}}(\Bun_G)^c}:\Dmod_{\frac{1}{2}}(\Bun_G)^c\to \IndCoh_\Nilp(\LS_\cG)$$
carries a uniquely defined $\QCoh(\LS_\cG)^c$-linear structure, so that its composition
with $\Psi_{\Nilp,\{0\}}$ reproduces the natural $\QCoh(\LS_\cG)^c$-linear structure on 
$$\BL_{G,\on{coarse}}|_{\Dmod_{\frac{1}{2}}(\Bun_G)^c}:\Dmod_{\frac{1}{2}}(\Bun_G)^c\to \QCoh(\LS_\cG).$$

\medskip

By the construction of $\BL_G$, the restriction $\BL_G|_{\Dmod_{\frac{1}{2}}(\Bun_G)^c}$ factors via 
a full subcategory 
$$\IndCoh_\Nilp(\LS_\cG)^{>-\infty}\subset \IndCoh_\Nilp(\LS_\cG).$$

Hence, it suffices to show that
$$\BL_G|_{\Dmod_{\frac{1}{2}}(\Bun_G)^c}:\Dmod_{\frac{1}{2}}(\Bun_G)^c\to \IndCoh_\Nilp(\LS_\cG)^{>-\infty}$$
carries a uniquely defined $\QCoh(\LS_\cG)^c$-linear structure, so that its composition
with $\Psi_{\Nilp,\{0\}}$ reproduces the natural $\QCoh(\LS_\cG)^c$-linear structure on 
$$\BL_{G,\on{coarse}}|_{\Dmod_{\frac{1}{2}}(\Bun_G)^c}:\Dmod_{\frac{1}{2}}(\Bun_G)^c\to \QCoh(\LS_\cG)^{>-\infty}.$$

However, the latter is automatic since $\Psi_{\Nilp,\{0\}}$ is an equivalence on the eventually coconnective
subcategories.

\end{proof}

\ssec{The bound on the right}

The contents of this subsection are tangential to rest of this paper (and will not be needed either for
the construction of $\BL_G$, or for the proof that it is an equivalence). 

\sssec{}

In \secref{ss:bdd right} we will prove:

\begin{prop} \label{p:bdd right} 
The functor $\BL_{G,\on{coarse}}$ has a bounded cohomological amplitude on the right, i.e., there exists an integer $d$ so that 
$\BL_{G,\on{coarse}}[d]$ is right t-exact. 
\end{prop}

\sssec{}

Since the functor $\Psi_{\Nilp,\{0\}}$ is t-exact and induces an equivalence on eventually coconnective
subcategories, from \propref{p:bdd right} we obtain:

\begin{cor} \label{c:bdd right} 
The functor $\BL_G$ has a bounded cohomological amplitude on the right. 
\end{cor} 

\section{Proof of \thmref{t:compact to bdd below dr}} \label{s:proof of dr}

In this section we will prove \thmref{t:compact to bdd below dr}, which was used in order to bootstrap
the coarse Langlands functor $\BL_{G,\on{coarse}}$ to the actual Langlands functor $\BL_G$.

\medskip

The main input is a result of \cite{FR}, stated here as \thmref{t:L coarse left exact Nilp}, which says that the \emph{restricted} 
version of the coarse Langlands functor is t-exact (up to a cohomological shift). 

\ssec{Strategy of the proof}

The statement of \thmref{t:compact to bdd below dr} follows immediately from the combination of the
following two results:

\begin{prop} \label{p:compact bdd below}
Compact objects of $\Dmod_{\frac{1}{2}}(\Bun_G)$ are bounded below (i.e., are eventually coconnective).  
\end{prop}

\begin{thm} \label{t:L coarse left exact}
The functor $\BL_{G,\on{coarse}}$ has a bounded cohomological amplitude on the left,  
i.e., there exists an integer $d$ so that $\BL_{G,\on{coarse}}[-d]$ is left t-exact. 
\end{thm}

\begin{rem} \label{r:left ampl}
In the course of the proof we will see that the integer $d$ in \thmref{t:L coarse left exact} can be taken to be 
the dimension of the classical algebraic stack underlying $\LS_\cG$ plus $\dim(\Bun_{N,\rho(\omega_X)})$
minus $\dim(\Bun_G)$. 

\medskip

Note that when the genus $g$ of our curve is $>1$ and $G$ is semi-simple, the stack $\LS_\cG$ is a classical l.c.i., so its dimension 
equals the virtual dimension, i.e., $(2g-2)\cdot \dim(G)$. 

\medskip

When $G$ has a connected center (but $g$ is still $>1$) the dimension
of the classical prestack underlying $\LS_\cG$ is $(2g-2)\cdot \dim(G)+\dim(Z_G)$. 

\end{rem}

\sssec{}

It is likely that the integer $d$ in Remark \ref{r:left ampl} is not the sharp bound. 
For example, when $G=T$ is a torus, the functor $\BL_{G,\on{coarse}}$ is left t-exact as-is. We propose: 

\begin{quest}
What is the actual bound on the left amplitude of $\BL_{G,\on{coarse}}$?
\end{quest} 

\begin{rem} 
Note that \thmref{t:L coarse left exact} does \emph{not} imply that the functor $\BL_G$ has 
a bounded cohomological amplitude on the left. In fact, this amplitude \emph{is} unbounded,
unless $G$ is a torus. 

\medskip

Namely, we claim the ``constant sheaf"" $\ul{k}_{\Bun_G}\in \Dmod_{\frac{1}{2}}(\Bun_G)$ is sent by $\BL_G$ to an \emph{infinitely connective}
object, i.e., an object that lies in $\IndCoh_\Nilp(\LS_G)^{\leq -n}$ for any $n$. 

\medskip

Indeed, the statement is equivalent to
$\Psi_{\Nilp,\{0\}}\circ \BL_G(\ul{k}_{\Bun_G})=0$, which is a valid statement, since $\BL_{G,\on{coarse}}(\ul{k}_{\Bun_G})$ is indeed zero. 

\end{rem} 

\ssec{Proof of \propref{p:compact bdd below}}

Before we launch the proof, let us explain what we are battling against: the stack $\Bun_G$ is not quasi-compact, and we have
to estimate the cohomological amplitude of the !-direct image functors for open embeddings $U\hookrightarrow U'$,
for a fixed quasi-compact open substack $U$ and variable $U'$. 

\medskip

The idea of the proof is that (since $\Bun_G$ has an affine diagonal), 
these amplitudes only depend on $U$ and not on $U'$. 

\sssec{}

The proof given below applies for the category $\Dmod(\CY)$, where $\CY$ is any \emph{truncatable}
algebraic stack (see \cite[Sect. 0.2.3]{DG1} for what this means), such that each of its quasi-compact open substacks 
is a global quotient\footnote{It is likely that the latter assumption is not needed, and it suffices to assume that $\CY$ 
has an affine diagonal.} (i.e., is of the form $Z/H$, where $Z$ is a (separated) scheme and $H$ as affine algebraic group). 

\medskip

The stack $\Bun_G$ is truncatable by \cite[Theorem 0.2.5]{DG1}. And it is a standard fact that each of its
quasi-compact open substacks is a global quotient. 

\sssec{}

First, we remark that the assertion is clear when $\CY$ is quasi-compact. Indeed, in this case
the category $\Dmod(\CY)$ is compactly generated by objects of the form
$$\ind_r(\CM), \quad \CM\in \Coh(\CY),$$
and such objects are bounded on both sides.

\medskip

In the above formula, 
$$\ind_r: \IndCoh(\CY)\to \Dmod(\CY)$$
is the left adjoint of 
$$\oblv_r:\Dmod(\CY)\to \IndCoh(\CY),$$
see \cite[Sect. 5.1.5]{DG2}.

\sssec{}

Let now $\CY$ be an arbitrary truncatable algebraic stack. Write $\CY$ as a union of \emph{co-truncative}
open substacks 
$$U\overset{j}\hookrightarrow \CY,$$
so that the functor 
$$j_!:\Dmod(U)\to \Dmod(\CY),$$
left adjoint to the restriction functor $j^!=j^*$ is defined. 

\medskip

Then compact generators of $\Dmod(\CY)$ are of the form
$j_!(\CF_U)$, where $\CF_U$ are compact generators of $\Dmod(U)$.

\sssec{}

Write $U=Z/H$, and let $f$ denote the map $Z\to U$. Consider the \emph{partially defined} 
left adjoint $f_!$ of the functor
$$f^!:\Dmod(U)\to \Dmod(Z).$$

As in \cite[Sect. F.3.5]{AGKRRV}, we can take as compact generators of $\Dmod(U)$ objects of the form
$$f_!(\CF_Z), \quad \CF_Z\in \Dmod(Z)^c,$$
where $\CF_Z$ is such that the partially defined functor $f_!$ \emph{is} defined on 
$\CF_Z$.

\medskip

It suffices to show that each of the objects 
$$j_!(f_!(\CF_Z)),$$
for $\CF_Z$ as above, is bounded below. 

\medskip

Suppose that $Z$ can be covered by $n$ affines. We will show that if $\CF_Z\in \Dmod(Z)^{\geq n}$, then $j_!(f_!(\CF_Z))\in \Dmod(\CY)^{\geq 0}$. 
The assertion is equivalent to the fact that for any quasi-compact open $U'\supset U$, we have
$$j_!(f_!(\CF_Z))|_{U'}\in  \Dmod(U')^{\geq 0}.$$

Since the specified bound does not depend on $U'$, we can assume that $\CY$ itself is quasi-compact. 

\sssec{}

Denote $\wt{f}:=j\circ f$. Note that $j_!(f_!(\CF_Z))$ is isomorphic to the value on $\CF_Z$ of the partially defined
left adjoint $\wt{f}_!$ of $\wt{f}^!$. 

\medskip

Let $g:Y'\to \CY$ be a smooth cover, where $Y'$ is an affine scheme. Set
$$Z':=Z\underset{\CY}\times Y'.$$

\medskip

Denote by $g_Z$ the map $Z'\to Z$. Let $\wt{f}'$ denote the map $Z'\to Y'$. 
Since $g$ is smooth, the functors
$$g_Z^*:\Dmod(Z)\to \Dmod(Z') \text{ and } g^*:\Dmod(\CY)\to \Dmod(Y')$$
are well defined. 

\medskip

We obtain that the partially defined left adjoint $\wt{f}'_!$
of $\wt{f}'{}^!$ is well-defined on $g_Z^*(\CF_Z)$, and we have
$$g^*(\wt{f}_!(\CF_Z) \simeq \wt{f}'_!(g_Z^*(\CF_Z)).$$

\medskip

Hence, it suffices to show that for any 
$$\CF_{Z'}\in \Dmod(Z')^c\cap \Dmod(Z')^{\geq n},$$
on which the partially defined functor $\wt{f}'_!$ is defined, we have
$$\wt{f}'_!(\CF_{Z'})\in \Dmod(Y')^{\geq 0}.$$

\sssec{}

Note that the functor $\wt{f}'_!$ is always defined as a functor
$$\Dmod(Z')^c\to \on{Pro}(\Dmod(Y')^c),$$
and it is enough to show that this functor, shifted by $[-n]$, is left t-exact. 

\medskip

Let $Z_i$ be the $n$ affine schemes that cover $Z$, and let $Z'_i$ denote their preimages
in $Z$. For a subset $\ul{i}$ of indices, let $j'_{\ul{i}}$ denote the open embedding
$$\underset{i\in \ul{i}}\cap\, Z'_i=:Z'_{\ul{i}}\hookrightarrow Z'.$$

\medskip

By Cousin, for $\CF\in \Dmod(Z')^c$, the object $\wt{f}'_!(\CF)\in \on{Pro}(\Dmod(Y')^c)$ 
admits a canonical filtration (of length $n$), whose $m$'s subquotient is
$$\underset{|\ul{i}|=m}\oplus\, (\wt{f}'\circ j'_{\ul{i}})_!(\CF|_{Z_{\ul{i}}})[m].$$

Hence, it is enough to show that each of the functors
$$(\wt{f}'\circ j'_{\ul{i}})_!:\Dmod(Z'_{\ul{i}})\to \on{Pro}(\Dmod(Y')^c)$$
is left t-exact.

\sssec{}

Note that $\CY$ has an affine diagonal. Hence, the morphism $Y'\to \CY$ is affine,
and hence so is the projection $Z'\to Z$. We obtain that each $Z'_{\ul{i}}$ is an affine schenme,
and hence so is the morphism $\wt{f}'\circ j'_{\ul{i}}$. 

\medskip

Now the desired left t-exactness assertion follows from \cite[Theorem 3.4.1]{Ra}. 

\qed[\propref{p:compact bdd below}]

\ssec{Proof of \thmref{t:L coarse left exact}}

The idea of the proof is to reduce \thmref{t:L coarse left exact} to its restricted version, namely, \thmref{t:L coarse left exact Nilp},
using a Cousin-type argument. The latter is encapsulated by \propref{p:when >0} below. 

\sssec{}

We have the following assertion, proved in \secref{ss:proof >0}: 

\begin{prop} \label{p:when >0}
Let $\CY$ be an eventually coconnective algebraic stack. Let 
$\CM$ be an object of $\QCoh(\CY)$ such that for every field extension $k\subset k'$ and a $k'$-point $y$ of $\CY$, we have:
$$i_y^!(\CM')\in \Vect_{k'}^{\geq 0},$$
where:

\begin{itemize}

\item $\CM'$ denotes the pullback of $\CM$ along $\CY\underset{\Spec(k)}\times \Spec(k')=:\CY'\to \CY$;

\item $i_y$ denotes the map $\Spec(k')\to \CY'$ corresponding to $y$;

\smallskip

\item $i_y^!$ denotes the (not necessarily continuous) right adjoint of $(i_y)_*$,

\end{itemize}

\smallskip

\noindent then $\CM\in \QCoh(\CY)^{\geq -d}$,
where $d$ is the dimension of the classical algebraic stack underlying $\CY$. 

\end{prop}

\begin{rem}
In practice, the fields $k'$ that will appear are fields of rational functions
on irreducible subschemes of an affine scheme $Y$ that smoothly covers $\CY$.
\end{rem}

\sssec{} \label{sss:red to rat}

Let us prove \thmref{t:L coarse left exact}.

\medskip

Let $d$ be the integer from \propref{p:when >0} for the stack $\LS_\cG$.  We will show that the functor
$\BL_{G,\on{coarse}}$ has an amplitude bounded on the left by $d-\dim(\Bun_G)+\dim(\Bun_{N,\rho(\omega_X)})$.

\medskip

Let $\CF$ be an object of $\Dmod_{\frac{1}{2}}(\Bun_G)$, which is cohomologically 
$\geq \dim(\Bun_{N,\rho(\omega_X)})-\dim(\Bun_G)$.
Applying \propref{p:when >0}, we need to show that for every algebraically
closed field extension $k\subset k'$ and a $k'$-point $\sigma$ of 
$$\LS'_\cG:=\LS_\cG\underset{\Spec(k)}\times \Spec(k'),$$
the object
$$i_\sigma^!(\BL_{G,\on{coarse}}(\CF)')\in \Vect_{k'}$$
belongs to $\Vect_{k'}^{\geq 0}$, where $\BL_{G,\on{coarse}}(\CF)'$ denotes the pullback of 
$\BL_{G,\on{coarse}}(\CF)$ to $\LS'_\cG$. 

\medskip

Base changing everything along $\Spec(k')\to \Spec(k)$, we can assume that $k'=k$,
so that $\sigma$ is a rational point of $\LS_\cG$. 

\sssec{}  \label{sss:restr}

Consider the sub prestack
$$\LS^{\on{restr}}_\cG \subset \LS_\cG,$$
see \cite[Sect. 4.1]{AGKRRV}.

\medskip

Let $$\QCoh(\LS_\cG)_{\on{restr}}\overset{\iota^{\on{spec}}} \hookrightarrow \QCoh(\LS_\cG)$$
be the (fully faithful) embedding of the subcategory consisting of objects in with set-theoretic support on $\LS^{\on{restr}}_\cG$. 
Let $(\iota^{\on{spec}})^R$ denote the right adjoint functor. 

\sssec{} \label{sss:red to restr}

Note that for a $k$-rational point $\sigma$ of $\LS_\cG$, the morphism $i_\sigma$ factors as 
$$\Spec(k) \to \LS^{\on{restr}}_\cG \to \LS_\cG,$$
where the second arrow is the map of \cite[Equation (4.2)]{AGKRRV}.

\medskip

Hence, the functor $(i_\sigma)_*$ factors via the subcategory $\QCoh(\LS_\cG)_{\on{restr}}$. Hence, the functor $\iota_\sigma^!$
factors as
$$(i_\sigma^!|_{\QCoh(\LS_\cG)_{\on{restr}}}) \circ (\iota^{\on{spec}})^R.$$

Since the functor $i_\sigma^!|_{\QCoh(\LS_\cG)_{\on{restr}}}$ is right t-exact, it suffices to show that
$$(\iota^{\on{spec}})^R(\BL_{G,\on{coarse}}(\CF))\in \Vect^{\geq 0}.$$

\sssec{} \label{sss:red to Nilp}

Let 
$$\Dmod_{\frac{1}{2},\Nilp}(\Bun_G) \overset{\iota}\hookrightarrow \Dmod_{\frac{1}{2}}(\Bun_G)$$
be the full subcategory consisting of D-modules with singular support in the global nilpotent cone.

\medskip

According to \cite[Proposition 14.5.3]{AGKRRV}, we have
\begin{equation} \label{e:Dmod nilp as ten}
\Dmod_{\frac{1}{2},\Nilp}(\Bun_G) \simeq \Dmod_{\frac{1}{2}}(\Bun_G)\underset{\QCoh(\LS^{\on{restr}}_\cG)}\otimes 
\QCoh(\LS_\cG)_{\on{restr}}.
\end{equation} 

Hence, the restriction of $\BL_{G,\on{coarse}}$ to $\Dmod_{\frac{1}{2},\Nilp}(\Bun_G)$ is a functor
$$\BL^{\on{restr}}_{G,\on{coarse}}: \Dmod_{\frac{1}{2},\Nilp}(\Bun_G) \to \QCoh(\LS_\cG)_{\on{restr}}.$$

Furthermore, we have,
$$\BL^{\on{restr}}_{G,\on{coarse}}\circ \iota^R\simeq (\iota^{\on{spec}})^R\circ \BL_{G,\on{coarse}}.$$

Since the functor $\iota$ is t-exact, the functor $\iota^R$ is left t-exact. Hence, it suffices to show that 
show that the functor $\BL^{\on{restr}}_{G,\on{coarse}}[\dim(\Bun_G)-\dim(\Bun_{N,\rho(\omega_X)})]$ is left t-exact. 

\medskip

However, this is the ``left t-exactness half" of the following result of \cite{FR} (see Sect. 1.6.2 in {\it loc. cit.}): 

\begin{thm} \label{t:L coarse left exact Nilp}
The functor
$$\BL^{\on{restr}}_{G,\on{coarse}}[\dim(\Bun_G)-\dim(\Bun_{N,\rho(\omega_X)})]:\Dmod_{\frac{1}{2},\Nilp}(\Bun_G) \to
\QCoh(\LS^{\on{restr}}_\cG).$$
is t-exact. 
\end{thm} 

\qed[\thmref{t:L coarse left exact}]

\begin{rem} \label{r:right ampl}
We emphasize that, unlike $\BL^{\on{restr}}_{G,\on{coarse}}$, the functor $\BL_{G,\on{coarse}}$ (shifted cohomologically by 
$[\dim(\Bun_G)-\dim(\Bun_{N,\rho(\omega_X)})]$), is \emph{not} right t-exact. One can see
this already when $G=\BG_m$, in which case $\BL_{G,\on{coarse}}$ is the Fourier-Laumon transform.

\medskip

Note, however, that \propref{p:bdd right} says that its cohomological amplitude on the right 
is bounded. 

\end{rem}

\ssec{Proof of \propref{p:when >0}} \label{ss:proof >0}

\sssec{}

Let $f:Y\to \CY$ be a smooth cover of $\CY$ by an affine scheme. Since
$$\CM\in \QCoh(\CY)^{\geq 0}\, \Leftrightarrow \, f^!(\CM)[-\dim(Y/\CY)]\in \QCoh(Y)^{\geq 0},$$
we can assume that $\CY=Y$ is an affine scheme.

\sssec{}

Let $Z\overset{i_Z}\hookrightarrow Y$ be an 
irreducible subvariety and let 
$$\eta_Z \overset{j_Z}\hookrightarrow Z$$
be the embedding of its generic point. 

\medskip

Let 
$$\QCoh(Y)_Z\overset{(\wh{i}_Z)_!}\hookrightarrow \QCoh(Y)$$
be the embedding of the full subcategory consisting of objects set-theoretically
supported on $Z$. Let $(\wh{i}_Z)^!$ denote the right adjoint of $(\wh{i}_Z)_!$.

\medskip

Let 
$$j_Z^*:\QCoh(Y)_Z\rightleftarrows \QCoh(Y)_{\eta_Z}:(j_Z)_*$$
be the localization of $\QCoh(Y)_Z$ at the generic point.

\medskip

Applying Cousin decomposition, we obtain that $\CM\in \QCoh(Y)^{\geq -d}$ if and
only if for every $Z$ as above,
\begin{equation} \label{e:estim Z hat}
j_Z^*\circ (\wh{i}_Z)^!(\CM)\in \QCoh(Y)^{\geq -d}_{\eta_Z}.
\end{equation} 

\sssec{}

We can represent the formal completion of $Z$ in $Y$ as
$$\underset{n}{\on{colim}}\, Z_n,$$
where $Z_n\overset{i_{Z_n}}\hookrightarrow Y$ are regularly embedded closed 
subschemes of $Y$ (see. e.g., \cite[Proposition 6.7.4]{GaRo1}).

\medskip

We have
$$(\wh{i}_Z)^!(\CM) \simeq \underset{n}{\on{colim}}\, (i_{Z_n})_*\circ (i_{Z_n})^*(\CM).$$

Hence, in order to establish \eqref{e:estim Z hat}, it suffices to show that for every $n$ above, we have
\begin{equation} \label{e:estim Zn}
 j_Z^*\circ (i_{Z_n})^!(\CM)\in \QCoh(\eta_{Z_n})^{\geq -d},
\end{equation} 
where $\eta_{Z_n}$ is the localization of $Z_n$ at $\eta_Z$. 

\sssec{}

Let $K$ be the field of rational functions on $Z$, and let $\eta_{Z_n}=\Spec(A_{Z_n})$. Note, however, that since $Y$ is eventually,
coconnective, $A_{Z_{n}}$ is an extension of finitely many copies of $K$ with non-positive shifts. 

\medskip

Hence, in order to establish \eqref{e:estim Zn}, it suffices to show that 
\begin{equation} \label{e:estim Z}
 j_Z^*\circ i_Z^!(\CM)\in \Vect_{K}^{\geq -d}.
\end{equation} 

\sssec{}

Denote
$$Y':=Y\underset{\Spec(k)}\times \Spec(K_Z),\,\, Z':=Z\underset{\Spec(k)}\times \Spec(K_Z);$$
let $i'_Z$ denote the embedding $Z'\to Y'$, obtained by base-changing $i_Z$. 
Let $\CM'$ denote the pullback of $\CM$ to $Y'$. 

\medskip

Note that $Z'$ has a canonical $K$-rational point $z_{\on{can}}$. Denote by
$i_{\on{can}}$ the corresponding morphism $\Spec(K) \to Z'$. 

\medskip

The morphism 
$j_Z$ factors as
$$\Spec(K) \overset{i_{\on{can}}}\longrightarrow Z' \to Z.$$

Hence, 
$$j_Z^*\circ i_Z^!(\CM)\simeq (i_{\on{can}})^*\circ (i'_Z)^!(\CM').$$

\sssec{}

Note, however, that $z_{\on{can}}$ is a smooth point of $Z'$. Hence,
$$(i_{\on{can}})^* \simeq (i_{\on{can}})^![\dim(Z)]$$
(up to a determinant line, which we ignore).

\medskip

Hence, we can further rewrite
$$j_Z^*\circ i_Z^!(\CM)\simeq (i_{\on{can}})^!\circ (i'_Z)^!(\CM')[\dim(Z)]=
(i'_Z\circ i_{\on{can}})^!(\CM)[\dim(Z)].$$

\sssec{}

Now, the 
condition of the proposition implies that
$$(i'_Z\circ i_{\on{can}})^!(\CM)\in \Vect_K^{\geq 0}.$$

Hence,
$$(i'_Z\circ i_{\on{can}})^!(\CM)[\dim(Z)]\in \Vect_K^{\geq -d},$$
as desired. 

\qed[\propref{p:when >0}]

\ssec{Proof of \propref{p:bdd right}} \label{ss:bdd right}

\sssec{}

We have the following counterpart of \propref{p:when >0} for right t-exactness:

\begin{prop} \label{p:when <0}
Let $\CY$ be an eventually coconnective algebraic stack. There exists an integer $d$ such that the
following holds: 

\medskip

If $\CM$ be an object of $\QCoh(\CY)$ such that for every field extension $k\subset k'$ and a $k'$-point $y$ of $\CY$, we have:
$$i_y^*(\CM')\in \Vect_{k'}^{\leq 0}.$$
Then $\CM\in \QCoh(\CY)^{\leq d}$. 
\end{prop}

We take $d$ to be the dimension of the classical scheme
underlying an affine smooth cover $Y\to \CY$. 

\medskip 

The proof is parallel to that of \propref{p:when >0}. The only difference is that we use the fact that 
$$j_Z^*\circ (i_{Z_n})^!\simeq j_Z^*\circ (i_{Z_n})^*[-\on{codim}(Z)].$$

\sssec{}

We proceed to the proof of \propref{p:bdd right}. As a main ingredient we will use 
the fact that the functor 
$$\iota^R:  \Dmod_{\frac{1}{2}}(\Bun_G) \to \Dmod_{\frac{1}{2},\Nilp}(\Bun_G)$$
has a bounded cohomological dimension on the right, see \cite[Corollary 14.5.5 and Proposition 17.3.10]{AGKRRV}. 
Let $d'$ denote this bound. 

\medskip

We will show that the functor $\BL_{G,\on{coarse}}$ has a cohomological amplitude bounded on the right by
$d+d'+\dim(\Bun_G)-\dim(\Bun_{N,\rho(\omega_X)})$, where $d$ is the integer from \propref{p:when <0} for the stack $\LS_\cG$. 

\sssec{}

Applying \propref{p:when <0} and arguing as in \secref{sss:red to rat}, it suffices to show that for a $k$-rational
point $\sigma$ of $\LS_\cG$, the functor
$$i_\sigma^*\circ \BL_{G,\on{coarse}}[d'+\dim(\Bun_G)-\dim(\Bun_{N,\rho(\omega_X)})]$$
is right t-exact.

\medskip

Note that the functor $i_\sigma^*$ also factors as
$$(i_\sigma^*|_{\QCoh(\LS_\cG)_{\on{restr}}})\circ (\iota^{\on{spec}})^R,$$
while the functor $i_\sigma^*|_{\QCoh(\LS_\cG)_{\on{restr}}}$ is right t-exact. 

\medskip

Hence, it is enough to show that the functor
$$(\iota^{\on{spec}})^R \circ \BL_{G,\on{coarse}}[d'+\dim(\Bun_G)-\dim(\Bun_{N,\rho(\omega_X)})]$$
is right t-exact. 

\sssec{}

We rewrite $(\iota^{\on{spec}})^R \circ \BL_{G,\on{coarse}}$ as 
$$\BL_{G,\on{course}}^{\on{restr}}\circ \iota^R.$$

Now, the desired assertion follows from the fact that $\iota^R[d']$ is right t-exact and 
the fact that $\BL^{\on{restr}}_{G,\on{coarse}}[\dim(\Bun_G)-\dim(\Bun_{N,\rho(\omega_X)})]$ is right t-exact
(this is the ``right t-exactness half" of \thmref{t:L coarse left exact Nilp}). 

\qed[\propref{p:bdd right}]

\section{Geometric Langlands functor in the Betti context} \label{s:Betti}

In this section we take our ground field $k$ to be $\BC$. Throughout this section we will  work with the ``big" category of Betti 
sheaves of $\sfe$-vector spaces (see \cite[Appendix G]{AGKRRV}), where $\sfe$ is an arbitrary field of coefficients 
(assumed of characteristic zero). 

\medskip

We will construct the Langlands functor in the Betti setting. 

\medskip

Once the functor is constructed, we will formulate a theorem to the effect that this functor is an equvalence
if and only if its de Rham counterpart is. 

\ssec{The category with nilpotent singular support}

In this subsection we will introduce (following \cite{BZN}) the automorphic category in the Betti context. 

\sssec{}

Note that since $\det^{\frac{1}{2}}_{\Bun_G}$ is a gerbe with respect to $\mu_2\simeq \BZ/2\BZ$, it gives rise to an \'etale $\sfe^\times$-gerbe
for any field $\sfe$ of characteristic different from 2.

\medskip

Hence, it makes sense to consider the category $\Shv^{\on{Betti}}_{\frac{1}{2}}(\Bun_G)$ of sheaves of $\sfe$-vector spaces in the classical
topology on $\Bun_G$, twisted by $\det^{\frac{1}{2}}_{\Bun_G}$.

\medskip

Note, however, that by \secref{sss:pfaff}, we can identify 
$$\Shv^{\on{Betti}}_{\frac{1}{2}}(\Bun_G)\simeq \Shv^{\on{Betti}}(\Bun_G).$$

\sssec{}

Let 
$$\Shv^{\on{Betti}}_{\frac{1}{2},\Nilp}(\Bun_G) \overset{\bi}\hookrightarrow \Shv^{\on{Betti}}_{\frac{1}{2}}(\Bun_G)$$
be the full subcategory consisting of sheaves with singular support contained in $\Nilp$.

\medskip

According to \cite[Sect. 18.2.6]{AGKRRV}, the functor $\bi$ admits a left adjoint, to be denoted $\bi^L$. 

\sssec{} \label{sss:betti comp gen}

The functor $\bi^L$ allows us to construct compact objects in $\Shv^{\on{Betti}}_{\frac{1}{2}}(\Bun_G)$.

\medskip

For a stack $\CY$ and a rational point $y$, let $\delta_y\in \Shv^{\on{Betti}}(\CY)$ be the corresponding
$\delta$-function object, i.e., $(i_y)_!(\sfe)$, where $i_y$ be the morphism
$$\on{pt}\to \CY$$
corresponding to $y$.

\medskip 

According to \cite[Proposition G.3.5]{AGKRRV}, objects of the form 
$$\bi^L(\delta_y), \quad y\in \Bun_G$$
form a set of compact generators of 
$\Shv^{\on{Betti}}_{\frac{1}{2},\Nilp}(\Bun_G)$.

\begin{rem}
Note, however, that objects $\delta_y\in \Shv^{\on{Betti}}_{\frac{1}{2}}(\Bun_G)$ are \emph{not}
compact, and that the category $\Shv^{\on{Betti}}_{\frac{1}{2}}(\Bun_G)$ itself is \emph{not} compactly generated.
\end{rem} 

\ssec{The Hecke action in the Betti context}

\sssec{}

Let $\Rep(\cG)^{\on{Betti}}_\Ran$ be the Betti version of the category $\Rep(\cG)_\Ran$, defined as in 
\cite[Remark 11.1.10]{AGKRRV}. 

\medskip

As in the de Rham setting, we have a localization functor
\begin{equation} \label{e:loc Betti}
\Loc_\cG^{\on{spec}}:\Rep(\cG)^{\on{Betti}}_\Ran\to \QCoh(\LS^{\on{Betti}}_\cG),
\end{equation} 
which admits a fully faithful continuous right adjoint.

\sssec{}

We have a canonically defined monoidal action of $\Rep(\cG)^{\on{Betti}}_\Ran$ on $\Shv^{\on{Betti}}_{\frac{1}{2}}(\Bun_G)$,
which preserves the subcategory
$\Shv^{\on{Betti}}_{\frac{1}{2},\Nilp}(\Bun_G)$.

\medskip

The following is a key observation from \cite{NY} (see also \cite[Theorem 18.1.4]{AGKRRV}):

\begin{prop}
The action of $\Rep(\cG)^{\on{Betti}}_\Ran$ on $\Shv^{\on{Betti}}_{\frac{1}{2}}(\Bun_G)$ 
factors via the localization \eqref{e:loc Betti}. 
\end{prop}

As a result, we obtain that the category $\Shv^{\on{Betti}}_{\frac{1}{2},\Nilp}(\Bun_G)$ carries
a canonically defined action of the monoidal category $\QCoh(\LS^{\on{Betti}}_\cG)$.

\ssec{The vacuum Poincar\'e sheaf in the Betti context} \label{ss:Poinc Betti}

Our current goal is to construct the object 
\begin{equation} \label{e:Poinc Betti}
\on{Poinc}^{\on{Vac,glob}}_{!,\Nilp}\in \Shv^{\on{Betti}}_{\frac{1}{2},\Nilp}(\Bun_G).
\end{equation}

The slight hiccup is that the exponential sheaf is not defined as an object of
$\Shv^{\on{Betti}}(\BG_a)$. In order to circumvent this, we will first apply the procedure of $\BG_m$-averaging.

\sssec{}

Define the object
$$\on{exp}/\BG_m\in \Shv^{\on{Betti}}(\BG_a/\BG_m)$$
to be the *-extension of 
$$\sfe[-1]\in \Vect=\Shv^{\on{Betti}}(\on{pt})$$ along the open embedding
$$\on{pt}\simeq (\BG_a-\{0\})/\BG_m\hookrightarrow \BG_a/\BG_m.$$

\begin{rem} \label{r:Kir}
Note that when $\sfe=\BC$, under the Riemann-Hilbert equivalence, the object $\on{exp}/\BG_m$
corresponds to the !-direct image of $\on{exp}$ along the projection
$$\BG_a\to \BG_a/\BG_m.$$
\end{rem}

\sssec{}

We will now consider a variant of this construction. Consider the stack 
$$\underset{i}\Pi\, \BG_a/T,$$
where $T$ acts on $\underset{i}\Pi\, \BG_a$ by means of
$$T\to T/Z_G\overset{\text{simple roots}}\simeq \underset{i}\Pi\, \BG_m.$$

Consider the open embedding
\begin{equation} \label{e:pt mod Z emb}
\on{pt}/Z_G \simeq \left(\underset{i}\Pi\, (\BG_a-0)\right)/T\hookrightarrow (\underset{i}\Pi\, \BG_a)/T.
\end{equation}

Let 
$$(\underset{i}\boxtimes\, \on{exp})/T\in \Shv^{\on{Betti}}\left((\underset{i}\Pi\, \BG_a)/T\right)$$
be the *-direct image along \eqref{e:pt mod Z emb} of
$$R_{Z_G}[-r]\in \Shv^{\on{Betti}}(\on{pt}/Z_G),$$ 
where:

\begin{itemize}

\item $R_{Z_G}$ is the !-direct image of $\sfe\in \Vect=\Shv^{\on{Betti}}(\on{pt})$
along $\on{pt}\to \on{pt}/Z_G$;

\item $r$ is the semi-simple rank of $G$.

\end{itemize} 

\begin{rem}  \label{r:exp r/T}

It follows from Remark \ref{r:Kir} and the projection formula that when $\sfe=\BC$, under the Riemann-Hilbert equivalence, the object  
$(\underset{i}\boxtimes\, \on{exp})/T$ is the !-direct image along the projection 
$$\underset{i}\Pi\, \BG_a\to (\underset{i}\Pi\, \BG_a)/T$$
of 
$$\on{add}^*(\on{exp})\simeq \underset{i}\boxtimes\, \on{exp}.$$

\end{rem} 

\sssec{} \label{sss:exp chi/T}

Note that the map 
$$\fp:\Bun_{N,\rho(\omega_X)}\to \Bun_G$$
factors as
$$\Bun_{N,\rho(\omega_X)}\to \Bun_{N,\rho(\omega_X)}/T \overset{\fp/T}\to \Bun_G.$$

By construction, the map $\Bun_{N,\rho(\omega_X)}\overset{\chi}\to \BG_a$ factors as
$$\Bun_{N,\rho(\omega_X)}\overset{\ul\chi}\to \underset{i}\Pi\, \BG_a\overset{\on{sum}}\to \BG_a.$$

\medskip

Consider the resulting map
$$\Bun_{N,\rho(\omega_X)}/T\overset{\ul\chi/T}\longrightarrow
(\underset{i}\Pi\, \BG_a)/T.$$

\medskip

Let 
$$\on{exp}_\chi/T\in \Shv^{\on{Betti}}(\Bun_{N,\rho(\omega_X)}/T)$$
be the pullback of the object $(\underset{i}\boxtimes\, \on{exp})/T$ along the map $\ul\chi/T$. 

\begin{rem} \label{r:exp chi/T}
It follows from Remark \ref{r:exp r/T} that when $\sfe=\BC$, under the Riemann-Hilbert equivalence, the object $\on{exp}_\chi/T$
corresponds to the !-direct image along the projection
$$\Bun_{N,\rho(\omega_X)}\to \Bun_{N,\rho(\omega_X)}/T$$
of the object
$$\on{exp}_\chi:=\chi^*(\on{exp})\in \Dmod(\Bun_{N,\rho(\omega_X)}).$$
\end{rem} 

\sssec{} 

Let 
$$\on{Poinc}^{\on{Vac,glob}}_!:=(\fp/T)_!(\on{exp}_\chi/T) \in \Shv^{\on{Betti}}_{\frac{1}{2}}(\Bun_G).$$

\medskip

Finally, set
$$\on{Poinc}^{\on{Vac,glob}}_{!,\Nilp}:=\bi^L(\on{Poinc}^{\on{Vac,glob}}_!).$$

\sssec{}

We claim:

\begin{prop} \label{p:Poinc Nilp comp}
The object $\on{Poinc}^{\on{Vac,glob}}_{!,\Nilp}\in \Shv^{\on{Betti}}_{\frac{1}{2},\Nilp}(\Bun_G)$
is compact.
\end{prop}

\begin{proof}

Note that the object $(\underset{i}\boxtimes\, \on{exp})/T$ is a finite extension of 
$\delta$-functions. 

\medskip

The map 
$$\ul\chi:\Bun_{N,\rho(\omega_X)}\to \underset{i}\Pi\, \BG_a)$$
is a unipotent gerbe. Hence, the map
$$\ul\chi/T:\Bun_{N,\rho(\omega_X)}/T\to (\underset{i}\Pi\, \BG_a)/T$$
has a similar property. 

\medskip

It follows that the object $\on{exp}_\chi/T\in  \Shv^{\on{Betti}}(\Bun_{N,\rho(\omega_X)}/T)$
is a finite extension of $\delta$-functions. Hence, so is $\on{Poinc}^{\on{Vac,glob}}_!$.

\medskip 

The assertion of the proposition follows now from \secref{sss:betti comp gen}.

\end{proof} 

\ssec{Construction of the functor}

Having the object $\on{Poinc}^{\on{Vac,glob}}_{!,\Nilp}$ at our disposal, the construction of the Langlands functor
in the Betti setting mimics its de Rham counterpart. 

\sssec{}

As in the de Rham context, we define the functor 
$$\BL^{\on{Betti},L}_{G,\on{temp}}:\QCoh(\LS^{\on{Betti}}_\cG)\to \Shv^{\on{Betti}}_{\frac{1}{2},\Nilp}(\Bun_G)$$
to be given by acting on the object
$$\on{Poinc}^{\on{Vac,glob}}_{!,\Nilp}\in \Shv^{\on{Betti}}_{\frac{1}{2},\Nilp}(\Bun_G).$$

\medskip

Since the object $\on{Poinc}^{\on{Vac,glob}}_{!,\Nilp}$ is compact and $\QCoh(\LS^{\on{Betti}}_\cG)$ is compactly generated and
rigid, the functor $\BL^{\on{Betti},L}_{G,\on{temp}}$ preserves compactness.

\sssec{}

Let 
$$\BL^{\on{Betti}}_{G,\on{coarse}}:\Shv^{\on{Betti}}_{\frac{1}{2},\Nilp}(\Bun_G)\to \QCoh(\LS^{\on{Betti}}_\cG)$$
denote the right adjoint of $\BL^{\on{Betti},L}_{G,\on{temp}}$. 

\medskip 

Since $\QCoh(\LS^{\on{Betti}}_\cG)$ is compactly generated and 
$\BL^{\on{Betti},L}_{G,\on{temp}}$ preserves compactness, the functor $\BL^{\on{Betti}}_{G,\on{coarse}}$ is continuous.
By rigidity, it is automatically $\QCoh(\LS^{\on{Betti}}_\cG)$-linear.

\sssec{}

We have the following assertion, which is parallel to \thmref{t:compact to bdd below dr}:

\begin{thm} \label{t:compact to bdd below Betti}
The functor
$$\BL^{\on{Betti}}_{G,\on{coarse}}:\Shv^{\on{Betti}}_{\frac{1}{2},\Nilp}(\Bun_G)\to \QCoh(\LS^{\on{Betti}}_\cG)$$
sends compact objects in $\Shv^{\on{Betti}}_{\frac{1}{2},\Nilp}(\Bun_G)$ to bounded below 
objects in $\QCoh(\LS^{\on{Betti}}_\cG)$.
\end{thm} 

The proof will be given in \secref{s:proof of Betti >0}. 

\sssec{}

Assuming \thmref{t:compact to bdd below Betti}, as in the de Rham context, we obtain that there exists a uniquely 
defined continuous functor
\begin{equation} \label{e:L Betti}
\BL^{\on{Betti}}_G:\Shv^{\on{Betti}}_{\frac{1}{2},\Nilp}(\Bun_G)\to \IndCoh_\Nilp(\LS^{\on{Betti}}_\cG),
\end{equation} 
characterized by the requirements that:

\begin{itemize}

\item The functor $\BL^{\on{Betti}}_G$ sends compact objects in $\Shv^{\on{Betti}}_{\frac{1}{2},\Nilp}(\Bun_G)$ to eventually coconnective
objects in $\IndCoh_\Nilp(\LS^{\on{Betti}}_\cG)$;

\item $\Psi_{\Nilp,\{0\}}\circ \BL_G\simeq \BL_{G,\on{coarse}}$.

\end{itemize} 

Moreover, the functor $\BL^{\on{Betti}}_G$ is automatically $\QCoh(\LS^{\on{Betti}}_\cG)$-linear.

\sssec{}

The geometric Langlands conjecture in the Betti context, originally formulated by D.~Ben-Zvi and D.~Nadler in \cite{BZN}, says:

\begin{conj} \label{c:GLC Betti}
The functor $\BL^{\on{Betti}}_G$ of \eqref{e:L Betti} is an equivalence.
\end{conj} 

\ssec{Comparsion of de Rham and Betti versions of GLC}

Although the Betti and de Rham versions of the automorphic category (i.e., 
$\Shv^{\on{Betti}}_{\frac{1}{2},\Nilp}(\Bun_G)$ and $\Dmod_{\frac{1}{2}}(\Bun_G)$) look very different, it turns 
out that the respective versions of GLC in the two contexts are logically equivalent. 

\medskip

The passage is realized by showing that
either version is equivalent to its \emph{restricted} variant, and the latter can be compared via Riemann-Hilbert. 

\medskip

The material in this subsection relies heavily on that of \cite{AGKRRV}. 

\sssec{}

We are going to prove:

\begin{thm} \label{t:dR => Betti}
The de Rham version of GLC implies the Betti version.
\end{thm}

In fact, we will prove a more precise version of \thmref{t:dR => Betti}, see
\thmref{t:dR => Betti bis} below. 

\sssec{} \label{sss:Betti constr Nilp as ten}

Let 
$$\Shv^{\on{Betti,constr}}_{\frac{1}{2},\Nilp}(\Bun_G)$$
be the ind-constructible category, as defined in \cite[Sect. E.5]{AGKRRV}.

\medskip

Recall that according to \cite[Proposition 18.3.2]{AGKRRV}, the forgetful functor
$$\Shv^{\on{Betti,constr}}_{\frac{1}{2},\Nilp}(\Bun_G)\to \Shv^{\on{Betti}}_{\frac{1}{2},\Nilp}(\Bun_G)$$
is fully faithful.

\medskip

Moreover, by \cite[Theorem 18.3.6]{AGKRRV}, the subcategory 
\begin{equation} \label{e:Betti constr Nilp as ten}
\Shv^{\on{Betti,constr}}_{\frac{1}{2},\Nilp}(\Bun_G)\subset \Shv^{\on{Betti}}_{\frac{1}{2},\Nilp}(\Bun_G)
\end{equation}
equals
$$\Shv^{\on{Betti}}_{\frac{1}{2},\Nilp}(\Bun_G)\underset{\QCoh(\LS_\cG)}\otimes \QCoh(\LS^{\on{Betti}}_\cG)_{\on{restr}}.$$

Note also that we have an equivalence
\begin{equation} \label{e:restr as ten prod Betti}
\IndCoh_\Nilp(\LS^{\on{Betti}}_\cG)\underset{\QCoh(\LS_\cG)}\otimes \QCoh(\LS^{\on{Betti}}_\cG)_{\on{restr}}\simeq
\IndCoh_\Nilp(\LS^{\on{Betti,restr}}_\cG).
\end{equation} 

Hence, the functor $\BL^{\on{Betti}}_G$ induces a functor, to be denoted
$$\BL_G^{\on{Betti,restr}}:\Shv^{\on{Betti,constr}}_{\frac{1}{2},\Nilp}(\Bun_G)\to \IndCoh_\Nilp(\LS^{\on{Betti,restr}}_\cG).$$

\sssec{}

Similarly, we have
\begin{equation} \label{e:restr as ten prod dR}
\IndCoh_\Nilp(\LS_\cG)\underset{\QCoh(\LS_\cG)}\otimes \QCoh(\LS_\cG)_{\on{restr}}\simeq
\IndCoh_\Nilp(\LS^{\on{restr}}_\cG).
\end{equation} 

Hence, by \eqref{e:Dmod nilp as ten}, the functor $\BL_G$ induces a functor
$$\BL_G^{\on{restr}}:\Dmod_{\frac{1}{2},\Nilp}(\Bun_G) \to \IndCoh_\Nilp(\LS^{\on{restr}}_\cG).$$

\sssec{}

We will prove: 

\begin{thm} \label{t:dR => Betti bis} 
The following statements are logically equivalent:

\smallskip

\noindent{\em(i)} The functor $\BL_G$ is an equivalence.

\smallskip

\noindent{\em(i')} The functor $\BL_G^{\on{restr}}$ is an equivalence;

\smallskip

\noindent{\em(ii)} The functor $\BL^{\on{Betti}}_G$ is an equivalence.

\smallskip

\noindent{\em(ii')} The functor $\BL_G^{\on{Betti,restr}}$ is an equivalence;

\end{thm} 

Note that the implications (i) $\Rightarrow$ (i') and (ii) $\Rightarrow$ (ii') in \thmref{t:dR => Betti bis}
are immediate from the equivalences \eqref{e:restr as ten prod dR} and \eqref{e:restr as ten prod Betti}, respectively. 

\medskip

The equivalence of (i') and (ii') will be proved in \secref{sss:Betti vs dR restr}.

\medskip

The implications (i') $\Rightarrow$ (i) and (ii') $\Rightarrow$ (ii) will be proved in
\secref{s:restricted vs full}. 

\section{Proof of \thmref{t:compact to bdd below Betti}}  \label{s:proof of Betti >0}

The proof of \thmref{t:compact to bdd below Betti} follows ideas, similar to its
de Rham counterpart.

\medskip

The main difference is that in the Betti setting, the embedding of the category with nilpotent 
singular support has a \emph{left} adjoint, while in the de Rham setting it admitted a right adjoint. 

\ssec{Reduction steps}

The idea of the proof is to reduce the assertion of \thmref{t:compact to bdd below Betti} to statements
that concern the restricted version of the functor $\BL_G^{\on{Betti}}$. Those are formulated as
Propositions \ref{p:delta funs} and \ref{p:L coarse left exact Nilp Betti} at the end of this subsection. 

\sssec{}

Recall that the left adjoint $\bi^L$ to the tautological embedding 
$$\bi:\Shv^{\on{Betti}}_{\frac{1}{2},\Nilp}(\Bun_G)\hookrightarrow \Shv^{\on{Betti}}_{\frac{1}{2}}(\Bun_G)$$
is well-defined. 

\medskip

Recall also (see \secref{sss:betti comp gen}) that objects of the form 
$$\bi^L(\delta_y), \quad y\in \Bun_G$$
form a set of compact generators of $\Shv^{\on{Betti}}_{\frac{1}{2},\Nilp}(\Bun_G)$.

\medskip

Thus, in order to prove \thmref{t:compact to bdd below Betti}, it suffices to show that the objects
$$\BL^{\on{Betti}}_{G,\on{coarse}}\circ \bi^L(\delta_y)\in \QCoh(\LS^{\on{Betti}}_\cG)$$
are bounded below. 

\sssec{}

Note that for an extension of fields of coefficients $\sfe\subset \sfe'$, the base change
$$\LS^{\on{Betti}}_\cG\underset{\Spec(\sfe)}\times \Spec(\sfe')$$
identifies with the stack of Betti local systems with $\sfe'$-coefficients.

\medskip

Hence, as in \secref{sss:red to rat}, it suffices to show that the objects
$$i_\sigma^! \left(\BL^{\on{Betti}}_{G,\on{coarse}}\circ \bi^L(\delta_y)\right)\in \Vect_\sfe$$
belong to $\Vect_\sfe^{\geq n}$, where $n$ is an integer that only depends on $y$
(but not on the choice of $\sfe$ or $\sigma$). 

\sssec{}

Let
$$\iota^{\on{Betti,spec}}:\QCoh(\LS^{\on{Betti}}_\cG)_{\on{restr}}\rightleftarrows \QCoh(\LS^{\on{Betti}}_\cG):(\iota^{\on{Betti,spec}})^R$$
be the corresponding pair of adjoint functors. 

\medskip

As in \secref{sss:red to restr}, it suffices to show that 
$$(\iota^{\on{Betti,spec}})^R\circ \BL^{\on{Betti}}_{G,\on{coarse}}\circ \bi^L(\delta_y)\in \Vect_\sfe^{\geq n},$$
where $n$ only depends on $y$. 

\sssec{}

Consider the embedding 
$$\Shv^{\on{Betti,constr}}_{\frac{1}{2},\Nilp}(\Bun_G)\overset{\iota^{\on{Betti}}}\hookrightarrow 
\Shv^{\on{Betti}}_{\frac{1}{2},\Nilp}(\Bun_G),$$
see \secref{sss:Betti constr Nilp as ten}. 

\medskip

By \eqref{e:Betti constr Nilp as ten}, this embedding admits a continuous right adjoint, to be denoted 
$(\iota^{\on{Betti}})^R$. Moreover, the functor $\BL^{\on{Betti}}_{G,\on{coarse}}$ induces a functor, to be denoted
$$\BL_{G,\on{coarse}}^{\on{Betti,restr}}:\Shv^{\on{Betti,constr}}_{\frac{1}{2},\Nilp}(\Bun_G)\to \IndCoh_\Nilp(\LS^{\on{Betti,restr}}_\cG),$$
and we have
$$(\iota^{\on{Betti,spec}})^R\circ \BL^{\on{Betti}}_{G,\on{coarse}} \simeq \BL_{G,\on{coarse}}^{\on{Betti,restr}}\circ (\iota^{\on{Betti}})^R.$$
 
Hence, it suffices to show that
$$\BL_{G,\on{coarse}}^{\on{Betti,restr}}\circ (\iota^{\on{Betti}})^R\circ \bi^L(\delta_y)\in \Vect_\sfe^{\geq n}.$$

\medskip

This is obtained by combining the following two assertions:

\begin{prop} \label{p:delta funs}
Objects of the form 
$$(\iota^{\on{Betti}})^R\circ \bi^L(\delta_y)\in \Shv^{\on{Betti,constr}}_{\frac{1}{2},\Nilp}(\Bun_G)$$
are bounded below. 
\end{prop}

\begin{prop} \label{p:L coarse left exact Nilp Betti}
The functor $\BL_{G,\on{coarse}}^{\on{Betti,restr}}[\dim(\Bun_G)-\dim(\Bun_{N,\rho(\omega_X)})]$ is t-exact.
\end{prop}

\begin{rem}
When it comes to t-exactness properties, there is a substantial difference between the de Rham and Betti
settings: one can show that the entire functor 
$$\BL_{G,\on{coarse}}^{\on{Betti}}[\dim(\Bun_G)-\dim(\Bun_{N,\rho(\omega_X)})]:
\Shv^{\on{Betti}}_{\frac{1}{2},\Nilp}(\Bun_G)\to \QCoh(\LS^{\on{Betti}}_\cG)$$
is t-exact.
\end{rem} 

\ssec{Applications of Riemann-Hilbert}

In this subsection we will exploit the fact that Riemann-Hilbert allows us to compare directly the categories
$\Dmod_{\frac{1}{2},\Nilp}(\Bun_G)$ and $\Shv^{\on{Betti,constr}}_{\frac{1}{2},\Nilp}(\Bun_G)$, 
and also the prestacks $\LS^{\on{restr}}_\cG$ and $\LS^{\on{Betti,restr}}_\cG$.

\sssec{Proof of \propref{p:L coarse left exact Nilp Betti}}

With no restriction of generality, we can assume that $\sfe=\BC$. 

\medskip

Consider the de Rham context with $k=\BC$. Recall that according to \cite[Corollary 16.5.6]{AGKRRV}, the embedding
$$\Dmod^{\on{RS}}_{\frac{1}{2},\Nilp}(\Bun_G) \subset \Dmod_{\frac{1}{2},\Nilp}(\Bun_G)$$ 
is an equality, where the superscript ``RS" denotes to the sheaf-theoretic context of holonomic D-modules
with regular singularities in the sense of \cite[Sect. 1.1.1]{AGKRRV}\footnote{Recall the 
formation of $\Dmod^{\on{RS}}(-)$ on a stack involves the operation of in-completion on affine schemes, and
the inverse limit over affine schemes mapping to the given stack.}. 

\medskip

Now, Riemann-Hilbert defines a t-exact equivalence
$$\Dmod^{\on{RS}}_{\frac{1}{2},\Nilp}(\Bun_G) \simeq \Shv^{\on{Betti,constr}}_{\frac{1}{2},\Nilp}(\Bun_G)$$
and an isomorphism of stacks
$$\LS^{\on{restr}}_\cG\simeq \LS^{\on{Betti,restr}}_\cG.$$

Indeed, the above stacks in either context are defined by Tannakian formalism, which takes as an input the 
corresponding symmetric monoidal category of local systems $\Lisse(X)$ (see \cite[Sect. 1.2]{AGKRRV}),
and these categories are identified by Riemann-Hilbert.

\medskip

We claim now that  \thmref{t:L coarse left exact Nilp} implies 
\propref{p:L coarse left exact Nilp Betti}. Indeed, follows from the next assertion: 

\begin{lem} 
The diagram 
\begin{equation} \label{e:RH diag coarse}
\CD
\Dmod^{\on{RS}}_{\frac{1}{2},\Nilp}(\Bun_G)  @>{\BL_{G,\on{coarse}}^{\on{restr}}}>>  \QCoh(\LS^{\on{restr}}_\cG) \\
@V{\sim}VV @VV{\sim}V \\
\Shv^{\on{Betti,constr}}_{\frac{1}{2},\Nilp}(\Bun_G) @>{\BL_{G,\on{coarse}}^{\on{Betti,restr}}}>> \QCoh(\LS^{\on{Betti,restr}}_\cG)
\endCD
\end{equation}
commutes. 
\end{lem} 

\begin{proof}

We only have to show that the Riemann-Hilbert equivalence sents
$$\iota(\on{Poinc}_!^{\on{Vac.glob}})\in \Dmod^{\on{RS}}_{\frac{1}{2},\Nilp}(\Bun_G) 
\text{ and } \on{Poinc}_{!,\Nilp}^{\on{Vac.glob}}\in \Shv^{\on{Betti,constr}}_{\frac{1}{2},\Nilp}(\Bun_G).$$

Let us show that the functors that there objects co-represent get identified. However, in both contexts,
the functor in question is 
$$\CF\mapsto \CHom(\on{exp}_\chi/T,(\fp/T)^!(\CF)),$$
see \secref{sss:exp chi/T} and Remark \ref{r:exp chi/T}
 
\end{proof} 

\qed[\propref{p:L coarse left exact Nilp Betti}]

\begin{rem}
There was no actual need to apply to Riemann-Hilbert in order to prove \propref{p:L coarse left exact Nilp Betti}:
one could repeat the argument in \cite{FR} verbatim in the Betti setting.
\end{rem} 

\sssec{} \label{sss:Betti vs dR restr}

In the rest of this subsection we will assume \thmref{t:compact to bdd below Betti}, so that the functor
$\BL_G^{\on{Betti}}$ is defined, and we will prove the equivalence of statements
(i') and (ii') in \thmref{t:dR => Betti bis}. 

\medskip 

By Lefschetz principle, we can 
assume that in (i') the ground field $k$ is $\BC$ and in (ii') the field $\sfe$ of coefficients is also
$\BC$.

\medskip

Applying Riemann-Hilbert, it suffices to show that the diagram 
\begin{equation} \label{e:RH diag}
\CD
\Dmod^{\on{RS}}_{\frac{1}{2},\Nilp}(\Bun_G)  @>{\BL_G^{\on{restr}}}>>  \IndCoh_\Nilp(\LS^{\on{restr}}_\cG) \\
@V{\sim}VV @VV{\sim}V \\
\Shv^{\on{Betti,constr}}_{\frac{1}{2},\Nilp}(\Bun_G) @>{\BL_G^{\on{Betti,restr}}}>> \IndCoh_\Nilp(\LS^{\on{Betti,restr}}_\cG)
\endCD
\end{equation}
commutes. 

\medskip

In order to prove the commutativity of \eqref{e:RH diag}, it suffices to show that the diagram
\begin{equation} \label{e:RH diag comp}
\CD
\Dmod^{\on{RS}}_{\frac{1}{2},\Nilp}(\Bun_G)^c  @>{\BL_G^{\on{restr}}}>>  \IndCoh_\Nilp(\LS^{\on{restr}}_\cG) \\
@V{\sim}VV @VV{\sim}V \\
\Shv^{\on{Betti,constr}}_{\frac{1}{2},\Nilp}(\Bun_G)^c @>{\BL_G^{\on{Betti,restr}}}>> \IndCoh_\Nilp(\LS^{\on{Betti,restr}}_\cG)
\endCD
\end{equation}
commutes. 

\sssec{}

Recall (see \cite[Theorem 16.1.1]{AGKRRV}) that the category 
$$\Dmod^{\on{RS}}_{\frac{1}{2},\Nilp}(\Bun_G) \simeq \Dmod_{\frac{1}{2},\Nilp}(\Bun_G)$$ 
is compactly generated generated, and since the right adjoint $\iota^R$ to the embedding 
$$\Dmod_{\frac{1}{2},\Nilp}(\Bun_G)\overset{\iota}\hookrightarrow \Dmod_{\frac{1}{2}}(\Bun_G)$$
is continuous, compact objects of $\Dmod^{\on{RS}}_{\frac{1}{2},\Nilp}(\Bun_G)$ are compact as 
objects of $\Dmod_{\frac{1}{2}}(\Bun_G)$.

\medskip

In particular, by
\propref{p:compact bdd below}, we obtain that the compact generators of  
$\Dmod^{\on{RS}}_{\frac{1}{2},\Nilp}(\Bun_G)$ are bounded below.  By \thmref{t:compact to bdd below dr}
we obtain that the top horizontal arrow in \eqref{e:RH diag comp} takes values in 
$$\IndCoh_\Nilp(\LS^{\on{restr}}_\cG)^{>-\infty} \subset \IndCoh_\Nilp(\LS^{\on{restr}}_\cG).$$

\medskip

Recall (see \cite[Sect. 18.3.2]{AGKRRV}) that the embedding $\iota^{\on{Betti}}$ preserves compactness.
By \thmref{t:compact to bdd below Betti} we obtain that the bottom horizontal arrow in \eqref{e:RH diag comp} takes values in 
$$\IndCoh_\Nilp(\LS^{\on{Betti,restr}}_\cG)^{>-\infty}\subset \IndCoh_\Nilp(\LS^{\on{Betti,restr}}_\cG).$$

\medskip

Hence, we can replace \eqref{e:RH diag comp} by 
$$
\CD
\Dmod^{\on{RS}}_{\frac{1}{2},\Nilp}(\Bun_G)^c  @>{\BL_G^{\on{restr}}}>>  \IndCoh_\Nilp(\LS^{\on{restr}}_\cG)^{>-\infty}  \\
@V{\sim}VV @VV{\sim}V \\
\Shv^{\on{Betti,constr}}_{\frac{1}{2},\Nilp}(\Bun_G)^c @>{\BL_G^{\on{Betti,restr}}}>> \IndCoh_\Nilp(\LS^{\on{Betti,restr}}_\cG)^{>-\infty}. 
\endCD
$$

\medskip

However, the commutativity of the latter diagram follows formally from the commutativity of \eqref{e:RH diag coarse}.

\qed[Equivalence of (i') and (ii')]

\ssec{Proof of \propref{p:delta funs}}

The idea of the proof is to express the functor $(\iota^{\on{Betti}})^R\circ \bi^L$ via a particular Hecke functor, 
known as the \emph{Beilinson projector}. 

\medskip

This will allow us to replace the left adjoint $\bi^L$ by a right adjoint
for some other functor, which has an evident left-exactness property, 

\sssec{}

Let $\sP_{\LS^{\on{Betti,restr}}_\cG}$ be the Beilinson projector from \cite[Sect. 15.4.5]{AGKRRV},
viewed as an endofunctor of $\Shv^{\on{Betti}}_{\frac{1}{2}}(\Bun_G)$. 

\medskip

We claim:

\begin{lem} \label{l:Beil}
The functor $(\iota^{\on{Betti}})^R\circ \bi^L$, followed by
$$\Shv^{\on{Betti,constr}}_{\frac{1}{2},\Nilp}(\Bun_G)\overset{\iota^{\on{Betti}}}\hookrightarrow 
\Shv^{\on{Betti}}_{\frac{1}{2},\Nilp}(\Bun_G)\overset{\bi}\hookrightarrow \Shv^{\on{Betti}}_{\frac{1}{2}}(\Bun_G)$$
identifies canonically with $\sP_{\LS^{\on{Betti,restr}}_\cG}$.
\end{lem}

Let us assume this lemma and prove \propref{p:delta funs}. 

\sssec{}

Denote by $\bi^{\on{constr}}$ the embedding
$$\Shv^{\on{Betti,constr}}_{\frac{1}{2},\Nilp}(\Bun_G)\hookrightarrow \Shv^{\on{Betti,constr}}_{\frac{1}{2}}(\Bun_G).$$

By \cite[Theorem 16.4.10 and Proposition 17.2.3]{AGKRRV}, the functor $\bi^{\on{constr}}$ admits a continuous right
adjoint. Moreover, the comonad $\bi^{\on{constr}}\circ (\bi^{\on{constr}})^R$ is given by the Beilinson projector 
$\sP_{\LS^{\on{Betti,restr}}_\cG}$, viewed as an endofunctor of $\Shv^{\on{Betti,constr}}_{\frac{1}{2}}(\Bun_G)$. 

\medskip

Note also that we have a commutative diagram
$$
\CD
\Shv^{\on{Betti,constr}}_{\frac{1}{2}}(\Bun_G) @>{\oblv^{\on{constr}}}>> \Shv^{\on{Betti}}_{\frac{1}{2}}(\Bun_G) \\
@V{\sP_{\LS^{\on{Betti,restr}}_\cG}}VV @VV{\sP_{\LS^{\on{Betti,restr}}_\cG}}V \\
\Shv^{\on{Betti,constr}}_{\frac{1}{2}}(\Bun_G) @>{\oblv^{\on{constr}}}>> \Shv^{\on{Betti}}_{\frac{1}{2}}(\Bun_G), 
\endCD
$$
in which the horizontal arrows are the tautological forgetful functor. 

\begin{rem}
Note that the functor 
\begin{equation} \label{e:forget constr}
\oblv^{\on{constr}}:\Shv^{\on{Betti,constr}}_{\frac{1}{2}}(\Bun_G) \to \Shv^{\on{Betti}}_{\frac{1}{2}}(\Bun_G)
\end{equation} 
is not fully faithful. However, the composition
$$\Shv^{\on{Betti,constr}}_{\frac{1}{2},\Nilp}(\Bun_G) \overset{\bi^{\on{constr}}}\hookrightarrow 
\Shv^{\on{Betti,constr}}_{\frac{1}{2}}(\Bun_G) \overset{\oblv^{\on{constr}}}\longrightarrow  \Shv^{\on{Betti}}_{\frac{1}{2}}(\Bun_G)$$
is fully faithful, since it can be rewritten as
$$\Shv^{\on{Betti,constr}}_{\frac{1}{2},\Nilp}(\Bun_G) \overset{\iota^{\on{Betti}}}\hookrightarrow 
\Shv^{\on{Betti}}_{\frac{1}{2},\Nilp}(\Bun_G)\overset{\bi}\hookrightarrow \Shv^{\on{Betti}}_{\frac{1}{2}}(\Bun_G).$$
\end{rem}

\sssec{}

Applying \lemref{l:Beil}, we obtain that we have a canonical
isomorphism 
$$(\iota^{\on{Betti}})^R\circ \bi^L\circ \oblv^{\on{constr}}\simeq (\bi^{\on{constr}})^R$$
as functors
$$\Shv^{\on{Betti,constr}}_{\frac{1}{2}}(\Bun_G)\to \Shv^{\on{Betti,constr}}_{\frac{1}{2},\Nilp}(\Bun_G).$$

\medskip

In particular, since the functor $(\bi^{\on{constr}})^R$ is left t-exact, we obtain that if $\CF\in \Shv^{\on{Betti,constr}}_{\frac{1}{2}}(\Bun_G)$ 
is eventually coconnective, then so is $(\iota^{\on{Betti}})^R\circ \bi^L\circ \oblv^{\on{constr}}(\CF)$.

\sssec{}

Hence, it remains to show that the objects
$$\delta_y\in \Shv^{\on{Betti,constr}}_{\frac{1}{2}}(\Bun_G)$$
are eventually coconnective.

\medskip

This follows by Riemann-Hilbert from \propref{p:compact bdd below} (or can be reproved directly within 
$\Shv^{\on{Betti,constr}}_{\frac{1}{2}}(\Bun_G)$ by the same argument).

\qed[\propref{p:delta funs}]

\ssec{Proof of \lemref{l:Beil}}

The idea of the proof is that all functors in sight are given by the various versions of the Beilinson projector. 

\sssec{}

Recall (see \cite[Corollary 18.2.9(a)]{AGKRRV}) that the functor $\bi\circ \bi^L$ is given by 
the action of a \emph{Hecke functor}, denoted $\sP^{\on{enh}}_{\LS_\cG^{\on{Betti}}}$.

\medskip

Similarly, the functor $\iota^{\on{Betti}}\circ (\iota^{\on{Betti}})^R$ is given by the restriction to 
$$\Shv^{\on{Betti}}_{\frac{1}{2},\Nilp}(\Bun_G)\subset \Shv^{\on{Betti}}_{\frac{1}{2}}(\Bun_G)$$
of the Hecke functor $\sP_{\LS_\cG^{\on{Betti,restr}}}$. 

\medskip

Now, any two Hecke functors commute, so we have
$$\bi\circ \iota^{\on{Betti}}\circ 
(\iota^{\on{Betti}})^R\circ \bi^L\simeq \sP_{\LS_\cG^{\on{Betti,restr}}}\circ \sP_{\LS_\cG^{\on{Betti}}} \simeq 
\sP^{\on{enh}}_{\LS_\cG^{\on{Betti}}}\circ \sP_{\LS_\cG^{\on{Betti,restr}}}.$$

\sssec{}

Hence, in order to prove \lemref{l:Beil}, it remains to show that 
$$\sP^{\on{enh}}_{\LS_\cG^{\on{Betti}}}\circ \sP_{\LS_\cG^{\on{Betti,restr}}}\simeq \sP_{\LS_\cG^{\on{Betti,restr}}}.$$

I.e., we need to show that the functor $\sP^{\on{enh}}_{\LS_\cG^{\on{Betti}}}$ acts as identity on objects lying
in the essential image of the functor $\sP_{\LS_\cG^{\on{Betti,restr}}}$.

\medskip

By \cite[Corollary 18.2.9(a)]{AGKRRV}, it suffices to show that the essential image of the endofunctor 
$\sP_{\LS_\cG^{\on{Betti,restr}}}$ is contained in $\Shv^{\on{Betti}}_{\frac{1}{2},\Nilp}(\Bun_G)$.

\sssec{}

Note that the essential image of the functor $\sP_{\LS_\cG^{\on{Betti,restr}}}$ lies in 
$$\Shv^{\on{Betti}}_{\frac{1}{2}}(\Bun_G)^{\on{Hecke-fin.mon.}}\subset \Shv^{\on{Betti}}_{\frac{1}{2}}(\Bun_G)$$
(see \cite[Sect. 18.3]{AGKRRV}). 

\medskip

In particular, this essential image lies in 
$$\Shv^{\on{Betti}}_{\frac{1}{2}}(\Bun_G)^{\on{Hecke-loc.const.}}\subset 
\Shv^{\on{Betti}}_{\frac{1}{2}}(\Bun_G).$$

Hence, by \cite[Theorem 18.1.6]{AGKRRV}, it belongs to $\Shv^{\on{Betti}}_{\frac{1}{2},\Nilp}(\Bun_G)$, as required. 

\qed[\lemref{l:Beil}]

\section{Restricted vs full GLC} \label{s:restricted vs full}

In this section we will be prove the implications (i') $\Rightarrow$ (i) and (ii') $\Rightarrow$ (ii) 
in \thmref{t:dR => Betti bis}. 

\medskip

While doing so, we will encounter yet another version of GLC
(in both de Rham and Betti settings), namely, the \emph{tempered} GLC. Ultimately, the logic of the proof 
(in either context) will be:
$$\text{Full GLC}\, \Leftrightarrow \text{Full tempered GLC}\, \Leftrightarrow \text{Restricted tempered GLC}\, 
\Leftrightarrow \text{Restricted GLC}.$$

\ssec{The de Rham context}

In this subsection we will introduce the tempered version of the Langlands functor. We will show that the 
full tempered GLC is equivalent to its restricted version. 

\begin{rem}

The implication (i') $\Rightarrow$ (i) was proved in \cite[Sect. 21.4]{AGKRRV}, under the assumption that the functor
$\BL_G$ admits a left adjoint.  

\medskip

We will prove that this is the case in one of the subsequent papers in this series\footnote{In fact, we will ultimately prove
that $\BL_G$ is an equivalence, so that (i) holds unconditionally.}. In this subsection we will give a different proof 
of this implication, which can then be adapted to the de Betti context.  

\end{rem}

%
%
%
%
%

\sssec{} \label{sss:temp}

Let 
$$\Dmod_{\frac{1}{2}}(\Bun_G)_{\on{temp}}\overset{\bu}\hookrightarrow \Dmod_{\frac{1}{2}}(\Bun_G)$$
denote the \emph{tempered} subcategory, defined as in \cite[Sect. 17.8.4]{AG}. 

\medskip

It follows easily from the definition that the above embedding admits a right adjoint, to be denoted $\bu^R$. 
Thus, we obtain an adjunction
$$\bu:\Dmod_{\frac{1}{2}}(\Bun_G)_{\on{temp}}\rightleftarrows \Dmod_{\frac{1}{2}}(\Bun_G):\bu^R,$$
and for most practical purposes, it is convenient to view $\Dmod_{\frac{1}{2}}(\Bun_G)_{\on{temp}}$
is a localization of $\Dmod_{\frac{1}{2}}(\Bun_G)$. 

\begin{rem}
The definition of $\Dmod_{\frac{1}{2}}(\Bun_G)_{\on{temp}}$ in {\it loc. cit.} depended on the choice of a point $x\in X$,
and independence of the point was conjectural. This conjecture was established by Beraldo (unpublished) and later
by Faegerman-Raskin by a different method (see \cite[Sect. 2.6.2]{FR}). 

\medskip

However, for the purposes of this paper, one can work with $\Dmod_{\frac{1}{2}}(\Bun_G)_{\on{temp}}$ without
knowing that it is point-independent. 

\end{rem} 

\sssec{} \label{sss:LL temp}

We claim that the object $\on{Poinc}_!^{\on{Vac,glob}}$ belongs to $\Dmod_{\frac{1}{2}}(\Bun_G)_{\on{temp}}$. This can be shown 
by re-interpreting $\on{Poinc}_!^{\on{Vac,glob}}$ via a local-to-global 
construction\footnote{Namely, $\on{Poinc}_!^{\on{Vac,glob}}$ is the image of the (vacuum) object in the \emph{local Whittaker category}
$\Whit(G)_x$ under the Poincar\'e functor $\on{Poinc}_{!.x}:\Whit(G)_x\to \Dmod_{\frac{1}{2}}(\Bun_G)$, while all of $\Whit(G)_x$
is tempered.}. This will be discussed in detail in the sequel to this paper.   

\medskip

Since the Hecke action preserves the tempered subcategory, we obtain that the essential image of the functor $\BL^L_{G,\on{temp}}$ is 
contained in $\Dmod_{\frac{1}{2}}(\Bun_G)_{\on{temp}}$.

\medskip

By adjunction, we obtain that the functor $\BL_{G,\on{coarse}}$
factors as
$$\Dmod_{\frac{1}{2}}(\Bun_G) \overset{\bu^R}\twoheadrightarrow \Dmod_{\frac{1}{2}}(\Bun_G)_{\on{temp}}
\overset{\BL_{G,\on{temp}}}\to \QCoh(\LS_G)$$
for a uniquely defined functor
$$\Dmod_{\frac{1}{2}}(\Bun_G)_{\on{temp}}
\overset{\BL_{G,\on{temp}}}\to \QCoh(\LS_G).$$

\medskip

Since
$$\Psi_{\Nilp,\{0\}}\circ \BL_G\simeq \BL_{G,\on{coarse}},$$
we have a commutative diagram
\begin{equation} \label{e:temp diag}
\CD
\Dmod_{\frac{1}{2}}(\Bun_G)  @>{\BL_G}>> \IndCoh_\Nilp(\LS_\cG) \\
@V{\bu^R}VV @VV{\Psi_{\Nilp,\{0\}}}V \\
\Dmod_{\frac{1}{2}}(\Bun_G)_{\on{temp}} @>{\BL_{G,\on{temp}}}>> \QCoh(\LS_\cG),
\endCD
\end{equation}
where the functor $\BL_{G,\on{temp}}$ is the right adjoint of the functor $\BL^L_{G,\on{temp}}$
of \eqref{e:LL temp}.

%

\sssec{}

Set
$$\Dmod_{\frac{1}{2},\Nilp}(\Bun_G)_{\on{temp}}:=\Dmod_{\frac{1}{2},\Nilp}(\Bun_G)\cap \Dmod_{\frac{1}{2}}(\Bun_G)_{\on{temp}}\subset
\Dmod_{\frac{1}{2}}(\Bun_G).$$

It follows from \eqref{e:Dmod nilp as ten} that
$$\Dmod_{\frac{1}{2},\Nilp}(\Bun_G)_{\on{temp}}=
\Dmod_{\frac{1}{2}}(\Bun_G)_{\on{temp}}\underset{\QCoh(\LS_\cG)}\otimes \QCoh(\LS_\cG)_{\on{restr}},$$
as subcategories of $\Dmod_{\frac{1}{2}}(\Bun_G)_{\on{temp}}$.

\medskip

It is easy to see that the functor $\bu^R$ sends
$$\Dmod_{\frac{1}{2},\Nilp}(\Bun_G)\to \Dmod_{\frac{1}{2},\Nilp}(\Bun_G)_{\on{temp}}.$$

\medskip

Applying the functor
$$-\underset{\QCoh(\LS_\cG)}\otimes \QCoh(\LS_\cG)_{\on{restr}},$$
from \eqref{e:temp diag}, we obtain a commutative diagram 
\begin{equation} \label{e:temp diag restr}
\CD
\Dmod_{\frac{1}{2},\Nilp}(\Bun_G)  @>{\BL^{\on{restr}}_G}>> \IndCoh_\Nilp(\LS^{\on{restr}}_\cG) \\
@V{\bu^R}VV @VV{\Psi_{\Nilp,\{0\}}}V \\
\Dmod_{\frac{1}{2},\Nilp}(\Bun_G)_{\on{temp}} @>{\BL^{\on{restr}}_{G,\on{temp}}}>> \QCoh(\LS^{\on{restr}}_\cG). 
\endCD
\end{equation}

\sssec{}

Suppose that the functor $\BL^{\on{restr}}_G$ is an equivalence. We claim that this implies that 
$\BL^{\on{restr}}_{G,\on{temp}}$ is also an equivalence. 

\medskip

Indeed, the fact that $\BL^{\on{restr}}_G$ is an equivalence implies that $\BL^{\on{restr}}_{G,\on{temp}}$ 
is a Verdier quotient. However, we also know that the functor $\BL^{\on{restr}}_{G,\on{temp}}$ is conservative: 
this is the main result of \cite{FR}. 

\begin{rem} \label{r:L temp}
One can avoid appealing to \cite{FR} for the proof of ``$\BL^{\on{restr}}_G$ is an equivalence" $\Rightarrow$ 
``$\BL^{\on{restr}}_{G,\on{temp}}$ is an equivalence".

\medskip

Indeed, it is not difficult to show that the functor $\BL_G$ respects the action of the full Satake category (at a given 
point $x\in X$) on the two sides\footnote{This will be elaborated on in the subsequent paper in this series.}. 
Now,
$$\Dmod_{\frac{1}{2}}(\Bun_G)_{\on{temp}} \simeq \Dmod_{\frac{1}{2}}(\Bun_G)\underset{\Sat_G}\otimes \Sat_{G,\on{temp}}$$
and
$$\QCoh(\LS_\cG) \simeq \IndCoh_\Nilp(\LS_\cG) \underset{\Sat_G}\otimes \Sat_{G,\on{temp}},$$
and also
$$\Dmod_{\frac{1}{2},\Nilp}(\Bun_G)_{\on{temp}} \simeq \Dmod_{\frac{1}{2},\Nilp}(\Bun_G)\underset{\Sat_G}\otimes \Sat_{G,\on{temp}}$$
and
$$\QCoh(\LS^{\on{restr}}_\cG) \simeq \IndCoh_\Nilp(\LS^{\on{restr}}_\cG) \underset{\Sat_G}\otimes \Sat_{G,\on{temp}}.$$

By a similar logic, if $\BL_G$ is an equivalence, then it formally follows that $\BL_{G,\on{temp}}$ is also an equivalence. 

\end{rem} 

\begin{rem}
In \secref{ss:temp to all} we will prove an inverse implication (which is far less trivial):
the fact that $\BL_{G,\on{temp}}$ is an equivalence implies that
$\BL_G$ is an equivalence.

\medskip

The same argument will show that if $\BL^{\on{restr}}_{G,\on{temp}}$ is an equivalence then
$\BL^{\on{restr}}_G$ is an equivalence.

\end{rem}

\sssec{}

We now claim that \emph{if} $\BL^{\on{restr}}_{G,\on{temp}}$ is an equivalence (for all extensions of the ground field
$k\subset k'$), \emph{then} $\BL_{G,\on{temp}}$ is an equivalence. 

\medskip

This follows by repeating verbatim the argument of \cite[Sect. 21.4]{AGKRRV}. 

\begin{rem}
Both deductions:
$$\text{Full tempered GLC}\, \Rightarrow \text{Restricted tempered GLC},$$
carried out above, and 
$$\text{Full GLC}\, \Rightarrow \text{Restricted GLC},$$
carried out in \cite[Sect. 21.4]{AGKRRV}, 
follow the same logic, and both need that the corresponding version of the Langlands functor admit an adjoint. 
Now, the latter is immediate from the construction for $\BL_{G,\on{temp}}$, but requires more work for the original
Langlands functor $\BL_G$. 
\end{rem}

\sssec{} \label{sss:red to temp}

Thus, we obtain that in order to prove the implication (i') $\Rightarrow$ (i) in \thmref{t:dR => Betti bis},
it suffices to show \emph{if} $\BL_{G,\on{temp}}$ is an equivalence \emph{then} $\BL_G$ is an equivalence.

\medskip

This will be done on the next section. 

\ssec{Full vs tempered Langlands} \label{ss:temp to all} 

In this subsection we will assume that $\BL_{G,\on{temp}}$ is an equivalence and deduce that
$\BL_G$ is an equivalence.
 
\medskip

The idea of the proof is that the behavior of $\BL_{G,\on{temp}}$ recovers the behavior of $\BL_G$ 
on \emph{compact objects}. 

\sssec{}

Recall that the category $\IndCoh_\Nilp(\LS^{\on{restr}}_\cG)$ is compactly generated (see \cite[Sect. 11.1.6]{AG});
its subcategory of compact objects is $\Coh_\Nilp(\LS^{\on{restr}}_\cG)$, the category of \emph{coherent} sheaves
with nilpotent singular support. 

\medskip 

Note that the functor 
$$\Psi:\IndCoh(\LS_\cG)\to \QCoh(\LS_\cG)$$
is fully faithful, when restricted to the subcategory of compact objects.
Indeed, this restriction is the natural embedding 
$$\Coh(\LS_\cG)\hookrightarrow \QCoh(\LS_\cG)$$.

\medskip

In particular, we obtain that the functor $\Psi_{\Nilp,\{0\}}$ is also fully faithful when restricted
to 
$$\Coh_\Nilp(\LS^{\on{restr}}_\cG)=\IndCoh_\Nilp(\LS^{\on{restr}}_\cG)^c.$$

\sssec{}

The proof of the desired implication is a formal consequence of the combination of the following two assertions:

\begin{prop} \label{p:uR ff on compacts}
The functor 
$$\bu^R|_{\Dmod_{\frac{1}{2}}(\Bun_G)^c}:\Dmod_{\frac{1}{2}}(\Bun_G)^c\to \Dmod_{\frac{1}{2}}(\Bun_G)_{\on{temp}}$$
is fully faithful.
\end{prop}

\begin{prop} \label{p:images}
Assume that $\BL_{G,\on{temp}}$ is an equivalence (for the group $G$ and all its Levi subgroups). 
Then the essential image of $\Dmod_{\frac{1}{2}}(\Bun_G)^c$ in $\QCoh(\LS_\cG)$ of $\Dmod_{\frac{1}{2}}(\Bun_G)^c$ under 
$$\BL_{G,\on{temp}}\circ \bu^R\simeq \BL_{G,\on{coarse}}$$ 
equals $\Coh_\Nilp(\LS_\cG)$.
\end{prop}

\ssec{Proof of \propref{p:uR ff on compacts}}

The proof will be obtained by combining the following two ingredients. One is the \emph{miraculous functor} on
$\Bun_G$. The other is a result of \cite{Be1}, which says that objects $\Dmod_{\frac{1}{2}}(\Bun_G)$ that are *-extended
from quasi-compact open substacks are tempered. 

\sssec{}

For an object $\CF\in \Dmod_{\frac{1}{2}}(\Bun_G)$, let
$$\CF_{\on{temp}}\to \CF\to \CF_{\on{anti-temp}}$$
be the fiber sequence associated with 
$$\Dmod_{\frac{1}{2}}(\Bun_G)_{\on{temp}}\hookrightarrow \Dmod_{\frac{1}{2}}(\Bun_G).$$

Explicitly,
$$\CF\simeq \bu\circ \bu^R(\CF).$$

\medskip

We have to show that for a pair of compact object $\CF_1,\CF_2\in \Dmod_{\frac{1}{2}}(\Bun_G)$, the map
\begin{equation} \label{e:comp temp}
\Hom_{\Dmod_{\frac{1}{2}}(\Bun_G)}(\CF_1,\CF_2)\to \Hom_{\Dmod_{\frac{1}{2}}(\Bun_G)}(\CF_{1,\on{temp}},\CF_{2,\on{temp}})
\end{equation} 
is an isomorphism. As we shall see, just the assumption that $\CF_2$ be compact will suffice. 

\medskip

The map \eqref{e:comp temp} is tautologically an isomorphism if $\CF_1$ is tempered (for any $\CF_2$). Hence, it suffices to show that 
if $\CF_2$ is compact and $\CF_1$ is \emph{anti-tempered}, i.e., if $\CF_{1,\on{temp}}=0$ or, equivalently, if 
$$\CF_1\to \CF_{1,\on{anti-temp}}$$
is an isomorphism, then
$$\Hom_{\Dmod_{\frac{1}{2}}(\Bun_G)}(\CF_1,\CF_2)=0.$$

\sssec{}

We will now use the \emph{miraculous functor} 
$$\Mir_{\Bun_G}:\Dmod_{\frac{1}{2}}(\Bun_G)_{\on{co}}\to \Dmod_{\frac{1}{2}}(\Bun_G),$$
see \cite[Sect. 3.1.1]{Ga1}. 

\medskip

According to \cite[Theorem 3.1.5]{Ga1}), the functor $\Mir_{\Bun_G}$ is an equivalence. Hence, it is enough to show
that for $\CF_2$ compact and $\CF_1$ anti-tempered, 
\begin{equation} \label{e:Hom into mir comp}
\Hom_{\Dmod_{\frac{1}{2}}(\Bun_G)_{\on{co}}}(\Mir^{-1}_{\Bun_G}(\CF_1),\Mir^{-1}_{\Bun_G}(\CF_2))=0.
\end{equation}

\sssec{}

Recall that compact objects in $\Dmod_{\frac{1}{2}}(\Bun_G)$ are of the form 
$$(j_U)_!(\CF_U),$$
where 
$$U\overset{j_U}\hookrightarrow \Bun_G$$
is the embedding of a quasi-compact open, and $\CF_U\in \Dmod(U)^c$. 

\medskip

With no restriction of generality, we can assume that $U$ is \emph{co-truncative} (see \cite[Sect. 3.1]{DG1} for what this means). 
In this case, the functor
$$\Mir_{\Bun_U}:\Dmod_{\frac{1}{2}}(U)\to \Dmod_{\frac{1}{2}}(U)$$
is an equivalence (see \cite[Lemma 4.5.7]{DG1}) and we have
$$\Mir^{-1}_{\Bun_G}\circ (j_U)_!\simeq (j_U)_{*,\on{co}}\circ \Mir^{-1}_U,$$
where 
$$(j_U)_{*,\on{co}}:\Dmod(U)\to \Dmod_{\frac{1}{2}}(\Bun_G)_{\on{co}}$$
is the tautological functor, see \cite[Lemma 4.4.12]{DG1}. 

\medskip

Hence, in order the prove \eqref{e:Hom into mir comp}, it suffices to show that if $\CF\in \Dmod_{\frac{1}{2}}(\Bun_G)$
is anti-tempered, then
$$(j_U)^*_{\on{co}}\circ \Mir^{-1}_{\Bun_G}(\CF)=0,$$
where $(j_U)^*_{\on{co}}$ is the left adjoint of $(j_U)_{*,\on{co}}$. 

\sssec{}

Let 
$$\on{Id}^{\on{nv}}:\Dmod_{\frac{1}{2}}(\Bun_G)_{\on{co}}\to \Dmod_{\frac{1}{2}}(\Bun_G)$$
be the naive functor (see \cite[Sect. 2.1]{Ga1}). Recall that
$$(j_U)^*_{\on{co}}\simeq (j_U)^* \circ \on{Id}^{\on{nv}}$$
(see \cite[Corollary 2.1.5]{Ga1}). 

\medskip

Hence, it suffices to show that the composite functor 
$$\on{Id}^{\on{nv}}\circ \Mir^{-1}_{\Bun_G}$$
annihilates the anti-tempered subcategory.

\medskip

It is easy to see that both functors $\on{Id}^{\on{nv}}$ and $\Mir_{\Bun_G}$ commute with the Hecke
action. Hence, it suffices to show that for any $\CF\in \Dmod_{\frac{1}{2}}(\Bun_G)$, we have
$$(\on{Id}^{\on{nv}}\circ \Mir^{-1}_{\Bun_G}(\CF))_{\on{anti-temp}}=0.$$

\medskip

However, this follows follows from the fact the essential image of the functor $\on{Id}^{\on{nv}}$ 
is contained in $\Dmod_{\frac{1}{2}}(\Bun_G)_{\on{temp}}$, see \cite[Theorem B]{Be1}.

\qed[\propref{p:uR ff on compacts}]

\ssec{Proof of \propref{p:images}}

The idea of the proof is that compact generators of $\Dmod_{\frac{1}{2}}(\Bun_G)$ can be assembled from
tempered compact objects, and \emph{Eisenstein series}. 

\medskip

The behavior of the former is controlled by the
assumption that $\BL_{G,\on{temp}}$ is an equivalence. 
The behavior of the latter is an expression of a basic compatibility between Eisenstein series and the
Langlands functor. 

\sssec{}

Let 
$$\Dmod_{\frac{1}{2}}(\Bun_G)_{\Eis}\subset \Dmod_{\frac{1}{2}}(\Bun_G)$$
be the full subcategory, generated by the Eisenstein functors 
$$\Eis_!:\Dmod_{\frac{1}{2}}(\Bun_M)\to \Dmod_{\frac{1}{2}}(\Bun_G)$$
for \emph{proper} parabolic subgroups of $G$. 

\sssec{}

We claim:

\begin{lem} \label{l:temp and Eis gen}
The objects in
$$\left(\Dmod_{\frac{1}{2}}(\Bun_G)_{\on{temp}}\right)^c \text{ and } \left(\Dmod_{\frac{1}{2}}(\Bun_G)_{\Eis}\right)^c$$
generate $\Dmod_{\frac{1}{2}}(\Bun_G)$.
\end{lem}

\begin{proof} 

Let 
$$\Dmod_{\frac{1}{2}}(\Bun_G)_{\on{cusp}}:=\left(\Dmod_{\frac{1}{2}}(\Bun_G)_{\Eis}\right)^\perp.$$

Denote by $\be$ the embedding
$$\Dmod_{\frac{1}{2}}(\Bun_G)_{\on{cusp}}\hookrightarrow \Dmod_{\frac{1}{2}}(\Bun_G),$$
and let $\be^L$ denote its right adjoint. 

\medskip

Since the category $\Dmod_{\frac{1}{2}}(\Bun_G)_{\on{temp}}$ is compactly 
generated\footnote{On the one hand, since we have assumed that $\BL_{G,\on{coarse}}$ is an equivalence, 
the compact generation of $\Dmod_{\frac{1}{2}}(\Bun_G)_{\on{temp}}$ follows from that of $\QCoh(\LS_\cG)$.
On the other hand, this compact generation can be proved
unconditionally: in \cite{FR} it is shown that
the objects obtained by acting by $\Rep(\cG)_\Ran$ on $\on{Poinc}^{\on{Vac,glob}}_!$ generate
$\Dmod_{\frac{1}{2}}(\Bun_G)_{\on{temp}}$.}, it suffices to show that the essential image of the composite functor
$$\Dmod_{\frac{1}{2}}(\Bun_G)_{\on{temp}}\overset{\bu}\hookrightarrow 
\Dmod_{\frac{1}{2}}(\Bun_G) \overset{\be^L}\twoheadrightarrow \Dmod_{\frac{1}{2}}(\Bun_G)_{\on{cusp}}$$
generates the target category. We will show that this functor is in fact a localization (i.e., its right adjoint is fully faithful). 

\medskip

Indeed, the right adjoint to the above functor is the composition
$$\Dmod_{\frac{1}{2}}(\Bun_G)_{\on{temp}}\overset{\bu^R}\hookleftarrow  \Dmod_{\frac{1}{2}}(\Bun_G) 
\overset{\be}\twoheadleftarrow \Dmod_{\frac{1}{2}}(\Bun_G)_{\on{cusp}}.$$

Now, according to \cite[Theorem A]{Be1}, the essential image of the functor $\be$ is contained in 
$\Dmod_{\frac{1}{2}}(\Bun_G)_{\on{temp}}$, i.e., $\be$ factors as
$$\bu\circ \be_{\on{temp}}$$
for a uniquely defined (fully faithful) functor 
$$\be_{\on{temp}}: \Dmod_{\frac{1}{2}}(\Bun_G)_{\on{cusp}}\to \Dmod_{\frac{1}{2}}(\Bun_G)_{\on{temp}}.$$

Hence, the above right adjoint identifies with
$$\bu^R\circ \bu\circ \be_{\on{temp}}\simeq \be_{\on{temp}},$$
as required. 

\end{proof} 

\sssec{}

Let 
$$\IndCoh_\Nilp(\LS_\cG)_{\Eis}\subset \IndCoh_\Nilp(\LS_\cG)$$
be the full subcategory generated by the essential images of the \emph{spectral Eisenstein functors}
$$\Eis^{\on{spec}}:\QCoh(\LS_\cM)\to \IndCoh_\Nilp(\LS_\cG)$$
for proper parabolic subgroups. 

\medskip

According to \cite[Theorem 13.3.6]{AG}, the category $\IndCoh_\Nilp(\LS_\cG)$ is compactly generated by
$$\QCoh(\LS_\cG)^c \text{ and } \left(\IndCoh_\Nilp(\LS_\cG)_{\Eis}\right)^c,$$
where we view $\QCoh(\LS_\cG)$ as a subcategory of $\IndCoh_\Nilp(\LS_\cG)$ via the functor
$\Xi_{\{0\},\Nilp}$, the left adjoint of $\Psi_{\Nilp,\{0\}}$. 

\medskip

Hence, given \lemref{l:temp and Eis gen}, in order to prove \propref{p:images}, it suffices to show that:

\begin{enumerate}

\item The essential image of $\left(\Dmod_{\frac{1}{2}}(\Bun_G)_{\on{temp}}\right)^c$ in $\QCoh(\LS_\cG)$ under 
the functor $\BL_{G,\on{coarse}}$ equals $\QCoh(\LS_\cG)^c$;

\item  The essential image of $\left(\Dmod_{\frac{1}{2}}(\Bun_G)_{\Eis}\right)^c$  under 
$\BL_{G,\on{coarse}}$ equals $\left(\IndCoh_\Nilp(\LS_\cG)_{\Eis}\right)^c$.

\end{enumerate} 

\sssec{}

Point (1) is immediate from the fact that $\BL_{G,\on{temp}}$ is an equivalence. 

\medskip

Point (2) follows by induction on the semi-simple rank from the fact that $\BL_{G,\on{temp}}$ is 
fully faithful from the next assertion (which is well-known, and will be discussed in detail in a subsequent paper in the series):

\medskip

\begin{lem} \label{l:Eis}
For a given parabolic subgroup $P$ with Levi quotient $M$, the diagram
$$
\CD 
\Dmod_{\frac{1}{2}}(\Bun_M) @>{\BL_{M,\on{coarse}}}>> \QCoh(\LS_\cM) \\
@V{\Eis_!}VV @VV{\Eis^{\on{spec}}_{\on{coarse}}}V \\
\Dmod_{\frac{1}{2}}(\Bun_G) @>{\BL_{G,\on{coarse}}}>> \QCoh(\LS_\cG)
\endCD
$$
commutes up to cohomological shifts and tensoring by line bundles, where
$$\Eis^{\on{spec}}_{\on{coarse}}:= \Psi_{\Nilp,\{0\}}\circ \Eis^{\on{spec}}.$$
\end{lem}

\qed[\propref{p:images}]

\ssec{The Betti context}

The goal of this subsection is to prove the implication (ii') $\Rightarrow$ (ii) 
in \thmref{t:dR => Betti bis}. The argument will be parallel to the proof of 
the implication (i') $\Rightarrow$ (i) given above. 

\sssec{}

Define the full subcategory
$$\Shv^{\on{Betti}}_{\frac{1}{2},\Nilp}(\Bun_G)_{\on{temp}}\subset \Shv^{\on{Betti}}_{\frac{1}{2},\Nilp}(\Bun_G)$$
as in \cite[Sect. 12.8.4]{AG}. 

\medskip

Then the material in Sects. \ref{sss:LL temp}-\ref{sss:red to temp} applies as is. I.e., it suffices to show that
if the functor
$$\BL^{\on{Betti}}_{G,\on{temp}}:\Shv^{\on{Betti}}_{\frac{1}{2},\Nilp}(\Bun_G)_{\on{temp}}\to \QCoh(\LS^{\on{Betti}}_\cG)$$
is an equivalence, then so is $\BL^{\on{Betti}}_G$. 

\medskip

The latter in turn reduces to the combination of the following two propositions:

\begin{prop} \label{p:uR ff on compacts Betti}
The functor 
$$\bu^R|_{\Shv^{\on{Betti}}_{\frac{1}{2},\Nilp}(\Bun_G)^c}:\Shv^{\on{Betti}}_{\frac{1}{2},\Nilp}(\Bun_G)^c\to 
\Shv^{\on{Betti}}_{\frac{1}{2},\Nilp}(\Bun_G)_{\on{temp}}$$
is fully faithful.
\end{prop}

\begin{prop} \label{p:images Betti}
Assume that $\BL^{\on{Betti}}_{G,\on{temp}}$ is an equivalence (for the group $G$ and all its Levi subgroups). 
Then the essential image of $\Shv^{\on{Betti}}_{\frac{1}{2},\Nilp}(\Bun_G)^c$ in $\QCoh(\LS^{\on{Betti}}_\cG)$
under 
$$\BL^{\on{Betti}}_{G,\on{temp}}\circ \bu^R\simeq \BL^{\on{Betti}}_{G,\on{coarse}}$$ 
equals $\Coh_\Nilp(\LS^{\on{Betti}}_\cG)$.
\end{prop}

\sssec{}

\propref{p:uR ff on compacts Betti} is proved along the same lines as \propref{p:uR ff on compacts}
using the following ingredients:

\begin{itemize}

\item The category $\Shv^{\on{Betti}}_{\frac{1}{2},\Nilp}(\Bun_G)_{\on{co}}$ is well-defined
(see \cite[Theorem 14.1.5]{AGKRRV}), and the functor 
$$\Mir_{\Bun_G}:\Shv^{\on{Betti}}_{\frac{1}{2},\Nilp}(\Bun_G)_{\on{co}}\to
\Shv^{\on{Betti}}_{\frac{1}{2},\Nilp}(\Bun_G)$$
is an equivalence;

\smallskip

\item The essential image of the functor
$$\on{Id}^{\on{nv}}:\Shv^{\on{Betti}}_{\frac{1}{2},\Nilp}(\Bun_G)_{\on{co}}\to
\Shv^{\on{Betti}}_{\frac{1}{2},\Nilp}(\Bun_G)$$
lies in $\Shv^{\on{Betti}}_{\frac{1}{2},\Nilp}(\Bun_G)_{\on{temp}}$.

\end{itemize}

Both these assertions are proved in the same way as in the de Rham setting. 

\sssec{} 

\propref{p:images Betti} is proved along the same lines as \propref{p:images}
using the following ingredients:

\begin{itemize}

\item The subcategory
$$\Shv^{\on{Betti}}_{\frac{1}{2},\Nilp}(\Bun_G)_{\on{cusp}}\subset \Shv^{\on{Betti}}_{\frac{1}{2},\Nilp}(\Bun_G)$$
is contained in the essential image of the functor $\on{Id}^{\on{nv}}$.

\smallskip

\item The diagram 
$$
\CD 
\Shv^{\on{Betti}}_{\frac{1}{2},\Nilp}(\Bun_M) @>{\BL^{\on{Betti}}_{M,\on{coarse}}}>> \QCoh(\LS^{\on{Betti}}_\cM) \\
@V{\Eis_!}VV @VV{\Eis^{\on{spec}}_{\on{coarse}}}V \\
\Shv^{\on{Betti}}_{\frac{1}{2},\Nilp}(\Bun_G)@>{\BL^{\on{Betti}}_{G,\on{coarse}}}>> \QCoh(\LS^{\on{Betti}}_\cG)
\endCD
$$
commutes up to cohomological shifts and tensoring by line bundles.

\end{itemize}

Both these assertions are proved in the same way as in the de Rham setting. 

\section{The structure of Hecke eigensheaves}\label{s:eigensheaves}

In this section, we discuss the structure of Hecke eigensheaves for irreducible
spectral parameters, building on 
recent results from \cite{AGKRRV} and \cite{FR}.

\medskip 
  
Below, we will assume GLC in the de Rham setting, and hence in the Betti setting.

\subsection{Statement of the main result}

\subsubsection{The Hecke eigensheaf}\label{sss:the-eigensheaf}

Let $\sigma \in \LS_{\cG}$ be a $k$-point of $\LS_{\cG}$.

\medskip 

We define the category
\[
\Dmod_{\frac{1}{2}}(\Bun_G)_{\on{Hecke},\sigma} := 
\Dmod_{\frac{1}{2}}(\Bun_G) \underset{\QCoh(\LS_{\cG})}{\otimes} \Vect
\]
\noindent of \emph{Hecke eigensheaves} for $\sigma$. Here $\QCoh(\LS_{\cG})$ acts
on $\Vect$ by pullback along $\sigma:\Spec(k) \to \LS_{\cG}$.

\medskip

Using the pushforward functor 
$\sigma_*:\Vect \to \QCoh(\LS_{\cG})$, we obtain a functor
functor
\[
\oblv_{\on{Hecke},\sigma}: \Dmod_{\frac{1}{2}}(\Bun_G)_{\on{Hecke},\sigma} \to 
\Dmod_{\frac{1}{2}}(\Bun_G).
\]

\medskip 

Now assume that $\sigma$ is an \emph{irreducible} local system, i.e., $\sigma$ does
not admit a reduction to any proper parabolic $\cP \subsetneq \cG$. 
By \cite[Prop. 13.3.3]{AG} and our running assumption that GLC holds, the functor
$\BL_G$ induces an equivalence:
\[
\BL_{G,\sigma}:\Dmod_{\frac{1}{2}}(\Bun_G)_{\on{Hecke},\sigma} \simeq \Vect.
\]

Throughout this section, we let
\[
\widetilde{\CF}_{\sigma} \in \Dmod_{\frac{1}{2}}(\Bun_G)_{\on{Hecke},\sigma}
\]
\noindent denote the unique object with
\[
\BL_{G,\sigma}(\widetilde{\CF}_{\sigma}) = k[-\dim(\Bun_G)+\dim(\Bun_{N,\rho(\omega_X)})].
\]
\noindent 
The shift here is motivated by \thmref{t:L coarse left exact Nilp}.

\medskip 

We then let
\[
\CF_{\sigma} := \oblv_{\on{Hecke},\sigma}(\widetilde{\CF}_{\sigma}) \in \Dmod_{\frac{1}{2}}(\Bun_G)
\]
\noindent denote the underlying D-module of the canonical
Hecke eigensheaf.

\subsubsection{}

We will prove the following result. Most parts of it 
are not original to this work; we will discuss 
attributions after stating the result.

\begin{thm}\label{t:eigensheaves}

As above, let $\sigma$ be an irreducible local system and
assume the geometric Langlands conjecture.

\begin{enumerate}

\item\label{i:eigen-rs} $\CF_{\sigma}$ is holonomic and 
has regular singularities.

\item\label{i:eigen-nilp} 
The singular support of $\CF_{\sigma}$ lies in the
nilpotent cone $\Nilp \subset T^*(\Bun_G)$.

\item\label{i:eigen-perv} $\CF_{\sigma}$ lies in the heart 
$\Dmod_{\frac{1}{2}}(\Bun_G)^{\heartsuit}$ of 
$\Dmod_{\frac{1}{2}}(\Bun_G)$.

\item\label{i:eigen-l-packets} 

Let $S_{\sigma}$ denote the (non-derived) group of automorphisms of 
$\sigma$ as a local system. Then $\CF_{\sigma}$ has a decomposition: 
\[
\CF_{\sigma} \simeq 
\oplus_{\rho \in \on{Irrep}(S_{\sigma})}
\CF_{\sigma,\rho}^{\dim \rho}
\]
\noindent where each object
\[
\CF_{\sigma,\rho} \in \Dmod_{\frac{1}{2}}(\Bun_G)^{\heartsuit}
\]
\noindent is a \emph{simple} (holonomic) $D$-module.
Moreover, these simple objects are distinct: 
$\CF_{\sigma,\rho_1} \simeq \CF_{\sigma,\rho_2}$ if and only
if $\rho_1 = \rho_2$.

In particular,
$\CF_{\sigma} \in \Dmod_{\frac{1}{2}}(\Bun_G)^{\heartsuit}$ 
is semi-simple. 

\item\label{i:eigen-CC}

Assume the genus of $X$ is at least two and $G$ has connected center.

Let $[\Nilp]$ denote the cycle on $T^*(\Bun_G)$ defined
by the nilpotent cone. In other words, the multiplicities
of $[\Nilp]$ are the multiplicities of irreducible components
of the closed substack\footnote{We recall that $\Nilp$ is defined scheme-theoretically as the preimage of $0$ under the Hitchin map.}
$\Nilp \subset T^*(\Bun_G)$.
Then the characteristic cycle $\on{CC}(\CF_{\sigma})$ equals 
$[\Nilp]$.

\end{enumerate}

\end{thm}

In the above, \eqref{i:eigen-rs} and \eqref{i:eigen-nilp} 
are due to \cite{AGKRRV} and \eqref{i:eigen-perv} is due
to \cite{FR}. Also, \cite{FR} observed that its results imply
\eqref{i:eigen-l-packets} assuming GLC, and 
proved a special case of this result without GLC. The last result
\eqref{i:eigen-CC} is new.

\medskip 

The above results have old folklore status as conjectures. 
Most appear in \cite[Conj. 6.3.2]{L}. We expect
\eqref{i:eigen-CC} to hold without the restriction on $G$,
and our methods below prove something in this direction in general
(e.g., we prove that \eqref{i:eigen-CC} holds for a non-trivial class
of local systems $\sigma$).

\begin{rem}

We note that one can check directly that the multiplicity of the zero section in 
$[\Nilp]$ is $\prod_{d_i \geq 2} d_i^{\dim \Gamma(\Omega_X^{1,\otimes d_i})} = 
\prod d_i^{(2d_i-1)(g-1)}$ for $d_i$ running over the exponents
of $G$; it follows that this is the rank of $\CF_{\sigma}$ at a generic
point of $\Bun_G$.

\end{rem}

\begin{proof}[Proof of \thmref{t:eigensheaves}, 
\eqref{i:eigen-rs}-\eqref{i:eigen-l-packets}]

First, \eqref{i:eigen-nilp} is \cite[Cor. 14.4.10]{AGKRRV}.
As $\Nilp$ is Lagrangian, we deduce that Hecke eigensheaves 
are holonomic. They have regular singularities 
by \cite[Cor. 16.5.7]{AGKRRV}. 
Then \eqref{i:eigen-perv} is \cite[Thm. 11.2.1.2]{FR}.

\medskip 

We now turn to \eqref{i:eigen-l-packets}. Recall that 
$S_{\sigma}$ is an extension of a finite group by 
the center $Z_{\cG}$ of $\cG$; in particular, $S_{\sigma}$ is 
a (disconnected) reductive group.  

\medskip 

Let $\LS_{\cG,\sigma}^{\on{restr}} \subset \LS_{\cG,\sigma}^{\on{restr}}$ denote the connected component containing 
$\sigma$. Define:
\[
\Dmod_{\Nilp,\frac{1}{2}}(\Bun_G)_{\sigma} :=
\Dmod_{\Nilp,\frac{1}{2}}(\Bun_G) 
\underset{\QCoh(\LS_{\cG}^{\on{restr}})}
{\otimes}
\QCoh(\LS_{\cG,\sigma}^{\on{restr}}).
\]
\noindent As in \cite[Cor. 14.3.5]{AGKRRV}, this category
is a direct summand of $\Dmod_{\Nilp,\frac{1}{2}}(\Bun_G)$.

\medskip 

By \lemref{l:Eis}, every object of $\Dmod_{\Nilp,\frac{1}{2}}(\Bun_G)_{\sigma}$
is cuspidal. Therefore, by \cite{Be1}, every object of 
$\Dmod_{\Nilp,\frac{1}{2}}(\Bun_G)_{\sigma}$ is tempered. 
Therefore, $\BL_G$ induces an equivalence
\[
\BL_{G,\widehat{\sigma}}:\Dmod_{\Nilp,\frac{1}{2}}(\Bun_G)_{\sigma} \simeq 
\QCoh(\LS_{\cG,\sigma}^{\on{restr}}).
\]

We then have a diagram
\[
\begin{tikzcd} 
&
\Vect 
\arrow[dl,swap,"\CF_{\sigma}"] 
\arrow[dr,"\sigma_*"]
\\
\Shv_{\Nilp,\frac{1}{2}}(\Bun_G)_{\sigma}
\arrow[rr,"\BL_{G,\widehat{\sigma}}{[C]}
"]
&&
\QCoh(\LS_{\cG,\sigma}^{\on{restr}})
\end{tikzcd}
\]
\noindent that commutes by definition of $\CF_{\sigma}$. Here 
$C:= \dim(\Bun_G) - \dim(\Bun_{N,\rho(\omega_X)})$
and $\sigma_*$ denotes
pushforward along $\sigma:\Spec(k) \to \LS_{\cG}^{\on{restr}}$.

\medskip 

The bottom arrow in the above diagram is a t-exact equivalence by GLC and 
\thmref{t:L coarse left exact Nilp}. Therefore, Jordan-H\"older questions for 
$\CF_{\sigma}$ are the same as for 
$\sigma_*(k)$. 

\medskip 

We now remind (\cite[Prop. 4.3.5]{AGKRRV}) that $\sigma$ factors
as $\Spec(k) \to \Spec(k)/S_{\sigma} \to \LS_{\cG,\sigma}^{\on{restr}}$ with the
latter map being a closed embedding. Therefore, 
there is a fully faithful embedding 
$\Rep(S_{\sigma})^{\heartsuit} \subset \QCoh(\LS_{\cG,\sigma}^{\on{restr}})^{\heartsuit}$ 
with $\sigma_*(k)$ corresponding to the regular representation of $S_{\sigma}$.

\medskip 

Now the result follows from the fact that $S_{\sigma}$ is reductive,
so its (left) regular representation has the form
$R_{S_{\sigma}} \simeq \oplus_{\rho \in \on{Irrep}(S_{\sigma})} \rho^{\dim \rho}$.

\end{proof}

\subsection{Characteristic cycles for eigensheaves}

We now turn to the proof of \thmref{t:eigensheaves} \eqref{i:eigen-CC}.

\subsubsection{Idea of the proof}

Our main new result is the following. We always assume genus $\geq 2$ 
in what follows, but $G$ is allowed to be arbitrary.

\begin{thm}\label{t:cc-cts}

For every pair of irreducible local systems 
$\sigma_1,\sigma_2 \in \LS_{\cG}$ lying in the same irreducible component
of $\LS_{\cG}$, $\on{CC}(\CF_{\sigma_1}) = 
\on{CC}(\CF_{\sigma_2})$.

\end{thm}

The following is \cite[Prop. 5.1.2]{BD}:

\begin{thm}[Beilinson-Drinfeld]\label{t:cc-existence}

Let $\LS_{\cG,\on{neut}} \subset \LS_{\cG}$ be the irreducible
component of $\LS_{\cG}$ containing the trivial local system. 
Then there exists an irreducible local system $\sigma \in \LS_{\cG,\on{neut}}$ 
such that $\on{CC}(\CF_{\sigma}) = [\Nilp]$.

\end{thm}

We now recall that $\LS_{\cG}$ is irreducible when $\cG$ has simply-connected
derived group (and the genus is $\geq 2$), see \cite[Prop. 2.11.4]{BD}. 
Therefore, the above two results
imply the claim.

\begin{rem}

Theorem \ref{t:cc-existence} is proved by de Rham methods, but 
Theorem \ref{t:cc-cts} will be proved by Betti methods.

\end{rem}

\subsection{Proof of Theorem \ref{t:cc-cts}}

We will prove this result using some theory of Betti constructible sheaves.

\medskip 

By the Lefschetz principle, note that Theorem \ref{t:cc-cts} reduces to 
its Betti version. Therefore, we put ourselves in this setting in what follows.

\subsubsection{Review of microstalks}

Let $k = \BC$, let $\CY$ be a smooth algebraic stack of finite type over $\BC$, 
and recall that $\Shv^{\on{Betti}}(\CY)$ is the category of all Betti sheaves on
$\CY$. Let $\Lambda \subset T^*(\CY)$ be a conical (algebraic) Lagrangian 
and consider
$\Shv_{\Lambda}^{\on{Betti}}(\CY) \subset \Shv^{\on{Betti}}(\CY)$,
which we recall is the subcategory of sheaves with singular support in 
$\Lambda$ and is defined in \cite[Sect. F.6]{AGKRRV} for stacks.

\medskip

Let $\xi \in \Lambda$ be a point lying in the smooth locus of $\Lambda$. 
The following result summarizes the microstalk theory for our needs.

\begin{thm}\label{t:microstalk}

There is a functor $\mu_{\xi}:\Shv_{\Lambda}(\CY) \to \Vect$ with the
following properties:

\begin{enumerate}

\item\label{i:micro-exact} $\mu_{\xi}$ is t-exact.

\item\label{i:micro-corep} $\mu_{\xi}$ admits a left adjoint, i.e., 
$\mu_{\xi}$ is corepresented by an object $\mu\delta_{\xi} \in \Shv_{\Lambda}(\CY)$.

\item\label{i:micro-cc} For a constructible sheaf $\CF \in \Shv_{\Lambda}^{\on{Betti}}(\CY)$,
$\dim(\mu_{\xi}(\CF))$ equals the order of the characteristic cycle of 
$\CF$ at $\xi$.

\end{enumerate}

\end{thm}

\begin{proof}

First, assume $\CY$ is a smooth scheme, see 
e.g. \cite[Sect. 4.4]{GPS} for a review of the construction and
the corepresentability. Exactness (up to a normalizing shift) follows
from \cite[Cor. 10.3.13]{KS}. The relation to the characteristic cycle
is \cite[Examples 9.5.7]{KS}.

\medskip

In general, choose $\pi:Y \to \CY$ a smooth
cover by a smooth scheme, and we assume Let $y \in \CY$ be the point so $\xi \in T_y^*(\CY)$; let
$\widetilde{y} \in \pi^{-1}(y)$ be a lift of $y$ and let
$\widetilde{\xi} := \pi^*(\xi) \in T_{\widetilde{y}}^*(Y)$. 
Define: 
\[
\mu_{\xi}(\CF) := \mu_{\widetilde{\xi}}\circ \pi^!(\CF)[-\dim Y + \dim \CY] = 
\mu_{\widetilde{\xi}}\circ \pi^*(\CF)[\dim Y - \dim \CY].
\]

Finally, recall from \cite[Cor. G.7.6]{AGKRRV} that the embedding
$\iota:\Shv_{\Lambda}^{\on{Betti}}(\CY) \to \Shv^{\on{Betti}}(\CY)$
admits a left adjoint $\iota^L$. Then:
\[
\mu\delta_{\xi} := \iota^L(\pi_!(\mu\delta_{\widetilde{\xi}}))[\dim Y -\dim \CY].
\]
\noindent evidently corepresents the microstalk at $\xi$.

\end{proof}

\subsubsection{Microstalks and the Langlands equivalence}

Let $\Nilp^{\alpha} \subset \Nilp$ be an irreducible component and
let $\xi \in \Nilp^{\alpha}$ be a generic point lying in the smooth locus. 
We have a microstalk functor:
\[
\Shv_{\Nilp,\frac{1}{2}}^{\on{Betti}}(\Bun_G) \to \Vect
\]
\noindent and a corepresenting object $\mu\delta_{\xi}$.

\medskip 

Define $\CE_{\xi} \in \IndCoh_{\Nilp}(\LS_{\cG}^{\on{Betti}})$ as: 
\[
\CE_{\xi} := \BL_G^{\on{Betti}}(\mu\delta_{\xi})
[-\dim(\Bun_G) + \dim(\Bun_{N,\rho(\omega_X)})].
\]

We let $\CE_{\xi}^{\on{irred}} \in \QCoh(\LS_{\cG}^{\on{Betti},\on{irred}})$
be the restriction of $\CE_{\xi}$ to the irreducible locus.

\subsubsection{}

Observe that $\CE_{\xi}$ is compact, i.e., it lies in 
$\Coh_{\Nilp}(\LS_{\cG}^{\on{Betti}})$: this follows as it
corresponds to the compact object $\mu\delta_{\xi}$ under the Betti Langlands
equivalence.

\medskip 

In particular, $\CE_{\xi}^{\on{irred}}$ is perfect, so the Euler characteristic of
its fibers form a locally constant function on $\LS_\cG^{\on{Betti},\on{irred}}$.

\medskip  

We compute the dual: 
\[
\Hom_{\Vect}(\sigma^*(\CE_{\xi}),\BC)
\]
\noindent to this fiber at $\sigma \in \LS_{\cG}^{\on{Betti},\on{irred}}(\BC)$ as:
\[
\Hom_{\QCoh(\LS_{\cG}^{\on{Betti}}}(\CE_{\xi},\sigma_*(\BC))
\simeq 
\Hom_{\Shv_{\Nilp,\frac{1}{2}}^{\on{Betti}}}
(\mu\delta_{\xi},
\CF_{\sigma}) = \mu_{\xi}(\CF_{\sigma})
\]

\noindent for $\CF_{\sigma}$ the (Betti) eigensheaf corresponding to $\sigma$,
including the shift as in \secref{sss:the-eigensheaf}.

Applying \thmref{t:microstalk}, we see 
(\emph{i}) $\sigma^*(\CE_{\xi})$ is actually concentrated in 
one degree, so $\CE_{\xi}^{\on{irred}}$ is actually a vector bundle,
and (\emph{ii}) the fiber of $\CE_{\xi}^{\on{irred}}$ at 
$\sigma$ has the same dimension as the order of the characteristic cycle
of $\CF_{\sigma}$ at $\xi$. 

\medskip 

Therefore, the function $\sigma \mapsto \on{ord}_{\xi}(\on{CC}(\CF_{\sigma}))$
is a locally constant function on $\LS_{\cG}^{\on{Betti},\on{irred}}$.
It now suffices to observe that 
$\LS_{\cG}^{\on{Betti},\on{irred}}$ is smooth and dense in $\LS_{\cG}^{\on{Betti}}$, 
so connected components of $\LS_{\cG}^{\on{Betti},\on{irred}}$ are in bijection with
irreducible components of $\LS_{\cG}^{\on{Betti}}$.

%
%
%
%
%
%
%
%
%
%
%
%
%
%


\begin{thebibliography}{99}


\bibitem[AG]{AG} D.~Arinkin and D.~Gaitsgory, {\it Singular support of coherent sheaves, and the Geometric Langlands Conjecture}, 
Selecta Math. N.S. {\bf 21} (2015), 1--199.

\bibitem[AGKRRV]{AGKRRV} D.~Arinkin, D.Gaitsgory, D.~Kazhdan, S.~Raskin, N.~Rozenblyum and Y.~Varshavsky, \newline
{\em The stack of local systems with restricted variation and geometric Langlands theory with nilpotent singular support},
arXiv:2010.01906. 

\bibitem[BD]{BD} A.~Beilinson and V.~Drinfeld, {\it Quantization of Hitchin's integrable system and Hecke eigensheaves}, 
available at http://people.math.harvard.edu/$\sim$gaitsgde/grad$\underline{\text{\,\,\,}}$2009/



\bibitem[Be1]{Be1} D.~Beraldo, {\it On the geometric Ramanujan conjecture}, arXiv:2103.17211.

\bibitem[Be2]{Be2} D.~Beraldo, {\it The spectral gluing theorem revisited}, arXiv:1804.04861.

\bibitem[BeLi]{BeLi} D.~Beraldo and L.~Chen, {\it Automorphic Gluing}, 
arXiv:2204.09141.

\bibitem[BZN]{BZN} D.~Ben-Zvi and D.~Nadler, {\it Betti geometric Langlands}, arXiv:1606.08523.









\bibitem[DG1]{DG1} V.~Drinfeld and D.~Gaitsgory, {\it Compact generation of the category of D-modules on the stack of G-bundles on a curve}
Cambridge Math Journal, {\bf 3} (2015), 19--125. 

\bibitem[DG2]{DG2} V.~Drinfeld and D.~Gaitsgory, {\it On some finiteness questions for algebraic stacks},
GAFA {\bf 23} (2013), 149--294.

\bibitem[FR]{FR} J.~Faergeman and S.~Raskin, {\it Non-vanishing of geometric Whittaker coefficients for reductive groups}, \newline
arXiv:2207.02955.











%


\bibitem[Ga1]{Ga1} D.~Gaitsgory, {\it A "strange" functional equation for Eisenstein series and Verdier
duality on the moduli stack of bundles}, Annales Scientifiques de l'ENS {\bf 50} (2017), 1123--1162. 

\bibitem[Ga2]{Ga2} D. Gaitsgory, {\it Outline of the proof of the Geometric Langlands Conjecture for GL(2)}, Ast\'erisque {\bf 370} (2015),
1--112.

%
\bibitem[GLys]{GLys} D.~Gaitsgory and S.~Lysenko, {\it Parameters and duality for the metaplectic geometric Langlands theory}, 
Selecta Math. New Ser. {\bf 24} (2018), 227--301. Also arXiv: 1608.00284

\bibitem[GaRo1]{GaRo1}  D.~Gaitsgory and N.~Rozenblyum, {\it DG ind-schemes}, Contemporary Mathematics {\bf 610} (2014), 139--251. 

\bibitem[GaRo2]{GaRo2}  D.~Gaitsgory and N.~Rozenblyum, {\it Crystals and D-modules}, PAMQ {\bf 10}, no. 1 (2014), 57--155.

\bibitem[GPS]{GPS} S.~Ganatra, J.~Pardon, and V.~Shende,
{\it Microlocal Morse theory of wrapped Fukaya categories},
Ann. of Math. (2) {\bf 199} (2024), no.3, 943–1042.

\bibitem[KS]{KS} M.~Kashiwara and P.~Schapira, {\it Sheaves and Manifolds},
Grundlehren Math. Wiss., {\bf 292},
Springer-Verlag, Berlin, 1994. 

\bibitem[Laum]{L} G.~Laumon, {\it Correspondance de Langlands géométrique pour les corps de fonctions}, Duke Math. J. {\bf 54}, no. 1 (1987): 309-359.

\bibitem[Lin]{Lin} K.~Lin, {\it Poincar\'e series and miraculous duality}, arXiv:2211.05282.

\bibitem[NY]{NY} D.~Nadler and Z.~Yun, {\it Spectral action in Betti geometric Langlands, Israel Journal of Mathematics}, {\bf 232} (2019), 
299--349.



 
 
%
%
%

\bibitem[Ra]{Ra} S.~Raskin, {\it A generalization of the b-function lemma}, Comp. Math. {\bf 157} (2021), 2199--2214. 

%
%
%





\end{thebibliography}
\end{document}